\documentclass[reqno,12pt,a4paper]{amsbook}
\usepackage[utf8]{inputenc}
\usepackage{amsmath,amssymb,amsthm,graphicx,mathrsfs,url}
\usepackage[usenames,dvipsnames]{color}
\usepackage[colorlinks=true,linkcolor=red,citecolor=green]{hyperref}
\usepackage{enumitem}
\usepackage[french]{babel}
\usepackage{csquotes}
\usepackage[all]{xy}
\usepackage{amsfonts}
\usepackage{dsfont}
\usepackage{xcolor}
\usepackage{array}
\usepackage{tikz}

\usepackage[margin=1.2in]{geometry}
\usepackage[nothm]{thmbox}
\usepackage{mdframed}

\newcommand{\R}{\mathbb{R}}

\newcommand{\Z}{\mathbb{Z}}

\newcommand{\dd}{\mathrm{d}}

\newcommand{\inte}[1]{\underset{#1}{\int}}

\newcommand{\supp}[1]{\mathrm{supp}(#1)}

\newcommand{\supr}[1]{\underset{#1}{\sup}}
\newcommand{\somme}[1]{\underset{#1}{\sum}}

\newcommand{\tends}[1]{\underset{#1}{\longrightarrow}}
\newcommand{\equivaut}[1]{\underset{#1}{\sim}}
\newcommand{\quotient}[2]{{\raisebox{.2em}{$#1$}\left/\raisebox{-.2em}{$#2$}\right.}}
\newcommand{\quotientg}[2]{{ \raisebox{-.2em}{$#1$} \Big{\backslash}  \raisebox{.2em}{$#2$} }}

\newcommand{\fonction}[5]{\begin{array}{r r c l}
					#1 \hspace{2mm} : & #2 & \to & #3 \\
					& #4 & \mapsto & #5 \\
			   \end{array}}
\newcommand{\fonctionbis}[4]{\begin{array}{r c l}
					   #1 & \to & #2 \\
					 #3 & \mapsto & #4 \\
			   \end{array}}	  

\newcommand{\Sp}{ \mathrm{Sp}_4(\mathbb{R})}

\newcommand{\TT}{\mathbb{T}}

\newcommand{\HH}{\mathbb{H}}
\newcommand{\ZZ}{\mathbb{Z}}
\newcommand{\RR}{\mathbb{R}}
\newcommand{\CC}{\mathbb{C}}
\newcommand{\NN}{\mathbb{N}}
\newcommand{\scal}[2]{ \left< {#1} \cdot {#2} \right>} 

\let\Re=\Real

\DeclareMathOperator{\SL}{SL}

\DeclareMathOperator{\cF}{\mathcal{F}\!} % Transformée de Fourier 

\numberwithin{equation}{chapter}

\theoremstyle{definition}

\newtheorem{definition}[equation]{Définition}
\newtheorem{remark}[equation]{Remarque}

\theoremstyle{theorem}

\newtheorem{theoreme}[equation]{Théorème}
\newtheorem{proposition}[equation]{Proposition}

\newtheorem{conjecture}[equation]{Conjecture}
\newtheorem{lemma}[equation]{Lemme}

\begin{document}

\title{Interférences pour les chats quantiques}
\author{Jean-Michel Pipeau}
\address{Institut International de Pipeau Mathématique, Auffargis, France}

\date{ \today}

\begin{abstract}
Ce texte est mon mémoire de M2, il traite de la dynamique quantique des quantifiés d'automorphismes linéaire du tore, après le temps d'Ehrenfest.
Je montre que, dans la base des paquets d'ondes, la ``matrice'' du propagateur associé est bien approchée par des sommes de Birkhoff de nilrotations sur le tore. Dans un second temps, j'étudie en profondeur ces sommes et je relie l'équidistribution des évolués de paquets d’ondes avec un problème d’approximation diophantienne.
\end{abstract}

\maketitle

\tableofcontents

\newpage

\chapter{Introduction}

Ce court chapitre est une introduction au reste du texte. Le format est le suivant : après avoir rappelé le contexte général dans lequel se placent les mathématiques étudiées ici, on retrace le cheminement qui nous a conduit aux résultats démontrés dans la suite du texte.

\section{Chaos quantique}

On s'intéresse au problème suivant (décrit dans l'encadré ci-dessous), lui-même motivé par l'étude de la mécanique ondulatoire. On ne cherchera pas à le motiver (le lecteur intéressé pourra consulter \cite{Haake} ou \cite{JMP}). 

\vspace{2mm}

\paragraph*{\it \'Equation de Schrödinger} Dans ce paragraphe, $u(t,x)$ dénote une fonction $\mathbb{R} \times M  \longrightarrow \mathbb{C}$, où $t \in \mathbb{R}$ est une variable de temps et $(M,g)$ est une variété riemannienne.

\medskip

\fbox{
\parbox{\textwidth}{

\textbf{Problème.} \\Pour de petites valeurs du paramètre $h > 0$, \textbf{décrire en temps long} les solutions de l'équation aux dérivées partielles 

$$ ih\frac{\partial}{\partial t} u(t,x) = h^2 \Delta_g u(t,x) + V(x)\cdot u(t,x) $$ qui est communément appelée \textit{équation de Schrödinger}.

}}

\vspace{2mm}

On sait depuis le début de la mécanique quantique que pour des temps suffisamment courts, l'évolution dans le temps d'une fonction d'onde solution de l'équation de Schrödinger peut-être approximée de manière très satisfaisante grâce à la dynamique \textit{classique} définie par le système hamiltonien

\begin{equation*}
  \begin{cases}
    \dot{x} &  =  \frac{\partial H}{\partial v}(x,v), \\
    \dot{v} & =   -\frac{\partial H}{\partial x}(x,v),
  \end{cases}
\end{equation*}
pour le hamiltonien défini par $H(x,v) = \frac{1}{2} v^2 + V(x)$. Cette dynamique classique est une équation différentielle ordinaire définie sur $TM$ le fibré tangent de la variété $M$. 

\vspace{2mm} Cette approximation peut s'exprimer concrètement par la propagation des \textit{paquets d'ondes}, qui suivent pour des temps petits la trajectoire classique des points qu'ils représentent. On trouve sur l'Internet un certain nombre de simulations numériques mettant ce phénomène en évidence :

\begin{itemize}

\item \href{https://www.youtube.com/watch?v=e21VNqx_nWkab_channel=YourPhysicsSimulator-YTChannel}{pour un billard planaire cliquer ici.}

\item \href{https://www-fourier.ujf-grenoble.fr/~faure/enseignement/meca_q/animations/node40.html}{Pour le quantifié d'applications Anosov du tore cliquer ici.}

\end{itemize}

L'énoncé formel qui rend cette approximation rigoureuse est connu sous le nom de \textit{théorème d'Egorov} dans la littérature.

\section{Propagation et états propres}

\subsection{Temps d'Ehrenfest et interférences} Si variété ambiante $M$ est compacte (avec ou sans bord), on verra qu'après un certain temps se produisent des interférences, phénomène caractéristique des problèmes ondulatoires. C'est la limite de validité du théorème d'Egorov énoncé sous sa forme classique et de l'approximation mentionnée ci-dessus. Ces effets d'interférence se produisent à un temps de l'ordre de grandeur de 
$$
\frac{|\log h|}{\chi}
$$
où $\chi \geq 0$ est l'\textit{exposant de Liapounoff} de la dynamique hamiltonienne associée.\footnote{Dans le cas où $\chi = 0$, les effets d'interférences tendent à se produire à des échelles de temps polynomiales en $\frac{1}{h}$.} Ce temps, qu'on note $t_E(h)$, est souvent appelé \textit{temps d'Ehrenfest}.

\vspace{2mm}

\paragraph*{\it Description des états propres du laplacien et propagation} Insistons sur une de nos motivations à comprendre qualitativement les solutions de l'équation de Schrödinger au-delà du temps d'Ehrenfest:

\vspace{2mm} Si $(M,g)$ est une variété riemannienne compacte, son laplacien $\Delta_g$ est un opérateur 
%d'intérêt
très étudié. Le théorème spectral nous assure qu'il existe une base hilbertienne $f_0, f_1, \cdots, f_k, \cdots$ de $\mathrm{L}^2(M,\mathbb{C})$  associée à une suite de valeurs propres $\lambda_0 = 0 < \lambda_1 \leq \cdots \leq \lambda_k \leq \cdots $ qui tend vers l'infini. Il existe des gens qui trouvent les deux questions suivantes intéressantes :

\begin{enumerate}

\item à quoi ressemble une fonction $f_k$ typique pour $k$ très grand ? En particulier peut-elle se concentrer sur des sous-ensembles stricts de $M$?

\item à quoi ressemble, d'un point de vue statistique, la suite $(\lambda_k)_{k \in \mathbb{N}}$?

\end{enumerate}

\noindent Il est conjecturé que la réponse à ces deux questions dépend essentiellement des propriétés qualitatives de la dynamique du flot géodésique de la variété $M$. On pourra consulter l'ouvrage \cite{Haake} et le survol \cite{Nonnenmacher} pour plus de détails sur ces questions.

On mentionne ces questions car on sait que pour leur apporter des réponses satisfaisantes, il suffirait (presque) d'avoir une compréhension suffisante en temps long (comprendre "bien au-delà du temps d'Ehrenfest") des solutions à l'équation de Schrödinger dans le cas particulier où le potentiel $V$ est identiquement nul. La question générale qui motive notre démarche est la suivante.

\begin{center}

\fbox{
\parbox{\textwidth}{
Peut-on décrire et contrôler les interférences qui se produisent quand on propage une fonction d'onde, afin de pouvoir décrire (partiellement ou complètement) le propagateur de l'équation de Schrödinger pour des temps supérieurs au temps d'Ehrenfest ?
}}

\end{center}

Dans le reste de cette introduction, on essaye d'expliquer comment, en essayant de répondre à cette question assez générale, on est venu à s'intéresser au cas particulier des chats quantiques.

\vspace{3mm}

Dans le cas des variétés riemanniennes que l'on considère jusqu'à présent, on prédit que lorsque le flot géodésique associé est \textit{fortement chaotique}, 

\begin{enumerate}

\item les fonctions propres du laplacien pour des grandes valeurs propres sont toujours équidistribuées (c'est la propriété connue sous le nom d'\textit{unique ergodicité quantique});

\item la distribution statistique des valeurs propres du laplacien ressemble à celle de grandes matrices aléatoires.

\end{enumerate}

Un principe vague (qui peut être rendu rigoureux) est que ce qu'on peut dire de précis pour ces deux questions dépend des échelles de temps auxquelles on peut décrire le propagateur de l'équation de Schrödinger. Si $h > 0$ petit est fixé et $\lambda$ est l'exposant de Liapounoff de la dynamique classique, on distingue les échelles de temps remarquables suivantes (on a avidement pompé Frédéric Faure \cite[paragraphe 4.4.1]{HDRFaure})

\begin{itemize}

\item  $t < 1 $ l'échelle \textit{microlocalisée} : elle correspond au temps où la propagation d'un paquet d'onde suit exactement celle du flot géodésique dans l'espace des phases, \textit{sans délocalisation aucune} (un paquet d'ondes occupe une boule de diamètre de l'ordre $\sqrt{h}$ dans l'espace des phase au temps $0$ et au temps $1$).

\item $t \simeq \frac{|\log(h)|}{2 \log(\lambda)}$ l'échelle \textit{macroscopique} : un paquet d'ondes s'est transformée en une fonction d'onde occupant une boule de diamètre d'ordre $1$ dans l'espace des phases, toujours en suivant la dynamique classique.

\item $t \simeq \frac{|\log(h)|}{\log(\lambda)}$ l'échelle d'\textit{équidistribution} : l'paquet d'ondes s'étale le long de la variété instable, cette dernière remplissant l'espace des phases de manière uniforme l'image de notre paquet d'ondes est maintenant complètement délocalisée.

\item $t >> \frac{|\log(h)|}{\log(\lambda)}$ l'échelle des \textit{interférences} : les différentes branches de la variété instable le long de laquelle l'paquet d'ondes s'est étalé sont à distance $<< \sqrt{h}$ et interfèrent les unes avec les autres; on ne peut a priori plus dire grand chose.

\end{itemize}

Le fait d'être bloqué à la dernière étape est une difficulté conceptuelle qu'il conviendrait de lever afin de pouvoir progresser dans l'étude du chaos quantique. 

\vspace{2mm} La littérature physique et certaines simulations numériques suggèrent que pour des temps bien plus grand que $\frac{|\log(h)|}{\log(\lambda)}$, malgré les interférences, les propagés d'états cohérents restent équidistribués, comme c'est le cas pour la dynamique classique. 

\vspace{2mm} La principale difficulté est de comprendre le mécanisme (si mécanisme il existe) structurant les termes d'interférences. On a lu l'heuristique suivante dans l'HDR de Frédéric Faure : pour des temps bien supérieurs au temps d'Ehrenfest un paquet d'ondes s'est étalé le long d'une variété instable très dense. Une boule de rayon $\sqrt{h}$ est traversée par un grand nombre de branche de la variété instable, chacune d'entre elle contribuant, près de cette boule, d'une fonction d'onde d'une intensité égale, mais dont la phase peut varier. Pour comprendre, les interférences, il convient alors de comprendre la somme de ces termes de phase. Un calcul heuristique montre que si les phases se comportent comme des nombres tirés au hasard (cette somme pouvant être contrôlée grâce au théorème de la limite centrale), la contribution dans chaque boule est à peu près égale et le propagé de l'état cohérent est équidistribué. 

\vspace{2mm} On devrait maintenant être convaincu que l'enjeu est de décrire les termes intervenant dans la somme de termes de phase susmentionnée. Les deux difficultés principales auxquelles on se heurte en approchant cette question sont les suivantes :

\begin{enumerate}
\item les nombres de termes augmente exponentiellement vite avec le temps;
\item pis, on ne sait pas vraiment qui sont ces nombres (c'est là qu'un physicien fait l'hypothèse pratique que ces nombres sont aléatoires et passe tranquillement à la suite !).
\end{enumerate}

Voyons comment un résultat mathématique rigoureux nous a indiqué où regarder pour mieux comprendre la deuxième difficulté. 

\section{Chats quantiques et interférences constructives}

Il existe une manière (sur laquelle on revient en détail dans le deuxième chapitre) d'associer, à un automorphisme linéaire du tore $\mathbb{T}^2$, une dynamique de propagation quantique qui est de bien des manières analogues à la correspondance entre le flot géodésique et la propagation de l'équation de Schrödinger. 

\vspace{3mm} 
Le résultat suivant nous a convaincu de regarder ce système de plus près  %(expliqué de manière informelle
(voir la figure \ref{fig:fredimages}).

\vspace{3mm}

\fbox{
\parbox{\textwidth}{
\begin{theoreme}[Faure-Nonnemacher-deBièvre]
Si $U^h(t)$ est le propagateur quantique associé à un automorphisme linéaire du tore $T$, il existe des valeurs de $h >0$ telles que le propagé d'un paquet d'ondes se reconstruit complètement après un temps de propagation $t = 2\frac{|\log(h)|}{\log(\lambda)}$. 
\end{theoreme}
}}

\begin{figure}[!h]
	\begin{center}
		\includegraphics[scale=0.5]{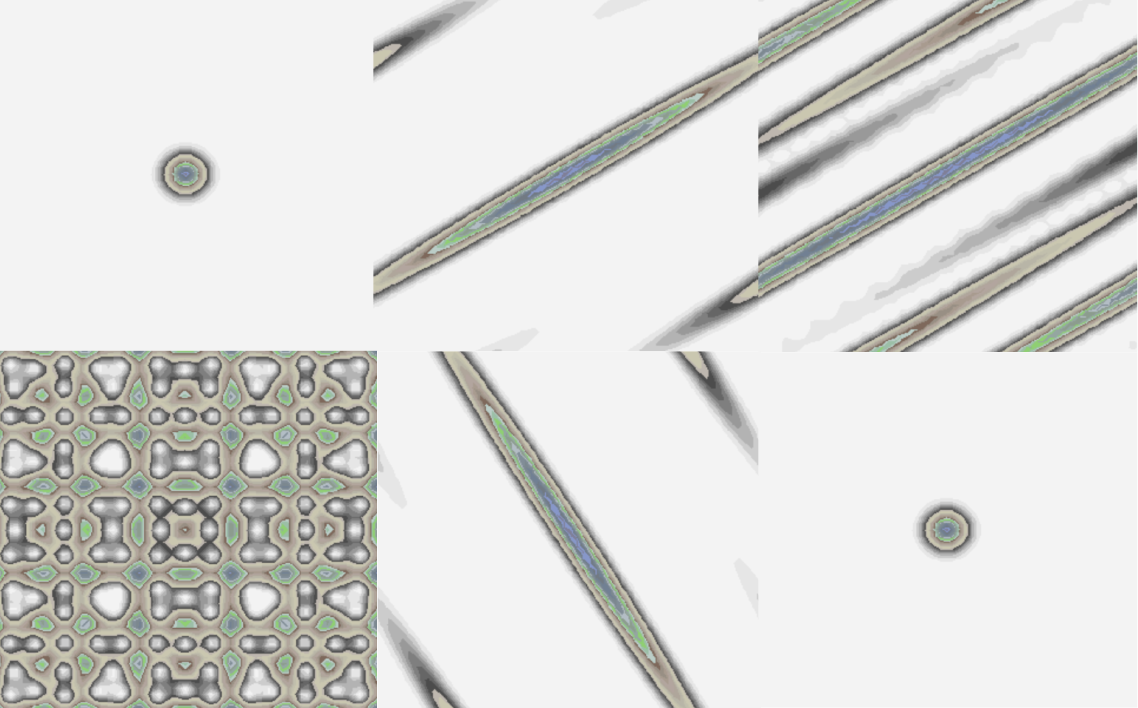}
	\caption{Le temps (nombre d'itérations du chat quantique) s'écoule d'abord de gauche à droite et ensuite de bas en haut. La première image, au temps $t=0$, correspond à la mesure associée au paquet d'ondes centrée en $0$. Le comportement reste classique sur les deux premières images, à la différence notable du fait que les variétés stables ne sont pas contractées (c'est le principe d'incertitude). La 4$^{\text{ème}}$ image est probablement la plus intéressante: on a dépassé le temps d'interférences et la fonction d'onde interfère non trivialement. On voit nettement des zones plus ou moins chargées, sans lien apparent avec la direction instable. Ce jeu d'interférences est non trivial; en itérant encore un peu l'application du chat quantique on parvient à reconstruire le paquet d'ondes original, sur la sixième image. (Source : Frédéric Faure}
	\end{center}
	\label{fig:fredimages}
	\end{figure}

\vspace{3mm} Ce temps $t = 2\frac{|\log(h)|}{\log(\lambda)}$ est le premier où des interférences constructives \textit{pourraient} se produire. Cela implique très concrètement, suivant l'heuristique expliquée dans le paragraphe précédent, que les sommes de termes d'interférences (permettant de calculer les différentes contribution du propagé de d'un paquet d'ondes dans l'espace des phases) sont \textit{très loin} de se comporter comme des sommes de nombres tirés "au hasard".

\vspace{3mm} Nous nous somme dit: il se passe forcément quelque chose de spécial dans ce cas, et le fait que ça se passe à un temps de propagation si petit nous donne l'espoir que le nombre de termes de la somme ne soit pas trop grand. \\

La dernière section de cette introduction est dédiée à conter la genèse de cet article. Nous pensons qu'il est intéressant, de manière général, de donner une idée de comment le produit fini s'est formé. Le lecteur pressé pourra directement s'orienter vers le chapitre suivant. \\

\fbox{
\parbox{\textwidth}{
\textbf{Attention :} On s'adresse d'ici à la fin de l'introduction à un.e lecteur.ice ayant une connaissance du modèle des applications du chat quantique et des notions micro-locales de base. Toutes les notions esquissées ici seront redéfinies rigoureusement dans le chapitre suivant.
}}

\section{La genèse de l'article.}
\subsection{Une bonne remarque} 
On reproduit ici un raisonnement trouvé dans l'HDR de Frédéric Faure (\cite[paragraphe 2.5,p18]{HDRFaure}. On note $\widehat{M}^h$ le quantifié d'un automorphisme linéaire du tore de dimension $2$. On considère $\psi$ un paquet d'ondes près du point $(0,0) \in \mathbb{T}^2$ qui est fixé par la dynamique classique. On sait les choses suivantes. 

\begin{itemize}
\item Pour un certain $h$, on a un paquet d'ondes $\psi$ tel que $(\widehat{M}^h)^{2t_E} \psi = \psi$ où $t_E$ est le temps d'Ehrenfest.

\item $(\widehat{M}^h)^{t_E} \psi$ est microlocalisé le long de la variété \textit{instable} passant par $(0,0)$.

\item $(\widehat{M}^h)^{-t_E} \psi$ est microlocalisé le long de la variété \textit{stable} passant par $(0,0)$.

\end{itemize}

Comme $(\widehat{M}^h)^{2t_E} \psi = \psi$, on a $(\widehat{M}^h)^{t_E} \psi = (\widehat{M}^h)^{-t_E} \psi$. Ainsi, $(\widehat{M}^h)^{t_E} \psi$ doit être localisé à l'intersection des variétés stables et instables passant par $(0,0)$. Ce sont des points bien connus en dynamiques, ils forment un réseau de points, souvent appelé \textit{intersections homoclines}. 

\begin{figure}[!h]
	\begin{center}
		\def\svgwidth{0.5 \columnwidth}
			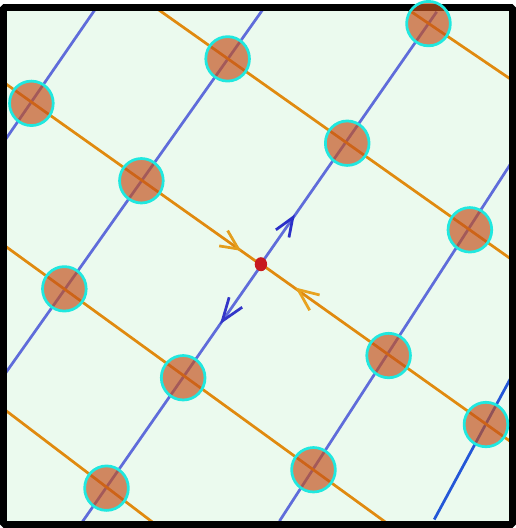
	\end{center}
	\caption{Intersections homoclines (des variétés stable et instable du point $(0,0)$)}
			\label{fig:intersection}
\end{figure}

Un des raisons pour lesquelles cette remarque nous a semblé être très importante est la suivante.

\vspace{3mm}

\fbox{
\parbox{\textwidth}{
Le temps d'Ehrenfest est le temps où se produisent les toutes premières interférences avec un petit nombre de branches superposées; si un tel phénomène d'interférences constructives près des intersections se produit, il doit être explicable rigoureusement.
}}

\medskip

On s'était donc fixé l'objectif suivant. 

\begin{center}
\it Trouver une démonstration "interférencielle" du théorème de reconstruction des paquets d'ondes.
\end{center}

Notre point vue initial était que, s'il y avait un mécanisme général structurant les termes interférences, ce serait dans ce cas particulier qu'on aurait le plus de chances de l'attraper, eu égard au fait que des interférences très constructives se produisent en des temps très courts (deux fois le temps d'Ehrenfest), avec des endroits précis où regarder (les intersections homoclines).

\subsection{Une tentative d'explication heuristique} 
On se concentre donc sur une tentative d'explication interférentielle (et non algébrique comme c'est le cas dans \cite{BNF}) du phénomène de reconstruction. Admettons qu'on a propagé un paquet d'ondes en $(0,0)$ jusqu'au temps d'Ehrenfest $t_E = \frac{|\log(h)|}{\log(\lambda)}$. Ce propagé est une superposition de fonctions d'ondes localisées le long de droites orientées dans la même direction $\theta$ (car le feuilletage instable d'un automorphisme linéaire du tore). Ces droites sont simplement les différentes branches de la feuille instable issue de $(0,0)$

\vspace{2mm}

\paragraph{\bf Fonction lagrangiennes et droites de $\mathbb{R}^2$} On rappelle que la quantification du tore $\mathbb{T}^2$ se fait (à des détails techniques près que l'on verra plus tard) en projettant celle de $\mathbb{R}^2$ pensé comme le fibré cotangent $T^*\mathbb{R}$ de la droite réelle. On fait la remarque suivante, sur laquelle on va baser le reste de notre raisonnement.

\vspace{3mm}
\fbox{
\parbox{\textwidth}{
La fonction $x \longmapsto e^{\frac{i}{h} (\frac{\tan \theta}{2} x^2 + \epsilon x)}$ microlocalise le long de la droite $\{ y = \tan \theta x + \epsilon \}$.
}}

\vspace{3mm}

Si on fait le calcul, au temps $t_E = \frac{|\log(h)|}{\log(\lambda)}$ le propagé d'un paquet d'ondes est étalé le long d'une droite instable de longueur $h^{-\frac{1}{2}}$, ses branches dans le tore sont régulièrement espacées; la distance entre deux branches consécutives est $\sqrt{h}$. Un modèle pour une telle fonction d'onde est donc la superposition de deux fonctions  
\begin{eqnarray*}
\psi_0 : &x \longmapsto e^{\frac{i}{h} \frac{\tan \theta}{2} x^2 },\\ \psi_{\epsilon} : &x \longmapsto e^{\frac{i}{h} (\frac{\tan \theta}{2} x^2 + \epsilon x)},
\end{eqnarray*}
avec $\epsilon = \sqrt{h}$.

\begin{figure}[!h]
	\begin{center}
		\def\svgwidth{0.5 \columnwidth}
			%% Creator: Inkscape 1.2.1 (9c6d41e4, 2022-07-14), www.inkscape.org
%% Accompanies image file 'interference.pdf' (pdf, eps, ps)
\begingroup%
  \makeatletter%
  \providecommand\color[2][]{%
    \errmessage{(Inkscape) Color is used for the text in Inkscape, but the package 'color.sty' is not loaded}%
    \renewcommand\color[2][]{}%
  }%
  \providecommand\transparent[1]{%
    \errmessage{(Inkscape) Transparency is used (non-zero) for the text in Inkscape, but the package 'transparent.sty' is not loaded}%
    \renewcommand\transparent[1]{}%
  }%
  \providecommand\rotatebox[2]{#2}%
  \newcommand*\fsize{\dimexpr\f@size pt\relax}%
  \newcommand*\lineheight[1]{\fontsize{\fsize}{#1\fsize}\selectfont}%
  \ifx\svgwidth\undefined%
    \setlength{\unitlength}{347.72883591bp}%
    \ifx\svgscale\undefined%
      \relax%
    \else%
      \setlength{\unitlength}{\unitlength * \real{\svgscale}}%
    \fi%
  \else%
    \setlength{\unitlength}{\svgwidth}%
  \fi%
  \global\let\svgwidth\undefined%
  \global\let\svgscale\undefined%
  \makeatother%
  \begin{picture}(1,0.72263975)%
    \lineheight{1}%
    \setlength\tabcolsep{0pt}%
    \put(0,0){\includegraphics[width=\unitlength,page=1]{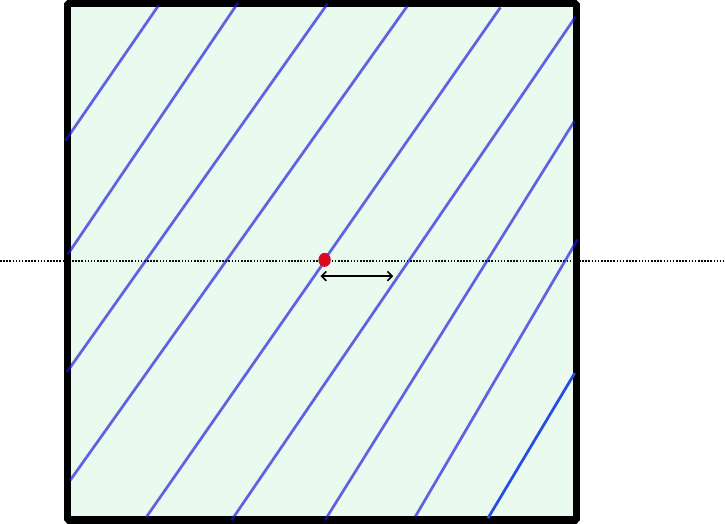}}%
    \put(0.45543838,0.28979721){\makebox(0,0)[lt]{\lineheight{1.25}\smash{\begin{tabular}[t]{l}$\sqrt{h}$\end{tabular}}}}%
  \end{picture}%
\endgroup%

	\end{center}
	\caption{Les branches lagrangiennes au premier temps d'interférence}
			\label{fig:interference}
\end{figure}

\vspace{3mm} \paragraph{\bf Premières interférences le long d'un réseau}

On peut écrire 

$$ \psi_0 + \psi_{\sqrt{h}} = e^{\frac{i}{h} \frac{\tan \theta}{2} x^2 } (1 + e^{\frac{i}{\sqrt{h}}x}).$$

\vspace{3mm}

Et là, on devrait être très content, car 

\begin{enumerate}
\item $e^{\frac{i}{h} \frac{\tan \theta}{2} x^2 }$ est la fonction lagrangienne représentant la droite d'équation $\{ y = \tan \theta x\}$.

\item $(1 + e^{\frac{i}{\sqrt{h}}x})$ est une fonction oscillante, de période $\sqrt{h}$. On peut voir cette dernière comme une addition de fonctions d'amplitudes de paquets d'ondes le long de la droite $\{ y = \tan \theta x\}$.
\end{enumerate}

En somme, si on fait la spéculation simplificatrice que les branches du propagé d'un paquet d'ondes sont des fonctions lagrangienne de la forme $x \longmapsto e^{\frac{i}{h} (\frac{\tan \theta}{2} x^2 + \epsilon x)}$, il doit se former des interférences constructrices et destructrices le long de la variété instable, selon un schéma périodique de période de l'ordre $\sqrt{h}$. Il en résulte une fonction d'onde qui est (à quelques détails près) une somme de paquets d'ondes disjoints centré aux points d'une grille dont la maille vaut $\sqrt{h}$.

\vspace{2mm} 
Il suffirait alors de calibrer $h$ pour que la maille du réseau soit la même que celui des interférences homoclines pour avoir une bonne explication phénoménologique de la remarque de Frédéric Faure. 

\vspace{3mm} 
\paragraph{\bf Retour au bercail} On verra plus tard que l'on peut rendre ces calculs rigoureux; on se retrouve alors avec l'énoncé suivant.

\vspace{3mm}
\fbox{
\parbox{\textwidth}{
Pour certains $h$ bien choisis, au temps d'Ehrenfest $t_E = \frac{|\log(h)|}{\log(\lambda)}$ le propagé d'un paquet d'ondes est une somme de paquets d'ondes centrés aux intersections homoclines du point fixe $(0,0)$.
}}

\vspace{3mm} Il nous faut encore déduire que cette collection de paquets d'ondes, après propagation pendant un temps $t_E$ supplémentaire, se recombinent pour reconstruire le paquet d'ondes initial. 

\vspace{2mm} Si les paquets d'ondes propagés n'avaient pas tendance à s'étaler, cela serait très facile : les paquets d'ondes qu'on a construits sont centrés à des points dont les orbites classiques viennent s'accumuler après un temps $t_E$ sur le point initial $(0,0)$. Malheureusement, l'étalement le long de la variété instable fait que l'on peut seulement dire la chose suivante.

\vspace{3mm}
\fbox{
\parbox{\textwidth}{
Pour certains $h$ bien choisis, à deux fois le temps d'Ehrenfest $2t_E = 2\frac{|\log(h)|}{\log(\lambda)}$, le propagé d'un paquet d'ondes est une somme de lagrangiennes représentants des droites de direction $\theta$ passant toute par un petit voisinage du point $(0,0)$.
}}

\vspace{3mm}

En d'autres termes, le propagé d'un paquet d'ondes (pour les $h$ bien choisis) au temps $2t_E$ est une fonction de la forme 

$$
 x \mapsto \sum_j{a_je^{\frac{i}{h}(\frac{\tan \theta}{2}x^2 + \epsilon_j x)}} ,
$$ 
où $\epsilon_j < \sqrt{h}$ et $a_j$ est un terme de phase qu'on ne connait pas (a priori). En se rappelant que les paquets d'ondes des interférences homoclines sont disposés le long d'un réseau on a de plus que les $\epsilon_j$ sont des multiples entiers d'une constante fixée (un calcul rapide donne que c'est constante est $h$, \textit{i.e.} $\epsilon_j = jh$). En factorisant on peut réécrire 
$$ 
x \mapsto e^{\frac{i}{h}(\frac{\tan \theta}{2}x^2)}\sum_j{a_je^{i j x}}.
$$
La lectrice avisée aura remarqué que si les $a_j$ sont tous égaux à $1$ (ou tous égaux) $\sum_j{e^{i j x}}$ n'est autre qu'une très bonne approximation d'une masse de Dirac en $0$. La fonction 

$$ x \mapsto e^{\frac{i}{h}(\frac{\tan \theta}{2}x^2)}\sum_j{a_je^{i j x}}$$ serait donc très proche d'être un paquet d'ondes en $0$. 

\vspace{3mm}
\fbox{
\parbox{\textwidth}{
Pour obtenir une démonstration interférentielle du résultat de reconstruction, il faut donc calculer ces termes de phase $a_j$.
}}

\vspace{3mm}

C'est en calculant ces termes que nous reconnûmes les sommes de Birkhoff  qui sont l'objet de ce mémoire, et qui sont l'élément central du théorème \ref{thm:principal}. On espère que cette introduction heuristique aura permis au lecteur de reconstruire le cheminement naturel qui nous a amené à  un tel énoncé.

\section{Contenu du texte}

Le \textbf{chapitre \ref{chap:propagation}} est le coeur technique du texte. On y démontre que le propagateur de l'application du chat quantique peut s'exprimer, à $h$ fixé et pour des temps arbitraires, en fonction des sommes de Birkhoff d'un système dynamique lié à la dynamique classique. Précisément, on introduit une méthode systématique pour exprimer purement dynamiquement les termes d'interférences qui apparaissent inexorablement après le temps d'Ehrenfest.

 On expliquera plus précisément d'où vient ce système dynamique dans la section \ref{sec:persp}.

\vspace{2mm} Dans le \textbf{chapitre \ref{chap:marklof}}, on explique une méthode due à Jens Marklov et ses co-auteurs qui permet de contrôler avec précision les sommes de Birkhoff mise en évidence dans le chapitre \ref{chap:propagation}, en fonction d'une condition arithmétique portant sur les valeurs propres de la matrice à coefficients dans $\mathbb{Z}$ définissant l'application du chat.

\vspace{2mm} Le \textbf{chapitre \ref{chap:applications}} est un chapitre de discussion.

\begin{itemize}
\item On explique comment réduire proprement les résultats du chapitre \ref{chap:marklof} à des énoncés concrets sur la propagation des paquets d'ondes. On montre (Théorème \ref{thm:husimi}) que mis à part les $h=1/N$ satisfaisant une certaine condition arithmétique, la mesure de Husimi de l'évolué d'un paquet d'onde reste bornée (uniformément en $h$) en temps très long. On conjecture que ces valeurs de $h$ sont exceptionnelles. 
%\item On y présente des simulations numériques qui nous permettent de comparer nos résultats aux résultats déjà connus sur les applications de chat quantique.
\item On présente un cadre général unifiant applications de chat quantique et l'équation de Schrödinger sur les surfaces riemanniennes à courbure négative. Dans ce cadre unifié, il est assez simple de conjecturer des généralisations de nos résultats décrivant la propagation quantique au-delà du temps d'Ehrenfest au cas des surfaces à courbure négative.
\end{itemize}

\subsection*{Remerciements} Je remercie Fred Faure pour ses nombreuses explications à propos de l'usage des paquets d'ondes en mécanique quantique, Jeff Galkowski pour m'avoir expliqué comment utiliser des fonctions "cut-off" pour obtenir l'estimation de la proposition \ref{prop:poubelle1}. 
Toute ma gratitude va à Jens Marklof qui m'a patiemment appris sa méthode pour gérer les intégrales ergodiques de produits fibrés sur le tore; cela constitue le coeur du chapitre \ref{chap:marklof}. Finalement, je remercie Maxime Ingremeau pour ses commentaires très détaillés sur une version préliminaire de ce texte.

\chapter{Propagation après le temps d'Ehrenfest}
\label{chap:propagation}

\subsection*{Automorphismes linéaires du tore quantifiés}

Dans le cas des variétés riemanniennes, on prédit que lorsque le flot géodésique associé est \textit{fortement chaotique}, 

\begin{enumerate}

\item les fonctions propres du laplacien pour des grandes valeurs propres sont toujours équidistribuées (c'est la propriété connue sous le nom d'\textit{unique ergodicité quantique});

\item la distribution statistique des valeurs propres du laplacien ressemble à celle de grandes matrices aléatoires.

\item Une fonction propre de haute énergie ressemble à une somme aléatoire d'ondes planes.

\end{enumerate}

Afin de mieux comprendre ces questions, physiciens et mathématiciens ont proposé des modèles simplifiés de dynamiques quantiques tels que la dynamique hamiltonienne sous-jacente soit un exemple le plus simple possible de dynamique chaotique, et nous allons nous concentrer sur l'un deux.

\vspace{2mm} Le difféomorphisme du tore $\mathbb{T}^2 = S^1 \times S^1$ défini de la manière suivante 

$$ \begin{array}{ccccc}
M & := & \mathbb{T}^2 & \longrightarrow & \mathbb{T}^2 \\
 & & \begin{pmatrix}
x \\ y
\end{pmatrix} & \longmapsto &  \begin{pmatrix}
2 & 1 \\
1  & 1\end{pmatrix} \cdot \begin{pmatrix}
x \\ y
\end{pmatrix}
\end{array}$$ est fortement chaotique au sens où il satisfait la propriété d'\textit{Anosov}. Cette application est souvent appelée, pour des raisons fantaisistes indépendantes de notre volonté, l'\href{%
https://fr.wikipedia.org/wiki/Chat_d\%27Arnold
}{ \textit{application du chat}}%
. 

\vspace{3mm} Il existe une manière de construire, pour tout $h > 0$ de la forme $h = \frac{1}{N}$ pour $N \in \mathbb{N}^*$, un espace de Hilbert $\mathcal{H}^h$ (qui est l'analogue de $\mathrm{L}^2(M, \mathbb{C})$ dans le cas où on regarde l'équation de Schrödinger sur une variété $M$) et un opérateur unitaire 

$$ \widehat{M^h} :  \mathcal{H}^h \longrightarrow \mathcal{H}^h $$ qui "quantifie" l'application du chat $M$, au sens où c'est un analogue quantique de l'application classique $M : \mathbb{T}^2 \longrightarrow \mathbb{T}^2$. Cette construction, bien connue mais pas complètement évidente, est expliquée en détail dans le paragraphe~\ref{sec:quantification}. Le point important est que cet opérateur quantifié satisfait les propriétés classiques du propagateur de l'équation de Schrödinger :

\begin{itemize}
\item les itérés de $\widehat{M^h}$ sont bien approchés par la dynamique classique de $M$ jusqu'à des temps de l'ordre de $\frac{ |\log h |}{\chi}$ où $\chi$ est l'exposant de Liapounoff de $M$;

\item passé le temps d'Ehrenfest $t_E(h) = \frac{ |\log h |}{\chi}$, des effets d'interférence apparaissent et la dynamique de $\widehat{M^h}$ cesse d'être approximée de manière simple par celle de $M$.

\end{itemize}

Cette application du chat quantique a beaucoup été étudiée, aussi bien par des physiciens que des mathématiciens. On n'infligera pas à notre lectrice la traditionnelle liste-d'articles-traitants-du-sujet-qu'il-convient-de-citer. On ne peut cependant pas faire l'économie d'une référence à l'important article \cite{BNF} qui est le point de départ de notre travail. Il y est démontré, par des méthodes algébriques, que de surprenantes interférences constructives peuvent se produire (dans certains cas) juste après le temps d'Ehrenfest. En plus de préciser les limites du spectaculaire résultat de délocalisation \cite{A,AN}, il indique exactement où chercher pour comprendre les mécanismes contrôlant les interférences.

\subsection*{Résultat(s)}

Dans ce chapitre, on propose une approche qui permet de calculer les termes d'interférence qui se produisent après le temps d'Ehrenfest. 

\vspace{3mm} Le résultat principal qu'on obtient est que, grossièrement,  ces interférences peuvent s'exprimer comme les \textit{sommes de Birkhoff} d'un système dynamique \textit{parabolique} (dérivé de la dynamique du feuilletage instable). Dans une note ultérieure, on montrera que ce système dynamique peut être analysé avec précision, et que cette analyse permet de décrire qualitativement le propagateur quantique, sans limite de temps, à $h$ fixé.

\vspace{2mm} Plus précisément, on décrit la dynamique de l'application du chat quantique en calculant les coefficients de ses itérés, écrits dans la base des \textit{paquets d'ondes}. Un paquet d'ondes est une fonction de la forme

$$ \begin{array}{ccccc}
\Phi_{(x_0,v_0)}^h & := & \mathbb{R} & \longrightarrow & \mathbb{R} \\
 & & x & \longmapsto & C_h e^{\frac{-(x-x_0)^2}{h}} e^{2\pi i \frac{v_0}{h}x} 
\end{array}$$ pour une paire $(x_0,v_0) \in \mathbb{R}^2$. La construction de l'espace $\mathcal{H}^h$ fait qu'il est en fait possible de penser à $\Phi_{(x_0,v_0)}^h$ comme un élément de $\mathcal{H}^h$, où $(x_0,v_0)$ représente maintenant un point dans $\mathbb{T}^2 = \mathbb{R}^2/ \mathbb{Z}^2$. Il est de plus possible, en choisissant les point $x_0,v_0$ sur un quadrillage dont les points sont espacés d'une distance $\sqrt{h}$, de construire une base (presque) orthonormée de $\mathcal{H}^h$ faite de paquets d'onde. On a alors un résultat qui \textit{ressemble} au théorème suivant (on renvoie le lecteur au théorème \ref{thm:principal} pour un énoncé complet et correct).

\vspace{3mm}

\fbox{
\parbox{\textwidth}{
\begin{theoreme}
\label{thm:A}
Il existe un système dynamique parabolique $ T_h : X \longrightarrow X$ tel que pour tout $(x,v) \in \mathbb{T}^2$, il existe une observable $f_{(x,v)}^h : X \longrightarrow \mathbb{C}$  tellr qu'on ait pour tout $(x_0,v_0) \in \mathbb{T}^2$
$$ \left\langle  (\widehat{M^h})^n \cdot \Phi_{(x_0,v_0)}, \Phi_{(x,v)} \right\rangle_{\mathcal{H}^h} = \frac{1}{\sqrt{\lambda^n}} \sum_{k= -\lambda^n}^{k = \lambda^n}{f_{(x,v)}^h \circ T_h^k(s(x_0,v_0)) } + \text{erreur}$$
où $s(x_0,v_0)$ est un point de $X$ ne dépendant que de $(x_0,v_0)$ et l'erreur tend vers $0$ quand $h$ tends vers $0$, uniformément en $n$.
\end{theoreme}
}
}
\vspace{2mm}

On commente cet énoncé.

\begin{itemize}

\item Le terme $\left\langle  (\widehat{M^h})^n \cdot \Phi_{(x_0,v_0)}, \Phi_{(x,v)} \right\rangle_{\mathcal{H}^h}$ apparaissant dans cet énoncé est le coefficient de la matrice $(\widehat{M^h})^n$, écrite dans une base de paquets d'ondes. On peut donc considérer cet énoncé comme une description explicite de $(\widehat{M^h})^n$ (toujours pour des temps $n$ plus grand que le temps d'Ehrenfest). 

\item Le terme de droite (sans l'erreur) s'écrit sous la forme 

$$
\frac{1}{\sqrt{T}} \sum_{k = -T}^{T}{ f \circ T_h^k(s)},
$$
qui est une \textit{somme de Birkhoff}. C'est un objet central en théorie ergodique et en probabilité. On notera que la normalisation en $\frac{1}{\sqrt{T}}$, est l'échelle classique du théorème de la limite centrale. Quand un système dynamique est suffisamment chaotique, on peut espérer que de tels termes se comportent comme un tirage aléatoire de nombres suivant une distribution gaussienne quand on fait varier $s$. Les choses ne sont pas aussi simple que ça (car le système dynamique $T$ est loin d'être chaotique), mais c'est une analogie intéressante à avoir en tête.

\end{itemize}

\medskip

\subsection*{Contenu du chapitre.} Voici le chemin que l'on suit pour démontrer le théorème~\ref{thm:A}. 

\medskip

\paragraph*{\it Quantification des applications linéaires de $\mathbb{R}^2$} Les deux premières sections \ref{sec:quantification} et \ref{sec:microlocal} expliquent comment construire la \textit{quantification} d'une application linéaire de $\mathbb{R}^2$ associée à une matrice $M \in \mathrm{SL}_2(\mathbb{R})$. Il s'agit de construire un opérateur unitaire $\widehat{M}^h : \mathrm{L}^2(\R) \longrightarrow \mathrm{L}^2(\R)$ associé en utilisant les règles de quantification usuelles. Dans un second temps, on explique comment, dans le cas où la matrice $M$ est à coefficients entiers, l'opérateur $\widehat{M}^h$ induit une action unitaire sur un espace de distributions invariantes par translations, en espace et en fréquence. Cet opérateur induit devient l'analogue quantique de l'application que la matrice $M$ induit sur le tore $\mathbb{T}^2 = \mathbb{R}^2/ \Z^2$.

\medskip

\paragraph*{\it Evolution en temps long d'un paquet d'ondes au revêtement universel, i.e. dans $\mathrm{L}^2(\mathbb{R})$} L'idée derrière le calcul qui aboutit au théorème \ref{thm:A} est relativement simple: on commence par obtenir une description simple de l'évolution d'un paquet d'ondes sous l'action de $\widehat{M}^h : \mathrm{L}^2(\R) \longrightarrow \mathrm{L}^2(\R)$ pour des temps arbitrairement longs. Comme la dynamique classique induite par $M$ n'est pas récurrente, on est en mesure de montrer qu'aucun effet d'interférence ne limite la validité de l'approximation de l'évolution par $M^h$ par la dynamique classique de $M$. On précise:

\vspace{2mm} 
Quand $M \in \mathrm{SL}_2(\Z)$ est une matrice hyperbolique, comme dans le cas classique $M = \begin{pmatrix}
2 & 1 \\
1 & 1
\end{pmatrix}$, un nuage de points autour de $(0,0)$ tend à grossir sous l'action répétée de $M$, et à venir s'étaler le long d'une droite dans la direction de la droite propre de $M$ qui est dilatée (droite instable), et ce à une vitesse exponentielle. 
\medskip Le pendant quantique de cet énoncé est qu'un paquet d'ondes ressemble de plus en plus à l'\textbf{état lagrangien} représentant la droite instable. Si cette droite est la droite de pente $\theta$, en d'autres termes la droite d'équation $y = \tan \theta \cdot x$, l'état lagrangien associé est  la fonction 
$$ x \mapsto e^{\frac{i \pi}{h} \tan \theta \cdot x^2}.$$ 

En pratique on démontre que l'image d'un paquet d'ondes au temps $n$ est égale à cet état lagrangien, tronqué par une fonction à décroissance lente dont le support est essentiellement l'intervalle $[-\sqrt{h} \lambda^n, -\sqrt{h} \lambda^n]$\footnote{Le nombre $\sqrt{h} \lambda^n$ est exactement le diamètre de l'image d'une boule de rayon $\sqrt{h}$ par la dynamique classique (qui est l'espace qu'occupe un paquet d'ondes dans l'espace des phases).}, où $\lambda$ est le taux de dilatation de la matrice $M$. C'est le contenu de la proposition \ref{prop:approx}, dont la démonstration est donnée dans la section \ref{sec:approx}. Ce calcul se fait directement en utilisant la description explicite du propagateur $\hat{M}^h$ donnée par la \textbf{représentation métapléctique}, qu'on redémontre dans l'annexe \ref{ann:propagateur}.

\medskip

\paragraph*{\it Calcul des termes d'interférence} On termine par le calcul des termes d'interférences à proprement parler. Ici, l'idée est que la lagrangienne $ x \mapsto e^{\frac{i \pi}{h} \tan \theta \cdot x^2}$, quand on la projette au quotient contribue au produit scalaire avec un paquet d'ondes donné $\Phi^h_{(x_0,v_0)}$ à chaque fois que la droite $y = \tan \theta \cdot x$ (quand on la projette dans le tore $\mathbb{R}^2/\Z^2$) passe proche de $(x_0,v_0)$ (en particulier seuls les passages à une distance de l'ordre de $\sqrt{h}$ ou moins contribuent réellement). 

\medskip

Une fois qu'on a cette idée en tête, il y a un travail préliminaire à faire pour montrer

\begin{itemize}
\item  que l'approximation de la proposition \ref{prop:approx} passe bien au quotient, c'est la proposition \ref{prop:poubelle1};

\item que les contributions loin de la droite $v = \tan \theta \cdot x$ ne comptent pas, c'est la proposition \ref{prop:poubelle2}.

\end{itemize}

Un peu plus formellement, on utilise une estimation de la forme suivante 
$$ \langle  e^{\frac{i \pi}{h} \tan \theta x^2}, \Phi^h_{(x,v)} \rangle_{\mathrm{L}^2} \sim \exp{\left(-\frac{d\left((x,v), \mathcal{D}_{\theta}\right)}{\sqrt{h}}\right)^2} $$ 
où $\mathcal{D}_{\theta}$ est la droite d'équation $v = \tan \theta \cdot x$ et $d\left((x,v), \mathcal{D}_{\theta}\right)$ est la distance du point $(x,v)$ dans $\mathbb{R}^2$ à la droite $\mathcal{D}_{\theta}$. Cela nous permet finalement de bien approximer les termes d'interférences par des sommes qu'on voit être des sommes de Birkhoff du système dynamique suivant 

$$ \begin{array}{ccccc}
T_h & : & \mathbb{T}^2 &  \longrightarrow & \mathbb{T}^2 \\
 &  & (x,y) & \longmapsto  &(x + \tan \theta, y + \frac{x}{h})
\end{array}.$$

\noindent La partie $x \mapsto x + \alpha$ de la première coordonnée correspond à la dynamique du feuilletage instable de l'application classique; l'autre dimension de la dynamique est là pour tenir compte du fait que le long de la lagrangienne, la phase varie et que le produit scalaire $\langle  e^{\frac{i \pi}{h} \tan \theta \cdot x^2}, \Phi^h_{(x,v)} \rangle_{\mathrm{L}^2}$, bien que ne dépendant en module que de la distance entre $(x,v)$ et $\mathcal{D}_{\theta}$, a une phase qui varie le long de la droite, exactement en $e^{\frac{i \pi}{h}  \tan \theta \cdot x^2}$.

\medskip Ce calcul est fait par la proposition \ref{prop:interference} et mis sous forme de sommes de Birkhoff dans le théorème \ref{thm:principal}.

%
%\part*{L'application du chat quantique, au-delà du temps d'Ehrenfest}

\section[Quantification de hamiltoniens quadratiques]{Quantification de hamiltoniens quadratiques et d'applications linéaires de $\R^2$}

\label{sec:quantification}

\subsection{L' (les) application(s) du chat}

On rappelle au lecteur ce qu'on a communément pris l'habitude d'appeler "application du chat". Il s'agit du système dynamique induit par une application de la forme 

$$\begin{array}{ccc}
\mathbb{T}^2 = \R^2 / \mathbb{Z}^2 & \longrightarrow &  \TT \\
 \begin{pmatrix}
 x \\ 
 y
 \end{pmatrix} & \longmapsto & A \cdot \begin{pmatrix}
 x \\ 
 y
 \end{pmatrix} 
\end{array}$$ où $A$ est une matrice \textit{hyperbolique} de $\mathrm{SL}(2,\Z)$. On rappelle qu'on dit d'une matrice qu'elle est hyperbolique si ces valeurs propres sont toutes de module différent de $1$. Un exemple typique de telle matrice est 

$$ \begin{pmatrix}
2 & 1 \\
1 & 1
\end{pmatrix},$$ c'est ce cas particulier qui est souvent appelé application du chat d'Arnold. Dans ce texte on fait l'abus de terminologie d'appeler toute telle application (avec $A$ hyperbolique) application du chat. Il s'agit des exemples les plus simples de difféomorphismes dits \textit{d'Anosov}.

\subsection{Quantification d'un flot hamiltonien quadratique de $\RR^2$.} 
Cette section a pour but de décrire la quantification de l'application du chat avec laquelle nous allons travailler. \\

Notons qu'il n'est pas trivial de quantifier une telle application : le tore $\TT^2$ n'est pas le fibré tangent d'une variété, c'est une variété compacte, et il n'y a pas de raison \textit{a priori} de penser que l'on peut lui trouver un analogue quantique. \\

Rappelons par ailleurs que l'application du chat $M$ n'agit pas a priori sur $\TT^2$ mais sur $\RR^2$ en tant qu'application linéaire: il faut vérifier qu'elle descend bien une application bien définie sur $\TT^2$. On va procéder de manière analogue pour définir sa quantification : on va d'abord quantifier l'action linéaire de $M$ sur $\RR^2$ pour ensuite la restreindre à l'analogue du tore quantique. Pour cela, nous allons procéder en trois étapes naturelles dans ce contexte, 

\begin{enumerate}
	\item (partie classique) on trouve un hamiltonien $H : T \, \RR = \RR^2 \to \RR$ tel que l'application du chat $M$ soit le temps $1$ du flot hamiltonien associé ;
		$$ M = \varphi^H_1 \ . $$ 
	\item on quantifie ce flot hamiltonien pour obtenir une version quantique de l'application du chat, c'est à dire une famille d'opérateurs indexée par $h > 0$ 
		$$ \widehat{M}^h : L^2 (\RR) \to L^2 (\RR) \ . $$
	\item on restreint cet opérateur à un certain espace de distributions périodiques pour le voir agir sur l'équivalent quantique du tore.
\end{enumerate}

Notons que l'on identifie $\RR^2$ avec le fibré tangent de $\RR$ afin d'utiliser la quantification bien connue de $T \RR$ en tant que fibré tangent.  \\

Les deux premières étapes du processus précédent sont assez classiques, mais nous les rappellerons ici quand même pour le confort du lecteur. La troisième étape est un peu plus originale et sera également détaillée.

\subsection{Applications linéaires comme temps 1 de flots hamiltoniens.} Rappelons pour commencer ce qu'est un système hamiltonien dans ce contexte. \\

On dit d'un champ de vecteurs $X_H$ sur $T\, \RR = \RR^2$ que c'est un champ \textbf{hamiltonien} s'il s'écrit comme le gradient symplectique d'une fonction $H : T\, \RR = \RR^2 \to \RR$, c'est à dire si pour tout vecteur $Y$ 
	$$ d H(Y) = \omega (Y, X_H) \ , $$ 

où $\omega := dx \wedge d \xi$ est la forme symplectique canonique de $\RR^2$, c'est à dire sa forme volume canonique. Notons qu'on utilise ici les conventions classiques : la variable $x$ est la variable position et la variable $\xi$ la variable vitesse. Notons également qu'un flot hamiltonien préserve, par construction, la forme volume du plan : on ne peut donc espérer quantifier que les applications linéaires qui préservent le volume de $\RR^2$, comme par exemple l'application du chat. \\ 

Soit donc une matrice $M$ de déterminant 1, matrice que l'on confondra avec l'application linéaire associée. \\

Le but ici est  de trouver $H$ tel que le flot au temps $1$ du champ hamiltonien~$X_H$ correspondant soit donné par la matrice $M$. On va voir qu'on  va pouvoir choisir~$H$ de la forme
$$ H(x,\xi) := \alpha x^2 /2 + \gamma x \xi + \beta \xi^2 / 2 \ .$$

En effet le gradient symplectique $X_H$ de cette application est 
	$$ X_H(x,\xi) = \begin{pmatrix}
		  \gamma &   \beta \\  - \alpha &  - \gamma 
	\end{pmatrix} \begin{pmatrix}
	x \\ \xi
\end{pmatrix} \ . $$
On note 
	$$ m := \begin{pmatrix}
		 \gamma &  \beta \\  - \alpha & - \gamma 
	\end{pmatrix} $$

C'est donc un champ de vecteurs dont les coefficients dans la base canonique sont données par des combinaisons linéaires des coordonnées du point au dessus duquel se pose le vecteur. Un tel champ de vecteur s'intègre explicitement grâce à l'exponentielle d'une matrice : on peut vérifier que 
	$$ \fonctionbis{\RR^2 \times \RR}{\RR^2}{((x,\xi), t)}{\exp( t m) \begin{pmatrix}
	x \\ \xi 
\end{pmatrix}}  $$

est le flot du gradient symplectique de $H$. En particulier, si on veut $\exp( t m) = M$ au temps $t=1$ il nous faut trouver $m$ tel que
	$$ \exp(m) = M  \ , $$
ce qui nous donne la relation entre les coefficients $\alpha, \beta$ et $\gamma$ et les coefficients de la matrice $M$. 

\subsection{Quantification des applications linéaires du plan.}

Un hamiltonien quadratique comme ceux présentés ci-dessus donne naissance à un flot dont le temps 1 est une application linéaire qui préserve le volume de $\RR^2$. L'analogue quantique d'une telle application linéaire est donnée par l'opérateur qui à une à une fonction $f$ associe la solution au temps $1$ de l'équation de Schrödinger correspondante.  Afin de rendre tout cela explicite, commençons par un bref rappel de la quantification de Weyl des hamiltoniens quadratiques. \\

On pensera souvent dans la suite à $\RR^2$ comme à l'espace tangent de $\RR$. Nous cherchons à associer au hamiltonien 
	$$ H = \alpha x^2 / 2 + \gamma x \xi + \beta \xi^2 /2 $$
un opérateur qui agit sur les fonctions de la base, c'est à dire les fonctions de $\RR$ dans $\CC$. On va utiliser la quantification de Weyl : à la variable $x$ (que l'on pense ici comme une observable de $T \RR \to \RR$) on associe l'opérateur multiplication par $x$
	$$ \widehat{x}^h(f) := x f \ , $$
qui ne dépend pas de $h$. A la variable $\xi$, on associe l'opérateur multiplication par $\xi$ mais conjugué par la $h$-transformée de Fourier  (la transformée de Fourier à l'échelle $h$) dont nous rappelons la définition 
$$ \cF_h f (\xi ) := \frac{1}{ h} \inte{\RR} e^{ - \frac{2i \pi x \cdot \xi}{h}} \  f(x) \ dx \ . $$

Rappelons que le coefficient $h^{-1}$ fait de $\cF_h f$ l'application qui donne la décomposition en onde plane, à l'échelle $h$, d'une fonction $f$:
$$ f(x) = \inte{\RR} \widehat{f}^h(\xi) e^{\frac{2i \pi x \cdot \xi}{h}} \ d \xi \ . $$

Une autre façon de le formuler est de dire que l'opérateur $f \mapsto \widehat{f}^h$ est inversible d'inverse 
	$$ f \mapsto \left(x \mapsto \inte{\RR} f(\xi) e^{\frac{2i \pi x \cdot \xi}{h}} \ d \xi \right) \ . $$
Revenons à l'opérateur moment : on peut vérifier que 
\begin{align*}
	\widehat{\xi}^h(f) & =  \cF_h^{-1} \xi \cF_h f \\
	&  = \frac{h}{2 i \pi} \dfrac{df}{dx} \ .
\end{align*}

On obtient alors le quantifié de $H$ en utilisant les règles de la quantification de Weyl: on a
	$$ \widehat{H}^h = \frac{1}{2} \left( 
	\alpha (\widehat{x}^h)^2  + \gamma (\widehat{x}^h \widehat{\xi}^h + \widehat{\xi}^h \widehat{x}^h) + \beta (\widehat{\xi}^h)^2  \right) \ . $$
Dans notre cas, explicitement, pour toute fonction $f$,
	$$ [\widehat{H}^h(f)](x) = \frac{1}{2} \left( \alpha x^2 f(x) + \gamma \left( x \frac{h}{2i \pi} \partial_x f (x) +  \frac{h}{2i \pi} \partial_x (xf)(x) \right) - \beta  \frac{h^2}{4 \pi^2} (\partial_x^2 f)(x)  \right) \ .$$
	
L'équation de Schrödinger correspondante s'écrit alors 
\begin{equation}
\label{eq:defschrodinger}
	\frac{i h}{2 \pi} \partial_t  = \widehat{H}^h \ . 
\end{equation}	 
C'est une équation \emph{d'évolution} dont les solutions sont des familles de fonctions de $\RR$ dans $\RR$, indexées par $t$.
	
On note $\widehat{M}^h_t$ le propagateur de l'équation précédente : c'est l'opérateur qui à une fonction $f=f_0$ (condition initiale), associe la solution $f_t$ correspondante de l'équation \eqref{eq:defschrodinger} au temps $t$.
 	
On définit 	$a_t, b_t$ et $c_t$  des fonctions de $t$ avec $a_t > 0$  par
$$
\begin{pmatrix} a_t & b_t \\ c_t & d_t \end{pmatrix} = \exp(tm) \quad \text{où} \quad \ m = \begin{pmatrix} \gamma &  \beta  \\ - \alpha & - \gamma \end{pmatrix}.
$$

\begin{proposition}
\label{prop:propagateurtempt}
	On a l'expression suivante pour $\widehat{M}^h_t$ : 
\begin{equation}
	\label{eq:propagateurtempst}
		 [\widehat{M}^h_t f] (x) = \frac{1}{a_t^{1/2}}\inte{\RR} e^{ \frac{2 \pi i S_t(x, \xi)}{h}} \, \widehat{f}^h(\xi) \ \dd \xi,
\end{equation}

où on a noté
\begin{equation}
	\label{eq:actiontempst}
		S_t(x, \xi) = \dfrac{1}{2 a_t}\left(c_t x^2 + b_t\xi^2 + 2 x \xi\right)
\end{equation}
\end{proposition}

On peut vérifier, par exemple, que pour $\gamma = \ln(\lambda)$ et $\alpha = \beta = 0$ le membre de droite de l'équation \eqref{eq:propagateurtempst} donne 
	\begin{equation}
		\label{eq:soldilationschrodinger}
		 (x,t) \mapsto \lambda^{-t/2} f(\lambda^{-t} x) 
	\end{equation}
est bien solution de l'équation de Schrödinger \eqref{eq:defschrodinger}. Ceci est développé dans l'annexe \ref{ann:propagateur} \\ 

\fbox{
\parbox{\textwidth}{
\begin{remark}
On peut se demander si l'on peut justifier \textit{a priori} l'écriture de $S$ ci-dessus. Qualitativement, on voudrait dire que la quantification d'une application linéaire agit linéairement sur les états lagrangiens. \\
L'état lagrangien associé à la droite du plan  $\mathcal{D}_{\xi_0}$ d'équation $ \xi = \xi_0 $, c'est à dire une onde plane de fréquence de $\xi_0$, devrait donc être envoyé sur l'état lagrangien associé à la droite image par l'application linéaire $M$, c'est à dire la droite $ M (\mathcal{D}_{\xi_0}) $. C'est la droite d'équation $ \xi = c/a \cdot x + \xi_0 /a $  où 
	$$ M = \begin{pmatrix}
		a & b \\ c & d 
	\end{pmatrix} \, . $$
La fonction lagrangienne associée à cette droite est donnée par 
	$$ 
		x \mapsto \exp\big ( \frac{2 i\pi}{h} \ \frac{ c x^2 + 2 x \xi_0  }{2 a}\big ) \ , 
	$$
et on retrouve bien les coefficients de la phase qui concerne la variable $x$. 
\end{remark}
}}

%  
%\begin{remark}\footnote{LB: Je vois pas ce que cette remarque apporte, je la comprends pas en fait.} La notation $S$ suggère la notion d'action. Afin de faire le lien, faisons le rappel suivant. On se donne un hamiltonien $H$ qui ne dépend pas du temps et on note $\mathcal{L}$ le lagrangien associé. Pour deux points $x,y$ proches et un temps $t > 0$ on peut définir
%	$$ S(x,y,t) := \inte{0}^t \mathcal{L}(\gamma(s),  \dot{\gamma}(s)) ds \ , $$ 
%où $\Gamma := (\gamma, \dot{\gamma})$ est l'unique solution de $\dot{\Gamma} = X_H(\Gamma) $ avec les conditions au bord $\gamma(0) = x$ et $\gamma(t) = y$. On peut alors passer d'un voisinage de la diagonale dans le produit, c'est à dire des variables $x, y$, à une variable dans le cotangent, c'est à dire à une forme différentielle $\xi_x$ basée en $x$. Explicitement, on définit $\xi_x = \xi_x(y,t)$ comme l'unique élément du cotangent tel que 
%	$$ \pi \left[ \Phi_t^H (\xi_x(y,t)) \right] =  y \ . $$
%\end{remark}

\vspace{3mm}

Revenons à notre but initial. Nous définissons l'analogue quantique de l'action d'une application linéaire sur $\RR^2$ en imposant $t=1$ dans \eqref{eq:propagateurtempst}. 

\begin{definition}[Représentation métaplectique]
Si 
$$M := \begin{pmatrix} a & b \\ c & d \end{pmatrix} $$ est une matrice de déterminant $1$ alors on définit	
\begin{equation}
	\label{eq:propagateurtemps1}
		 [\widehat{M}^h f] (x) = \frac{1}{a^{1/2}}\inte{\RR} e^{ \frac{2 \pi i S(x, \xi)}{h}}\, \widehat{f}^h(\xi) \ \dd \xi \ ,
\end{equation}

où on a noté
\begin{equation}
	\label{eq:actiontemps1}
		S(x, \xi) = \dfrac{1}{2 a}\left(c x^2 + b \xi^2 + 2 x \xi\right) \ .
\end{equation} 
\end{definition} 

On appelle l'application 
	$$ M \mapsto \widehat{M}^h $$
la représentation métaplectique de paramètre $h$. Comme son nom l'indique, c'est un morphisme de $\mathrm{SL}_2(\RR) $ dans les isométries $U(L^2(\RR))$ de  $L^2(\RR)$. Cette propriété de morphisme va nous être utile dans la section suivante.

\subsection{Quantification des applications linéaires du tore $\TT = \mathbb{R}^2 / \Z^2$}	

On a montré dans les paragraphes précédents comment \textit{quantifier} les applications linéaires de $\R^2$ préservant le volume : pour tout paramètre $h > 0$, on a associé à une matrice $M \in \mathrm{SL}(2,\R)$ un opérateur unitaire 
$$ \widehat{M}^h : \mathrm{L}^2(\R) \longrightarrow \mathrm{L}^2(\R) $$ dont l'action sur certaines fonctions lagrangiennes émule l'action de $M$ sur les sous-variétés lagrangiennes de $\R^2 = \mathrm{T} \R$ correspondantes.

\vspace{2mm} 

Dans le cas particulier où $M$ est à coefficients entiers, on voudrait étendre cette construction à une action unitaire $  \widehat{M}^h : \mathcal{H}^h \longrightarrow  \mathcal{H}^h$ où $ \mathcal{H}$ serait un espace de Hilbert qui serait un analogue de l'espace $\TT = \R^2 / \Z^2$, de la même manière que $\mathrm{L}^2(\R)$ est un analogue de $\mathrm{T}\R = \mathbb{R}^2$. 

\vspace{2mm} On propose les analogues de l'action par translations de $\mathbb{Z}^2$ sur $\R^2$ suivants. On choisit sur $\R^2$ les coordonnées $(x, \xi)$ pour bien se rappeler qu'on y pense comme le fibré tangent de $\R$.

\begin{itemize}

\item La translation $T_{(1,0)}$ de vecteur $(1,0)$ a pour analogue 
$$\widehat{T}^h_{(1,0)} := f \longmapsto (x \mapsto f(x-1)).$$

\item La translation $T_{(0,1)}$ de vecteur $(0,1)$ a pour analogue 
$$\widehat{T}^h_{(0,1)} := f \longmapsto (x \mapsto e^{\frac{2i\pi x}{h}}f(x)).$$

\end{itemize}

Cette suggestion n'est pas arbitraire, si on cherche des hamiltoniens dont les flots au temps $1$ sont les translations qu'on considère, la procédure de quantification expliquée dans les paragraphes précédents nous donnera  $\widehat{T}^h_{(1,0)}$ et $\widehat{T}^h_{(0,1)}$. On définit de manière générale le quantifié d'une translation de vecteur $(a,b)$ de la manière suivante:
%\begin{definition}
$$ \widehat{T}^h_{(a,b)} := \exp \big(-\frac{2i\pi}{h}(-b \widehat{x}^h + a \widehat{\xi}^h)\big).$$

%\end{definition}

\begin{proposition}
\label{prop:commutation}
Avec la définition ci-dessus on a 

\begin{enumerate}

\item $\widehat{T}^h_{(a,b)}(u) =  x \mapsto e^{-\frac{i\pi}{h} ab} e^{\frac{2i \pi}{h}bx} \, u(x-a)$.

\item $\widehat{T}^h_{\vec{v}_1}  \circ \widehat{T}^h_{\vec{v}_2}  =  e^{-\frac{i\pi}{h}\det(\vec{v}_1, \vec{v}_2)} \widehat{T}^h_{\vec{v}_1 + \vec{v}_2}$  pour deux vecteurs $\vec{v}_1$ et $\vec{v}_2$.

\item En particulier $ \widehat{T}^h_{(1,0)}  \circ \widehat{T}^h_{(0,1)} =  e^{\frac{2i\pi}{h}} \widehat{T}^h_{(0,1)} \circ \widehat{T}^h_{(1,0)}.$

\end{enumerate}

\end{proposition}

On voit en particulier que 	$\widehat{T}^h_{(1,0)}$ et $\widehat{T}^h_{(0,1)}  $ commutent si et seulement si $h = \frac{1}{N}$ pour $N \in \mathbb{N}^*$.

\begin{proof}
On rappelle que l'opérateur $\widehat{T}^h_{(a,b)}$ est défini comme suit: pour toute fonction $u$, la fonction $\widehat{T}^h_{(a,b)} u$ est la solution $u_t$ au temps $t = -\frac {2i\pi}{h}$ de l'équation au dérivées partielles suivante:
\begin{equation}
\label{edp}
\partial_t u_t(x) = (-b \widehat x^h + a \widehat \xi^h)u_t (x)= -bx \,  u_t (x) + \frac {ah}{2i\pi} \, \partial_x u_t (x)
\end{equation}
avec condition initiale $u_0 = u$.
\begin{enumerate}
\item On a juste à vérifier que la fonction 
$$u_t(x) = e^{\frac{t ab}{2}}\,e^{-t bx} u(x+ath/2i\pi)$$ est solution de \eqref{edp}.
\item On utilise (1), pour $\vec v_1 = (a_1,b_1)$ et $\vec v_2 = (a_2,b_2)$, on a
\begin{multline*}
\widehat{T}^h_{\vec{v}_1}  \circ \widehat{T}^h_{\vec{v}_2} \, u(x) = 
\widehat T^h_{(a_1,b_1)} \left( e^{-\frac {i\pi}{h} a_2b_2} e^{\frac {2i\pi}{h} b_2 x}\, u(x-a_2)\right)\\
=e^{-\frac {i\pi}h a_1b_1} e^{\frac {2i\pi}h b_1x}e^{-\frac {i\pi}h a_2b_2} e^{\frac{2i\pi}h b_2(x-a_1)} \, u(x-a_2-a_1)\\
=e^{-\frac {i\pi}h (a_1b_1 + a_2 b_2)} e^{\frac {2i\pi}h (b_1+b_2)x} e^{-\frac{2i\pi}h b_2a_1} \, u(x-(a_1+a_2))
\end{multline*}
et par ailleurs
\begin{multline*}
e^{-\frac{i\pi}{h}\det(\vec{v}_1, \vec{v}_2)} \widehat{T}^h_{\vec{v}_1 + \vec{v}_2}\, u(x)\\
= e^{-\frac{i\pi}{h} (a_1b_2-a_2b_1)} e^{-\frac {i\pi}h (a_1+a_2)(b_1+b_2)} e^{\frac{2i\pi}h (b_1+b_2)x} \, u(x-(a_1+a_2))\\
=e^{-\frac {i\pi}h (a_1b_1 + a_2 b_2)} e^{\frac {2i\pi}h (b_1+b_2)x} e^{-\frac{2i\pi}h b_2a_1} \, u(x-(a_1+a_2))
\end{multline*}
%\textcolor{red}{a écrire
\item Découle de (2).
\end{enumerate}
\end{proof}

Finalement on a aussi la relation d'équivariance suivante. 

\begin{proposition}
\label{prop:equiv}
Pour toute matrice $M \in \mathrm{SL}_2(\R)$ et tout vecteur $\vec{v} \in \mathbb{R}^2$

$$ \widehat M^h \circ \widehat T^h_{\vec{v}} =  \widehat T^h_{M\cdot \vec{v}}  \circ  \widehat M^h.$$

\end{proposition}

\begin{proof}

Il suffit de vérifier que la relation $$ \widehat M^h \circ \widehat T^h_{\vec{v}} =  \widehat T^h_{M\cdot \vec{v}}  \circ  \widehat M^h$$ est vraie pour les ondes planes de la forme $x \mapsto e^{i\xi \cdot x}$, qui est un calcul direct qu'on laisse à la lectrice.

\end{proof}

\vspace{2mm}

\fbox{
\parbox{\textwidth}{
Dans la suite de ce chapitre, on fera l'hypothèse implicite que $\frac{1}{h} \in 2\mathbb{N}$.
}
}
	
\vspace{3mm}

\subsection{L'espace de distributions}

Pour construire notre espace de distributions analogue de $\mathbb{R}^2 / \Z^2$, on voudrait prendre le quotient de $\mathrm{L}^2(\R)$ par l'action du groupe (isomorphe à $\Z^2$) engendré par $\widehat T^h_{(1,0)}$ et $\widehat T^h_{(0,1)} $. En d'autres termes, on voudrait considérer les éléments de $\mathrm{L}^2(\R)$ invariants par  $\widehat T^h_{(1,0)}$ et $\widehat T^h_{(0,1)}  $. Si les éléments de $\mathrm{L}^2(\R)$  invariants par $\widehat T^h_{(1,0)}$ sont seulement les fonctions $1$-périodiques, il n'existe pas de \textit{fonction} dans $\mathrm{L}^2(\R)$   invariante par $\widehat T^h_{(0,1)} $. On va donc utiliser un espace de distributions. 

\vspace{2mm} On rappelle la définition de l'espace des fonctions de Schwartz 

$$ \mathcal{S}(\R)  := \{ f : \mathbb{R} \longrightarrow \mathbb{C} \ \text{de classe} \ \mathcal{C}^{\infty} \ | \ \forall k,n \in \mathbb{N} \ \Vert x \mapsto x^k \mathrm{D}^n(f)(x) \Vert_{\infty} < \infty \} $$ et de l'espace des distributions de Schwartz

 $$\mathcal{S}'(\R) := \{ D : \mathcal{S}(\R) \longrightarrow \R \ | \ \text{forme linéaire continue} \}.$$ Ici les formes linéaires sont continues pour la structure d'espace de Fréchet induite par les semi-normes $\Vert x \mapsto x^k \mathrm{D}^n(f)(x) \Vert_{\infty}$. Les opérateurs $\widehat T^h_{(1,0)}$ et $\widehat T^h_{(0,1)} $ agissent naturellement sur $\mathcal{S}'(\R)$ par pré-composition (i.e. $\widehat T^h_{(1,0)} \cdot D := f \longmapsto D(\widehat T^h_{(1,0)}(f))$ pour $f \in \mathcal S(\R)$). On peut donc définir l'espace suivant.

\begin{definition}

Si $\frac{1}{h} = N \in \mathbb{N}$, on définit
$$ \mathcal{H}^h = \mathcal{H}^N := \{ D \in \mathcal{S}'(\R) \ | \ \widehat T^h_{(1,0)} \cdot D = \widehat T^h_{(0,1)} \cdot D = D \}.$$

\end{definition}

Une distribution $D \in \mathcal{S}'(\mathbb{R})$ qui est invariante par la translation $\widehat T^h_{(1,0)}$  définit de la manière suivante (pas complètement triviale) une distribution 
$$ \bar{D} : \mathcal{C}^{\infty}(S^1) \longrightarrow \R.$$ 
En effet toute fonction $\mathbb{Z}$-périodique $\bar{f}$ (\textit{i.e.} une fonction sur $S^1$) s'écrit 
$$ x \mapsto \sum_{k\in \Z}{f(x+k)} $$ 
pour une certaine fonction $f \in \mathrm{S}(\R)$ (cf. Lemme \ref{lemme:periodise}). On définit 
$$ \bar{D}(\bar{f}) := D(f).$$ 
Il suffit de vérifier les choses suivantes (ce qu'on laisse en exercice) :
\begin{itemize}
\item La valeur de  $\bar{D}(\bar{f})$ ne dépend pas du choix de $f$ tel que $\bar{f}(x) = \sum_{k\in \Z}{f(x+k)}$.

\item La distribution ainsi définie est continue pour la topologie d'espace de Fréchet de $\mathcal{C}^{\infty}(S^1)$.

\end{itemize}

Une telle distribution est complètement déterminée par ses coefficients de Fourier 
$$ d_n := \bar{D}(x \mapsto e^{\frac{2i\pi n x}h}) $$ 
car les $(x \mapsto e^{\frac{2i\pi n x}h})_{n \in \Z}$ forment une base hilbertienne de $\mathrm{L}^2(S^1)$ et la topologie de $\mathcal{C}^{\infty}(S^1)$ est plus fine que celle induite par $\mathrm{L}^2(S^1)$.  Si on suppose de surcroît que $D$ est invariante par $\widehat T^h_{(0,1)}$, ça se traduit en 
$$ \forall n \in \mathbb{Z}, d_{n + N} = d_n.$$ 
On a donc démontré la chose suivante:

\begin{proposition}

Pour $h= \frac 1 N$, l'espace vectoriel $\mathcal{H}^N$ est canoniquement isomorphe à l'espace des distributions de $S^1$ dont les coefficients de Fourier satisfont $ \forall n \in \mathbb{Z}, d_{n + N} = d_n.$ En particulier il est de dimension $\dim(\mathcal H^N) = N$.

\end{proposition}

\subsection{Symétrisation}

\label{subsec:symetrisation}

Si $\varphi$ est une fonction de Schwartz dans $\mathcal{S}(\R)$, on peut lui associer l'élément suivant de $\mathcal{H}^N$, qu'on note $\Sigma^h(\varphi)$:
$$ 
\Sigma^h(\varphi)=\sum_{(k_1,k_2) \in \mathbb{Z}^2} \widehat T^h_{(k_1,k_2)}(\varphi) \in \mathcal{H}^N
$$ 
qui est défini de la manière suivante: pour toute fonction $f \in \mathcal S(\R)$, %\footnote{LB: question notation: on met des crochets pour l'évaluation des distributions?}
$$
\langle \Sigma^h(\varphi) , f \rangle := \sum_{(k_1,k_2) \in \mathbb{Z}^2}{\langle \widehat T^h_{(k_1,k_2)}(\varphi), f \rangle}.
$$ 
Il est clair qu'une telle distribution est invariante par $\widehat T^h_{(1,0)}$ et $\widehat T^h_{(0,1)}$, si on peut montrer qu'elle est bien définie, c'est à dire que la somme converge. Ceci est fait dans l'annexe \ref{ann:demos}. 

En fait, la proposition suivante dit précisément que toute distribution dans $\mathcal H^N$ s'écrit comme le moyenné $\Sigma^h(\varphi)$ d'une fonction $\varphi$ de $\mathcal S(\R)$.
\begin{proposition}
\label{prop:symsurjective}
L'application linéaire "moyenne" 
$$\mathcal S(\R) \to \mathcal H^N$$
$$\varphi \mapsto \Sigma^h(\varphi)$$
est surjective.
\end{proposition}
\begin{proof}
Voir dans l'annexe, proposition \ref{prop:surjective}.
\end{proof}

\subsection{Structure hermitienne sur $\mathcal{H}^N$}

On explique maintenant comment définir un produit scalaire sur $\mathcal{H}^N$. 
Si $D_1$ et $D_2$ sont deux éléments de $\mathcal{H}^N$ tels que $D_1 = \Sigma^h(\varphi_1)$ et $D_2 = \Sigma^h(\varphi_2)$ pour deux fonctions  $\varphi_1$, $\varphi_2 \in \mathcal{S}(\R)$, on définit 
$$ 
\langle(\Sigma^h(\varphi_1), \Sigma^h(\varphi_2) \rangle_{\mathcal{H}^N} :=  \sum_{(k_1,k_2) \in \mathbb{Z}^2}{\langle \widehat{T}^h_{(k_1,k_2)}(\varphi_1), \varphi_2 \rangle}  = D_1(\varphi_2).
$$ 
Le membre de droite montre que la définition ne dépend pas du choix d'un $\varphi_1$ tel que $D_1 = \Sigma^h(\varphi_1)$. Comme $\langle \widehat{T}^h_{(k_1,k_2)}(\varphi), \psi \rangle = \overline{\langle \varphi, \widehat{T}^h_{-(k_1,k_2)}(\psi)\rangle}$, on trouve que $\langle(\Sigma^h(\varphi_1), \Sigma^h(\varphi_2) \rangle_{\mathcal{H}^N} = \overline{D_2(\varphi_1)}$ ce qui montre que d'une part le produit $\langle(\Sigma^h(\varphi_1), \Sigma^h(\varphi_2) \rangle_{\mathcal{H}^N}$ ne dépend pas du choix d'un $\varphi_2$ tel que $D_2 = \Sigma^h(\varphi_2)$ et que $\langle \cdot, \cdot \rangle_{\mathcal{H}^N}$ est bien hermitien.

\subsection{Action de $\widehat{M}^h$ induite sur $\mathcal{H}^h$}	
	
On montre maintenant que l'action de $\widehat{M}^h$ sur $\mathrm{L}^2(\R)$ induit une action sur $\mathcal{H}^N = \mathcal{H}^h$ pourvu que $M$ appartienne à $\mathrm{SL}(2, \mathbb{Z})$. De la même manière qu'on a défini le produit hermitien sur $\mathcal{H}^h$, on définit pour $D = \Sigma^h(\varphi) \in \mathcal{H}^h$ 
$$ 
\widehat{M}^h\cdot D := \Sigma^h \big(\widehat{M}^h(\varphi) \big).
$$ 
Il nous faut juste vérifier que ceci est bien défini, c'est à dire que le membre de droite de l'égalité ci-dessus ne dépend pas du choix de $\varphi$ tel que $D = \Sigma^h(\varphi)$. Pour une fonction test $f \in \mathcal{S}(\R)$
$$
 \Sigma^h \big(\widehat{M}^h(\varphi) \big)(f) = \sum_{(k_1,k_2) \in \mathbb{Z}^2}{\langle \widehat{T}^h_{(k_1,k_2)}(\widehat{M}^h(\varphi)), f \rangle} 
 $$ 
mais comme $\widehat{T^h}_{(k_1,k_2)} \circ \widehat{M^h} = \widehat{M^h} \circ \widehat{T^h}_{M^{-1}(k_1,k_2)}$ on obtient 
$$ 
\Sigma^h \big(\widehat{M}^h(\varphi) \big)(f) = \sum_{(k_1,k_2) \in \mathbb{Z}^2}{\langle \widehat{T}^h_{(k_1,k_2)}(\varphi), (\widehat{M}^h)^{-1} f \rangle} = D((\widehat{M}^h)^{-1} f )
$$
car comme $M \in \mathrm{SL}_2(\mathbb{Z})$, $M$ permute les vecteurs à coordonnées entières (c'est ici qu'on utilise que $M \in \mathrm{SL}_2(\Z)$). On a donc $(\widehat{M}^h\cdot D) (f) = D((\widehat{M}^h)^{-1} f)$, le membre de droite ne dépend pas du choix de $\varphi$.

\subsection{Unitarité de l'action de $\widehat{M}^h$ sur $\mathcal{H}^h$.}

Finalement, on vérifie la proposition suivante.

\begin{proposition}
 $\widehat{M}^h$ agit sur $\mathcal{H}^h$ en préservant la structure hermitienne définie par $\langle\cdot, \cdot \rangle_{\mathcal{H}^N}$.
\end{proposition}

\begin{proof}

On calcule. Soit $D_1$ et $D_2$ deux distributions de $\mathcal{H}^h$ qu'on réalise comme des symétrisées de fonctions $\varphi_1, \varphi_2 \in \mathcal{S}(\R)$. On a 
\begin{multline*}
 \langle \widehat{M}^h(D_1), \widehat{M}^h(D_2) \rangle_{\mathcal{H}^N}   = \\\langle \widehat{M}^h(\Sigma^h(\varphi_1)), \widehat{M}^h(\Sigma^h(\varphi_2)) \rangle_{\mathcal{H}^N} = 
 \langle(\Sigma^h(\widehat{M}^h(\varphi_1)), \Sigma^h(\widehat{M}^h(\varphi_2)) \rangle_{\mathcal{H}^N} 
 \end{multline*}

ce qui donne 
$$ 
\langle \widehat{M}^h(D_1), \widehat{M}^h(D_2) \rangle_{\mathcal{H}^N}    =\sum_{(k_1,k_2) \in \mathbb{Z}^2}{\langle \widehat{T}^h_{(k_1,k_2)}(\widehat{M}^h(\varphi_1)), \widehat{M}^h(\varphi_2)\rangle}.
$$
Comme on a $\widehat{T}^h_{(k_1,k_2)}(\widehat{M}^h(\varphi_1)) = \widehat{M}^h \circ \widehat{T}^h_{M^{-1}(k_1,k_2)}(\varphi_1)$ on obtient 
$$ 
\langle \widehat{M}^h(D_1), \widehat{M}^h(D_2) \rangle_{\mathcal{H}^N}    =\sum_{(k_1,k_2) \in \mathbb{Z}^2}{\langle \widehat{M}^h(\widehat{T}^h_{M^{-1}(k_1,k_2)}(\varphi_1)), \widehat{M}^h(\varphi_2)\rangle}.
$$ 
Mais $\widehat{M}^h$ est unitaire pour le produit scalaire de $\mathrm{L}^2(\R)$ donc 
$$\langle \widehat{M}^h(\widehat{T}^h_{M^{-1}(k_1,k_2)}(\varphi_1)), \widehat{M}^h(\varphi_2)\rangle = \langle \widehat{T}^h_{M^{-1}(k_1,k_2)}(\varphi_1), \varphi_2 \rangle.$$
Ce qui montre 
$$  
\langle \widehat{M}^h(D_1), \widehat{M}^h(D_2) \rangle_{\mathcal{H}^N} =  \sum_{(k_1,k_2) \in \mathbb{Z}^2}{\langle \widehat{T}^h_{M^{-1}(k_1,k_2)}(\varphi_1), \varphi_2 \rangle} = \langle D_1, D_2 \rangle_{\mathcal{H}^N} 
$$ 
ce qui conclut.
\end{proof}

\section{Notions microlocales}

\label{sec:microlocal}

Dans ce paragraphe, on introduit un certain nombre de notions (paquets d'ondes, mesure de Husimi, temps d'Ehrenfest) qui permettent d'exprimer formellement le lien entre la dynamique classique et la dynamique quantique, sous la forme du théorème dit d'Egorov (qu'on préférerait probablement appeler "théorème de propagation des paquets d'ondes).

\subsection{Paquets d'ondes}

\begin{definition}[Paquets d'ondes]
Si $q_0 \in \RR$ et $p_0 \in \RR$ on note $\Phi_{q_0,p_0}^h$ le $h$-paquet d'ondes centré en $(q_0, p_0) \in T \RR$ la fonction 
	$$ \fonction{\Phi_{q_0,p_0}^h}{\RR}{\RR}{x}{ C_h e^{ \frac{2i \pi p_0 (x-q_0)}{h}}  e^{ -\frac{\pi (x - q_0)^2}{h}}} \ , $$
où $C_h$ est la constante qui donne à $\Phi_{q_0,p_0}^h$ norme unité dans $L^2(\RR)$. On peut facilement voir que $C_h = C_1 h^{-\frac{1}{4}}$. On notera plus simplement $f^h$ le paquet d'ondes centré en $(0,0)$.
\end{definition}

On démontre aussi la proposition suivante dont on aura besoin à plusieurs reprises. Elle dit en substance que le translaté quantique d'un paquet d'ondes en un point est, à un terme de phase près qu'on calcule explicitement, le paquet d'ondes en le point translaté classiquement.

\begin{proposition}
\label{prop:transPO}
Soit $(a,b) \in \R^2$ et $(q,p) \in \R^2$. On a l'égalité suivante

$$ \widehat{T}^h_{(a,b)}(\Phi_{q,p}^h) = e^{-\frac{i\pi}{h}ab} e^{\frac{2i \pi}{h}bq} \cdot \Phi_{q+a,p+b}^h.$$

\end{proposition}

\begin{proof}
C'est une conséquence directe de la Proposition \ref{prop:commutation} (1), on a pour toute fonction $\Phi$ 
$$
 \widehat{T}^h_{(a,b)}(\Phi)(x) = e^{-\frac{i \pi}{h}ab} e^{\frac{2i \pi}{h}bx} \Phi(x-a) $$

\end{proof}

\subsection{Mesures de Husimi}
\label{sec:husimi}
Si $\Phi$ est une fonction quelconque de $\mathrm{L}^2(\R)$ on définit sa mesure de Husimi de la manière suivante.

\begin{definition}[Mesure de Husimi sur $\R^2$]

La $h$-mesure de Husimi de $\Phi$ est 
	$$ \mathrm{Hus}^h_{\Phi} := \frac{1}{h} \left\lvert\scal{\Phi}{ \Phi^h_{(q,p)}}\right\rvert^2 dqdp $$ c'est à dire la mesure absolument continue par rapport à la mesure de Lebesgue dont la densité est $\frac{1}{h} \left\lvert\scal{\Phi}{ \Phi^h_{(q,p)}}\right\rvert^2$.

\end{definition}

\noindent On définit de manière analogue la mesure de Husimi sur $\mathbb{T}^2$ d'un élément de $\mathcal{H}^h = \mathcal{H}^N$.

\begin{definition}[Mesure de Husimi sur $\mathbb{T}^2$]

La $h$-mesure de Husimi de $D \in \mathcal{H}^h$ est 
	$$ \mathrm{Hus}^h_{\Phi} := \frac{1}{h} \lvert D(\Phi^h_{(q,p)})\rvert^2 dqdp $$ c'est à dire la mesure absolument continue par rapport à la mesure de Lebesgue de $\mathbb{T}^2$ dont la densité est $\frac{1}{h} \lvert D(\Phi^h_{(q,p)})\rvert^2$.

\end{definition}

\noindent Il n'est pas complètement évident que cette mesure est bien définie car on a utilisé $\Phi^h_{(q,p)}$ dans la définition où $(q,p) \in \R^2$. Il nous suffit donc de vérifier que $\frac{1}{h} |D(\Phi^h_{(q,p)})|^2$ ne dépend pas du choix d'un représentant de $(q,p) \in \mathbb{T}^2$. Ceci est dû au fait que 

$$\Phi^h_{(q+k_1,p+k_2)} = a \cdot  \widehat{T}^h_{(k_1,k_2)}(\Phi^h_{(q,p)}) $$ avec $a$ un nombre complexe de module $1$. Comme $D$ est invariante par l'action des $ \widehat{T}^h_{(k_1,k_2)}$ on a

$$ \frac{1}{h} |D(\Phi^h_{(q,p)})|^2 = \frac{1}{h} |D(\Phi^h_{(q+k_1,p+k_2)})|^2$$ pour tout $(k_1,k_2) \in \mathbb{Z}^2$. 

\subsection{Théorème d'Egorov et temps d'Ehrenfest}	
	
On a déjà vu que la dynamique classique jouait un rôle dans l'évolution de fonction d'onde par l'équation de Schrödinger,  via l'expression de son propagateur. Ici, on donne un énoncé, utilisant les mesures de Husimi, qui rend ce lien plus direct. 

\medskip On considère $\Phi^h_{(q_0,p_0)}$ un $h$-paquet d'ondes en $(q_0,p_0)$. Sa mesure de Husimi est proche d'une masse de Dirac et on a le résultat 
%facile 
suivant
%  dont on laisse la démonstration en exercice.

\begin{proposition}

La suite de mesure $\mathrm{Hus}^h_{\Phi^h_{(q_0,p_0)}}$ converge faiblement vers une masse de Dirac en $(q_0,p_0)$ quand $h$ tend vers $0$.

\end{proposition}

On laisse la démonstration de cette proposition en exercice pour le lecteur curieux d'apprendre les bases de la théorie. C'est une conséquence assez facile de calculs d'intégrale, utilisant la formule de l'intégrale gaussienne.  
	
\vspace{2mm}

Le théorème d'Egorov, qu'on énonce ci-après, prédit une forme d'équivariance pour les mesures de Husimi sous les actions respectives du propagateur de l'équation de Schrödinger et de l'action de la dynamique classique.

\begin{theoreme}[Egorov]

Soit $n \in \mathbb{N}$. La suite de mesures de Husimi des fonctions $(\widehat{M^h})^n \cdot\Phi^h_{(q_0,p_0)}$ converge faiblement vers la masse de Dirac en $M^n(q_0,p_0)$ quand $h$ tend vers $0$.

\end{theoreme}

Ce théorème nous dit en substance que pour $n$ très grand, pourvu que $h$ soit suffisamment petit, $(\widehat{M^h})^n \cdot\Phi^h_{(q_0,p_0)}$ ressemble à un paquet d'ondes centré en l'image classique de $(q_0,p_0)$. On peut se demander, à $h$ fixé à quelle échelle de temps $n$ cette description est valide. La réponse est la suivante : cette description fonctionne pour des temps $n$ d'ordre inférieur à 
$$ t_E(h) = \frac{\lvert\log h\rvert}{2 \log \lambda},$$ 
où l'on rappelle que $\lambda$ est la valeur propre de $M$ plus grande que 1.

Plus précisément 

\begin{theoreme}[Egorov avec temps d'Ehrenfest]

Soit $(n_h)$ une suite d'entiers telle que $ \frac{n_h}{t_E(h)} < 1 - \epsilon$ pour un $\epsilon > 0$ fixé. La distance entre la mesure de de Husimi de la fonction $(\widehat{M^h})^{n_h} \cdot\Phi^h_{(q_0,p_0)}$ et la masse de Dirac en $M^{n_h}(q_0,p_0)$ tend vers $0$ quand $h$ tend vers $0$.

\end{theoreme} 	

Le constante $t_E(h) = \frac{\lvert\log h\rvert}{2 \log \lambda}$ est appelée \textit{temps d'Ehrenfest}. Une fois de plus, on laisse la démonstration de ces théorèmes en exercice. Il faut cette fois, pour mener à bien ces démonstrations, utiliser les formules pour le propagé d'un paquet d'ondes démontrées dans le paragraphe \ref{sec:approx}.

\section[Approximation par un état lagrangien]{Approximation du propagé d'un paquet d'ondes par un état lagrangien.}

\label{sec:approx}
Dans cette section on démontre que le propagé d'un paquet d'ondes par l'application du chat quantique peut être remplacé par une fonction lagrangienne amortie le long de la variété stable, au coût d'une erreur négligeable à toutes fins pratiques 

\begin{itemize}

\item pour tout temps $n$ plus grand que le temps d'Ehrenfest;

\item avec une erreur uniforme en le temps $n$, qui tend vers $0$ quand $h = \frac{1}{N}$ tends vers $0$.

\end{itemize}

\subsection{Avant passage au quotient}

La première étape consiste à travailler avant le passage au quotient, pour le propagateur agissant sur $\mathrm{L}^2(\mathbb{R})$. On rappelle qu'on note $\widehat{M}^h : \mathrm{L}^2(\R) \longrightarrow \mathrm{L}^2(\R) $ l'opérateur linéaire obtenu par quantification de l'application linéaire $$\begin{array}{ccc}
\R^2 & \longrightarrow &  \R^2 \\
X & \longmapsto & \begin{pmatrix}
2 & 1 \\
1 & 1
\end{pmatrix} X
\end{array}.$$ Pour simplifier les notations, on notera $f^h := \Phi^h_{(0,0)}$ le paquet d'ondes en $(0,0)$. Nous démontrons dans ce contexte une estimation ponctuelle pour la différence entre le propagé d'un paquet d'ondes en $(0,0)$ et une fonction lagrangienne amortie associée à la droite d'équation $\xi = \tan \theta \cdot x$ (qui est la variété instable en $(0,0)$).\\

La proposition suivante est l'étape clé en ce qui concerne ce résultat de propagation. Il montre que le propagé d'un paquet d'onde est proche de l'état lagrangien modulo un terme dont on montrera qu'il est négligeable.

\fbox{
\parbox{\textwidth}{
\begin{proposition}
\label{prop:approx}
Il existe une fonction $\chi : \R \longrightarrow \CC$ analytique à décroissance gaussienne et une constante $C_1 > 0$ telle que pour tout $x \in \mathbb{R}$ 
$$  
|\widehat{M}^h(f^h)(x) - \chi\big(\frac{x}{h^{\frac{1}{2}} \lambda^n}\big) \, e^{2 i \pi \tan \theta \frac{x^2}{2h}} |\leq |\widehat{M}^h(f^h)(x)  \cdot (1 - e^{-\frac{x^2}{h}R_n})| 
$$ 
pour tout $n \in \mathbb{N}$, où $R_n$ est une constante positive satisfaisant $R_n \leq C_1 \lambda^{-4n}$.
\end{proposition}
}
}

Quantiquement, cela suggère que l'image par un grand nombre d'itérés de $\widehat{M}^h$ du paquet d'ondes $f^h : \RR \to \RR$ devrait être proche d'un état lagrangien associé à la droite instable du plan, c'est à dire quelque chose comme 
\begin{equation}
\label{eq:etatlagr} \fonctionbis{\RR}{\RR}{y}{ \chi(x) e^{\frac{ 2 i \pi (\tan \theta) x^2}{ 2 h} }} \ .
\end{equation}

où $\chi$ est une fonction qui dépend de la position et de la vitesse du paquet d'ondes, du nombre d'itérés de $\widehat{M}^h$ mais pas de $h$. Notons qu'une fonction du type 
	$$ e^{\frac{ 2 i \pi f}{ h}}$$
microlocalise sur le graphe du gradient de $f$; si $f = \frac{(\tan \theta) x^2}{2}$ on retrouve bien que  
 $$ e^{\frac{ i \pi (\tan \theta) x^2}{ h}} $$
 microlocalise sur\footnote{par "microlocalise sur" on entend "le support des mesures de Husimi tend vers"}  la droite d'équation 
$ y = (\tan \theta) \,x $ quand $h$ tend vers $0$.

Rappelons que si $$f^h(x) := e^{-\pi x^2/h}$$ alors $\mathcal{F}^h(f^h) = f^h$. En particulier, on a d'après la proposition \ref{prop:propagateurtempt}
\begin{equation}
\label{eq:proppo}
\widehat{M}^n(f^h)(x) = C(h,n) e^{\frac{i\pi}{h}\frac{c_n}{a_n}x^2} \int_\R  e^{\frac {i\pi}{h} \left(\frac {b_n}{a_n} \xi^2 + \frac 2 {a_n} x\xi\right)} e^{-\pi \xi^2/h}  \  d\xi
\end{equation}
où $C(h,n)$ est la constante qui fait de $\widehat{M}^n(f^h)$ une fonction de norme $L^2$ unité et où
	$$ M^n := \begin{pmatrix}
	a_n & b_n \\ c_n & d_n 
	\end{pmatrix}.$$
	
\medskip La valeur de la constante $C(h,n)$ aura une importance par la suite, on fera donc bien attention à la garder dans les formules qu'on donnera.

Dans le cas où $M$ est l'application du chat, on a 
		\begin{align*}
				 a_n & =  \cos^2(\theta) \lambda^n + \sin^2(\theta) \lambda^{-n} \\
				 b_n & = c_n = \cos(\theta) \sin(\theta)( \lambda^n - \lambda^{-n} )
		\end{align*}
pour $\lambda$ la valeur propre de $M$ de module plus grand que $1$.

Le but de cette section est d'approximer le propagé de la gaussienne $\widehat{M}^n(f^h)(x)$ par une fonction de la forme \eqref{eq:etatlagr}. Plus précisément, on pose
	$$ \mathcal{L}^h(n)(x) :=  C(h,n) e^{ \frac{i \pi \tan(\theta) x^2}{h}} \cdot e^{- \frac{ \beta \pi x^2 }{h  \lambda^{2n}}} \ , $$
avec $$ \beta := \frac{1 + 2i \tan(\theta)}{\cos^2(\theta)} $$
et $C(h,n)$ la constante qui fait de $\widehat{M}^n(f^h)$ une fonction de norme $L^2$ unité. 
\medskip

On pose maintenant
	$$ \gamma_n := \frac{b_n}{a_n} \ . $$
et on réécrit \eqref{eq:proppo} :
\begin{align*}
 \widehat{M^n}(f^h)(x) & = C(h,n) e^{\frac{i\pi}{h} \gamma_n x^2} \int_\R  e^{\frac{i\pi}{h} \left( \gamma_n \xi^2 + \frac 2 {a_n} x\xi\right)} e^{-\pi \xi^2/h} \, d \xi \\
 & = C(h,n) e^{\frac{i\pi}{h} \gamma_n x^2} \int_\R  e^{ \frac{- \pi \xi^2 }{h}( 1 - i \gamma_n)} e^{ \frac{ 2 i x \xi}{h a_n} }  \, d\xi \\
 & = C(h,n) e^{\frac{i\pi}{h} \gamma_n x^2} \mathcal{F}^{-h} \left( \xi \mapsto e^{ \frac{- \pi \xi^2 }{h}( 1 - i \gamma_n)} \right) \left( \frac{x}{a_n} \right) \\
  & = C(h,n) \exp{ \left( \frac{i\pi}{h} \gamma_n x^2 \right)} \exp{ \left(  \frac{- \pi x^2 }{ ha_n^2}( 1 - i \gamma_n)^{-1} \right)} \\
  & = C(h,n) \exp{ \left(  \frac{- \pi x^2 }{ h} \left( \frac{1}{a_n^2 ( 1 - i \gamma_n)} - i \gamma_n  \right) \right)}
\end{align*} où on a utilisé la formule de la transformation en ondes planes d'une fonction gaussienne pour obtenir $\mathcal{F}^{-h} \left( \xi \mapsto e^{ \frac{- \pi \xi^2 }{h}( 1 - i \gamma_n)} \right) \left( x \right) = \exp{ \left(  \frac{- \pi x^2 }{ h}( 1 - i \gamma_n)^{-1} \right)}$.

Rappelons que
\begin{align*}
				 a_n & =  \cos^2(\theta) \lambda^n + \sin^2(\theta) \lambda^{-n} \\
				 b_n & = \cos(\theta) \sin(\theta)( \lambda^n - \lambda^{-n} ) \\
				 \gamma_n & = \frac{b_n}{a_n}
		\end{align*}

En particulier
	\begin{align*}
		 a_n^{-2} & = \frac{1}{\cos^2(\theta) \lambda^{2n}} + \mathcal{O}(\lambda^{-4n}) \\
		 \gamma_n & = \tan(\theta) - \frac{\tan(\theta)}{\cos^2(\theta)} \lambda^{-2n} + \mathcal{O}(\lambda^{-4n}) \ . 
	\end{align*}

où l'on note par $\mathcal{O}(\lambda^{-4n})$ une fonction de $n$ dont le module est au pire de l'ordre de $\lambda^{-4n}$. On obtient alors
	$$ (1 - i \gamma_n)^{-1} a_n^{-2} = \frac{1 + i \tan(\theta)}{\cos^2(\theta)} \lambda^{-2n} + \mathcal{O}(\lambda^{-4 n}) \ , $$
et donc 
 $$ \frac{1}{ a_n^2 (1 - i \gamma_n)} - i \gamma_n = -i  \tan(\theta) + \frac{1 + 2 i \tan(\theta)}{\cos^2(\theta)} \lambda^{-2n} + \mathcal{O}(\lambda^{-4 n}) \ , $$
Rappelons que l'on avait posé $\beta := \frac{1 + 2 i \tan(\theta)}{\cos^2(\theta)}$. On se retrouve finalement avec 
\begin{equation}
\label{eq:L2approx}
 \widehat{M}^n(f^h)(x) = C(h,n) \exp \left(  \frac{ i \pi \tan(\theta) x^2 }{ h} \right)   \exp \left( \frac{ - \pi x^2}{h} \left( \beta \lambda^{-2n} + \mathcal{O}(\lambda^{-4 n})  \right) \right) \, . 
 \end{equation}

Notons en particulier que la formule ci-dessus montre que 
	$$ C(h,n) \equivaut{n \to \infty} \left(\inte{\RR} \exp \left( \frac{ -2  \pi x^2}{h \cos^2(\theta)} \lambda^{-2 n} \right) dx\right)^{-\frac{1}{2}} = \mathrm{constante} \cdot \frac{1}{h^{\frac{1}{4}} \sqrt{\lambda^n}}.$$  
uniformément en $h$.

\subsection{Passage au quotient, échauffement formel}

Dans le paragraphe précédent, on a montré que l'évolué d'un paquet d'ondes en $(0,0)$ au temps $n$ est ponctuellement très proche de la fonction lagrangienne 
$$\mathcal{L}^h(n) := x \mapsto C(h,n) e^{ \frac{i \pi \tan(\theta) x^2}{h}} \cdot e^{- \frac{ \beta \pi x^2 }{h  \lambda^{2n}}}.$$ On voudrait montrer que ça implique que leurs projetés au quotient respectifs dans $\mathcal{H}^h$ sont également proche d'une certaine manière. Par "projeté au quotient" d'une fonction de $\mathrm{L}^2(\R)$ on entend le symétrisé comme défini dans le paragraphe \ref{subsec:symetrisation}.

\vspace{3mm}

\fbox{
\parbox{\textwidth}{
Il faut faire attention, il n'est en général pas vrai que si deux fonctions $f$ et $g \in \mathrm{L}^2(\mathbb{R})$ sont proches, leurs symétrisés $\Sigma^h(f)$ et $\Sigma^h(g)$ (si ils sont bien définis) sont aussi proches l'un de l'autre. Cela impliquerait la continuité de $\Sigma^h$ pour la norme $\mathrm{L}^2$ et rendrait la construction de l'application $\Sigma$, qui utilise les espaces de fonctions de Schwartz, artificiellement compliquée ! }
}

\vspace{3mm} On va donc devoir se salir les mains et utiliser les estimations ponctuelles qu'on a démontré précédemment. On rappelle que par définition le symétrisé d'une fonction de Schwartz $g$ est la distribution définie par 

$$ \begin{array}{ccccc}
\Sigma^h(g) &  : &  \mathcal{S}(\mathbb{R}) & \longrightarrow & \mathcal{S}(\mathbb{R}) \\
 & & \varphi &  \longmapsto & \sum_{k_1,k_2 \in \mathbb{Z}}{\langle g, \widehat{T^h}_{(k_1,k_2)}\varphi \rangle}
\end{array}.$$

Il y a plusieurs manières d'exprimer le fait que les symétrisés respectifs de $(\hat{M}^h)(f^h)^n$ et $\mathcal{L}^h(n)$ sont proches : on peut calculer la norme de leur différence dans $\mathcal{H}^N$ muni de son produit scalaire, ou alors montrer que leurs décompositions en paquets d'onde respectives sont proches. On choisit la seconde option. 

\vspace{2mm} On rappelle que si $\mathcal{D}$ est une distribution de $\mathcal{H}^h$, sa décomposition en paquets d'ondes est la fonction 
	$$ (q,p) \longmapsto \mathcal{D}(\Phi^h_{(q,p)}).$$ 
	
On veut donc estimer la différence entre les deux termes suivants, qui sont les décompositions en paquets d'ondes respectives de  $(\widehat{M}^h)^n(f_h)$ et $\mathcal{L}^h(n)$

\begin{itemize}

\item $\somme{ (k_1,k_2) \in \ZZ^2} \scal{(\widehat{M}^h)^n(f_h)}{ \Phi^h_{(q +k_1,p+k_2)}}$

\item $ \somme{ (k_1,k_2) \in \ZZ^2} \scal{\mathcal{L}^h(n)}{ \Phi^h_{(q +k_1,p+k_2)}}.$

\end{itemize}

On commence par montrer que de tous les termes de la sommes, seuls ceux correspondants à des points $(k_1,k_2)$ près de la droite instable d'équation $\xi =~\tan \theta \cdot x$ contribuent de manière non-négligeable.

\subsection{Passage au quotient, termes loin de la droite instable.}

On note $\mathcal{D}_{\theta}$ le $(\frac{\cos(\theta)}{2})$-voisinage de la droite d'équation $\xi =  \tan(\theta) \cdot x$. La largeur $\frac{\cos(\theta)}{2}$ est choisie telle qu'elle de telle sorte à ce que  pour tout $q \in \ZZ$, il existe un unique $p \in \ZZ$ tel que $(q,p) \in \mathcal{D}_{\theta}$. La proposition suivante nous dit en substance que lorsqu'on calcule la décomposition en paquets d'onde des fonctions $\mathcal{L}^h(n)$ et $\widehat{M^n}(f_h)$, seuls les termes proches de la droite instable comptent. Ca ne devrait pas paraître surprenant, les fonctions $\mathcal{L}^h(n)$ et $\widehat{M^n}(f_h)$ ressemblent beaucoup à la fonction lagrangienne associée à la droite $\xi =  \tan(\theta) \cdot x$, qui est une lagrangienne de $\mathrm{T}^* \mathbb{R}$.

\begin{proposition}
\label{prop:poubelle1}
Pour tout $k > 0$ entier, il existe une constante $D_k > 0$ telle que pour tout  $(q_0, p_0)$ dans $\TT^2$  et pour tout $n \in \mathbb{N}$ alors
$$\big| \sum_{\substack{(q,p) \in (q_0,p_0) + \ZZ^2  \\
 (q, p) \notin \mathcal{D}_{\theta}}} \scal{ \widehat{M^n}(f_h)}{ \Phi^h_{(q ,p)}}\big| \leq D_k h^k$$
%
%$$ \big\vert \somme{ \begin{split}
% (q,p) \in (q_0,p_0) + \ZZ^2  \\
% (q, p) \notin \mathcal{D}_{\theta}\end{split}} \scal{ \widehat{M^n}(f_h)}{ \Phi^h_{(q ,p)}} \big\vert  \leq D_k h^k $$ et 
 $$\big| \sum_{\substack{(q,p) \in (q_0,p_0) + \ZZ^2  \\
 (q, p) \notin \mathcal{D}_{\theta}}} \scal{\mathcal{L}^h(n)}{ \Phi^h_{(q ,p)}}\big| \leq D_k h^k$$

%$$ \Big|  \somme{ \begin{split}
% (q,p) \in (q_0,p_0) + \ZZ^2  \\
% (q, p) \notin \mathcal{D}_{\theta} \end{split}} \scal{ \mathcal{L}^h(n)}{ \Phi^h_{(q ,p)}} \Big|  \leq D_k h^k $$ 
 
\end{proposition}

\textbf{Démonstration.} On ne traite que le cas de $\mathcal{L}^h(n)$, celui de $\widehat{M^n}(f_h)$ se traite exactement de la même manière.

\vspace{3mm}

\noindent On va montrer que pour tout $(q,p) \notin \mathcal{D}_{\theta}$ 
\begin{equation}
\label{eq:ineq}
\lvert\scal{\mathcal{L}^h(n)}{ \Phi^h_{(q ,p)}}\rvert \le A_k(q,p)h^k
\end{equation} 
 où $A_k(q,p)$ est pour tout $k \in \mathbb{N}$ sommable sur l'ensemble précédent. \\
 
Ici le produit scalaire est explicite :
	$$ \scal{\mathcal{L}^h(n)}{ \Phi^h_{(q ,p)}} = C(h,n)  \inte{\RR} e^{ \frac{i \pi \tan(\theta) x^2}{h}} \cdot e^{- \frac{ \beta \pi x^2 }{h  \lambda^{2n}}} e^{ \frac{- 2i \pi p (x-q)}{h}} e^{\frac{- \pi (x-q)^2}{h}} \ dx \ . $$
En posant $\varphi(x) := \tan(\theta) x^2 - 2 p (x-q) $, on obtient
	$$ \scal{\mathcal{L}^h(n)}{ \Phi^h_{(q ,p)}} = C(h,n)  \inte{\RR} e^{ \frac{2 i \pi \varphi(x)}{h}} \cdot e^{\frac{- \pi (x-q)^2}{h}}  \cdot e^{- \frac{ \beta \pi x^2 }{h  \lambda^{2n}}} dx \ . $$
	
La partie la plus fortement gaussienne de l'intégrale devrait localiser au voisinage du point $q$. La partie oscillante de l'intégrale, l'intégrale à la phase complexe va se concentrer sur un voisinage du point critique de $\varphi$, c'est à dire le point 
	$$ x_c := \frac{p}{\tan(\theta)} $$
On devrait donc avoir une intégrale qui est petite si $q$ est loin de $x_c$, c'est à dire si $(q,p)$ est loin de la droite $\mathcal{D}_{\theta}$. Voyons comment en pratique montrer l'inégalité \eqref{eq:ineq}. Supposons donc que
	$$ (q,p) \notin \mathcal{D}_{\theta} $$
c'est à dire que 
	$$ |q \sin(\theta) - p \cos(\theta)| \ge \cos(\theta)/2$$ 
ce qui donne bien 
	$$ | q - x_c | \ge \tan(\theta)/2  \ .$$
%On notera 
%	$$ m := \frac{x_c + q}{2} \ . $$ 

On se donne alors une partition de l'unité $(\chi_c, \chi_q, \chi_0)$ supportée au voisinage de $x_c$ et $q$, c'est à dire trois fonctions lisses positives telles que 
	\begin{align*}
		 \supp{\chi_c} & \subset \left] x_c - \frac{\tan(\theta)}{6}, x_c + \frac{\tan(\theta)}{6} \right[ \\
		  \supp{\chi_q} & \subset  \left] q - \frac{\tan(\theta)}{6}, q + \frac{\tan(\theta)}{6} \right[ \\
		  \chi_q + \chi_c & + \chi_0 = 1 
	\end{align*}

On réécrit alors 
$$ \scal{\mathcal{L}^h(n)}{ \Phi^h_{(q ,p)}} = I_c + I_q + I_0 \ , $$
chacun des indices étant en correspondance avec la partition de l'unité. On traite ces trois intégrales séparément. Commençons par $I_c$, la plus facile, 
\begin{align*} 
|I_c| & \le C(h,n) \inte{\supp{\chi_c}} e^{\frac{- \pi (x-q)^2}{h}}  \cdot e^{- \frac{ \Re(\beta) \pi x^2 }{h  \lambda^{2n}}} dx \\
 & \le C(h,n) \overset{ x_c + \frac{\tan(\theta)}{6}}{\inte{x_c - \frac{\tan(\theta)}{6}}} e^{\frac{- \pi (x-q)^2}{h}}  \cdot e^{- \frac{ \Re(\beta) \pi x^2 }{h  \lambda^{2n}}} dx  \\
 & \le C(h,n) \cdot C_2 \cdot e^{\frac{- \pi |x_c-q|^2}{2 h}} \cdot e^{- \frac{C_3 \cdot x_c^2}{h \lambda^{2n}}} \ ,
\end{align*}
pour deux constantes positives $C_2$ et $C_3$. La série venant de \eqref{eq:ineq} est alors sommable (sous l'hypothèse que $x_c -q \ge \tan(\theta)/2$) car $x_c$ est une fonction linéaire de $q$. On vérifie aisément qu'il existe une constante $C_4 >0$ telle que la série est inférieure à $C_4e^{- \frac{C_4}{h}}$, ce qui conclut. Nous allons procéder de même pour les deux autres intégrales. Les deux se traitant de la même manière on ne fera que $I_q$. \\

On commence par réécrire 
$$ I_q = \inte{\RR} e^{\frac{i \varphi(x)}{h}} a(x) dx $$
avec $a(x) := e^{\frac{- \pi(x -q)^2}{h}} \cdot e^{\frac{- \beta \pi x^2 }{h \lambda^{2n}}} \chi_q(x) $. Pour aider Yann, on réécrit l'intégrale ci-dessus sous la forme 
\begin{align*}
 I_q & =  - i h \inte{\RR} \frac{i\varphi'(x)}{h} e^{\frac{i \varphi(x)}{h}} \frac{a(x)}{\varphi'(x)} dx \\
 	& = - i h \inte{\RR} \left(e^{\frac{i \varphi}{h}} \right)'(x) \frac{a(x)}{\varphi'(x)} dx 
\end{align*}
puis on lui demande de faire $k$ IPP sur ce modèle pour obtenir 
\begin{align*}
		 I_q  & = (ih)^k \inte{\RR} g(x) \frac{e^{\frac{i \varphi(x)}{h}}}{(\varphi'(x))^k} dx \\
		 & \le (ih)^k \overset{q + \frac{\tan(\theta)}{6}}{\inte{q - \frac{\tan(\theta)}{6}}}  g(x) \frac{e^{\frac{i \varphi(x)}{h}}}{(\varphi'(x))^k} dx 
\end{align*}

avec $g(x)$ s'écrit comme une combinaison linéaire finie de termes de la forme 
\begin{equation}
\label{eq:demerreurtermesipp}
a^{(m)} \underset{l}{\prod} \varphi^{(l)} 
\end{equation} 
où $j \ge 0$, $m \le k$ et $0  \le l $ et où le même indice $l$ peut apparaitre plusieurs fois dans le produit.\\

Dans notre cas, $\varphi^{(l)}$ est uniformément borné par le haut, $\varphi$ est un polynôme de degré deux après tout. \\

Rappelons que $\varphi'$ a un point critique non dégénéré en $p/ \tan(\theta)$. En particulier il existe une constante $C_5$ telle que sur l'intervalle $[ q - \frac{\tan(\theta)}{6}, q + \frac{\tan(\theta)}{6}]$ on a 
	$$ |\varphi'| \ge C_5 |q - x_c| \ .  $$
La forme explicite de la fonction $a$ donne aussi qu'il existe $C_7 > 0$ et pour tout $k$ il existe une constante $C_6^k >0$ telle 
	%$$ | a^{(k)}| \le C_6^k h^{-\frac{k}{2}} e^{ \frac{ - C_7 q^2}{h \lambda^{2n}}} $$
	$$ | a^{(k)}| \le C_6^k h^{-k} e^{-\frac {C_6^k}{h}}e^{ \frac{ - C_7 q^2}{h \lambda^{2n}}} $$
Au final on obtient une constante $C_8^k$ telle que  
%$$ \lvert I_q\rvert \le C_8^k C(h,n) \frac{h^{\frac{k}{2}}}{|x_c -q|^k}  e^{ \frac{ -C_7 q^2}{h \lambda^{2n}}} \ . $$
$$ \lvert I_q\rvert \le C_8^k C(h,n) \frac{e^{-\frac {C_6^k}{h}}}{|x_c -q|^k}  e^{ \frac{ -C_7 q^2}{h \lambda^{2n}}} \ . $$Comme précédemment la série sur $(q,p) \notin \mathcal{D}_{\theta}$ est uniformément dominée par une constante fois %$h^{\frac{k}{2}}$
$e^{-\frac {C_6^k}{h}}$. On laisse le cas de l'intégrale $I_0$ en exercice pour Yann, histoire de vérifier qu'il a bien compris comment faire des IPP. 

\vspace{3mm} En remettant ensemble les estimations obtenues pour $I_c$, $I_0$ et $I_q$ on obtient le résultat qu'on voulait.\hfill   $\blacksquare$

\subsection{Différence le long de la droite.}

Maintenant qu'on a montré que les termes loin de $\mathcal{D}_{\theta}$ ne comptent pas, il nous faut estimer la différence entre les deux termes 

$$\sum_{\substack{(q,p) \in (q_0,p_0) + \ZZ^2  \\
 (q, p) \in \mathcal{D}_{\theta}}} \scal{ \widehat{M^n}(f_h)}{ \Phi^h_{(q ,p)}}  \ \text{et} \ \sum_{\substack{(q,p) \in (q_0,p_0) + \ZZ^2  \\
 (q, p) \in \mathcal{D}_{\theta}}} \scal{ \mathcal{L}^h(n)}{ \Phi^h_{(q ,p)}} $$ 
%$$  \somme{ \begin{split}
% (q,p) \in (q_0,p_0) + \ZZ^2  \\
% (q, p) \in \mathcal{D}_{\theta}\end{split}} \scal{ \widehat{M^n}(f_h)}{ \Phi^h_{(q ,p)}} \ \text{et} \ \somme{ \begin{split}
% (q,p) \in (q_0,p_0) + \ZZ^2  \\
% (q, p) \in \mathcal{D}_{\theta} \end{split}} \scal{ \mathcal{L}^h(n)}{ \Phi^h_{(q ,p)}}.$$ 
 Pour faire cette estimation, on introduit les notations suivantes pour les sommes ci-dessus.
 
 \begin{itemize}
 
  \item $p(m)$ est l'unique entier tel que $(q_0 + m, q_0 + p(m))$ appartienne à $\mathcal{D}_{\theta}$.
 
\item  $M_{(n,h)}(q_0,p_0) := \somme{m \in \mathbb{Z}} \scal{ \widehat{M^n}(f_h)}{ \Phi^h_{(q_0 + m ,p_0 + p(m))}}$
 
 \item $L_{(n,h)}(q_0,p_0)  := \somme{m \in \mathbb{Z}} \scal{ \mathcal{L}^h(n)}{ \Phi^h_{(q_0 + m ,p_0 + p(m))}}$

\end{itemize}    

Dans ce paragraphe on démontre l'énoncé suivant. 
\vspace{2mm}

\fbox{
\parbox{\textwidth}{
\begin{proposition}
\label{prop:poubelle2}
Il existe une constante $C_{17} > 0$  telle que pour tout $(q_0,p_0) \in \mathbb{T}^2$ et pour tout $n \geq \frac{1}{3}  \frac{\lvert \log h\rvert}{\log \lambda}$ on ait 
$$ \left\lvert L_{(n,h)}(q_0,p_0)  - M_{(n,h)}(q_0,p_0)\right\rvert \leq C_{17} (\sqrt{h} \cdot \lambda^{-\frac{n}{2}} + e^{-\frac{1}{h}}).$$
\end{proposition} 
}
}

\vspace{3mm}

\paragraph*{\bf Estimations préliminaires.}

On cherche à estimer la différence 

\begin{equation*}
\left\lvert M_{(n,h)}(q_0,p_0)  -  L_{(n,h)}(q_0,p_0)\right\rvert  \leq \sum_{m \in \mathbb{Z}} \left\lvert\scal{ (\widehat{M^n}(f_h) - \mathcal{L}^h(n))}{ \Phi^h_{(q_0 + m ,p_0 + p(m))}}\right\rvert .
\end{equation*}
\noindent L'idée qu'on va implémenter est assez simple. On voudrait faire comme si le fait suivant était vrai: il existe une constante $C_9 >0$ telle que pour tout $m \in \mathbb{Z}$ 
$$ \left\lvert \scal{ (\widehat{M^n}(f_h) - \mathcal{L}^h(n))}{ \Phi^h_{(q_0 + m ,p_0 + p(m))}} \right\rvert \leq C_9 \cdot  \left\lvert(\widehat{M^n}(f_h) - \mathcal{L}^h(n))(q_0 + m)\right\rvert \cdot h^{\frac{1}{4}}.$$ Cette majoration traduirait simplement le fait que la fonction $|\widehat{M^n}(f_h) - \mathcal{L}^h(n)|$ varie lentement et que la fonction $|\Phi^h_{(q_0 + m ,p_0 + p(m))}|$ est une fonction très concentrée près de $q_0 + m$ de norme $\mathrm{L}^1$ de l'ordre $h^{\frac{1}{4}}$. Ce résultat est clairement vrai dès lors que $q_0 + m$ n'est pas trop grand (et donc le terme calculé n'est pas trop petit), mais pour les termes $m >> 1$ (qui sont négligeables), une telle formule n'est pas vraie et il faut quand même dire quelque chose qui permette d'effectivement s'en débarrasser. C'est la raison pour laquelle la démonstration qui suit est longue et pénible, quand bien même l'idée qui la sous-tend est relativement simple.

On a par définition
$$
\Big\langle\big(\widehat{M^n}(f_h) - \mathcal L^h(n)\big) \cdot \Phi^h_{(q_0+m, p_0+p(m))}\Big\rangle = \int_{\mathbb{R}}{(\widehat{M^n}(f_h)(x) - \mathcal L^h(n)(x)) \Phi^h_{(q_0+m, p_0+p(m))}(x)dx}.
$$
On découpe cette intégrale en trois. On définit

\begin{itemize}
\item $I_1 = [q_0+m-1,q_0+m+1]$;

\item $I_2 = [q_0+ \frac{m}{2}-1,q_0+\frac{3m}{2}+1] \setminus I_1$;

\item $I_3 = \mathbb{R} \setminus (I_1 \cup I_2)$.
\end{itemize}

{\bf L'intervalle } $I_1$ {\bf correspond au terme principal de l'heuristique et les deux autres correspondent au bruit qu'on veut négliger.}

\begin{lemma} \label{lemme:sommecoupee2}
  Il existe une constante $C>0$ telle que \\
  si $(q_0 -1 + \frac{m}{2}) \leq \lambda^{2n} \sqrt{h}$ on ait
\begin{equation*}
| \int_{I_2}{(\widehat{M^n}(f_h)(x) - \mathcal L^h(n)(x)) \Phi^h_{(q_0+m, p_0+p(m))}(x)dx} | \leq C   e^{-\frac{1}{h}}  (\lambda^{-\frac{3n}{2}}\cdot e^{-\frac{(q_0 - 1 + \frac{m}{2})^2}{h\lambda^{2n}}} + e^{-\lambda^{n}})
\end{equation*}
et si $(q_0 -1 + \frac{m}{2}) \geq \lambda^{2n} \sqrt{h}$ on ait
\begin{equation*}
| \int_{I_2}{(\widehat{M^n}(f_h)(x) - \mathcal L^h(n)(x)) \Phi^h_{(q_0+m, p_0+p(m))}(x)dx} | \leq C   e^{-\frac{1}{h}}  \lambda^{\frac{n}{2}}\cdot e^{-\frac{(q_0 - 1 + \frac{m}{2})^2}{h\lambda^{2n}}}.
\end{equation*}

\end{lemma}

\begin{proof}

En dehors de $I_1$, on a $\Phi^h_{(q_0+m, p_0+p(m))}(x) \leq C_h e^{-\frac{1}{h}}$ d'où 
$$| \int_{I_2}{(\widehat{M^n}(f_h)(x) - \mathcal L^h(n)(x)) \Phi^h_{(q_0+m, p_0+p(m))}(x)dx} |  \leq C_h e^{-\frac{1}{h}}  \int_{I_2}{|\widehat{M^n}(f_h)(x) - \mathcal L^h(n)(x)|dx}.$$ On fait maintenant la majoration 

$$|\widehat{M^n}(f_h)(x) - \mathcal L^h(n)(x)| \leq  C(h,n) e^{-\frac{x^2}{h \lambda^{2n}}} (1 - e^{-\frac{x^2}{h \lambda^{4n}}}).$$ On continue avec la majoration grossière donnée par $I_2 \subset [q_0 - 1 + \frac{m}{2}, + \infty[$ de laquelle on déduit 

$$ | \int_{I_2}{(\widehat{M^n}(f_h)(x) - \mathcal L^h(n)(x)) \Phi^h_{(q_0+m, p_0+p(m))}(x)dx} |  \leq  C(h,n) C_h  e^{-\frac{1}{h}}   \int_{q_0 - 1 + \frac{m}{2}}^{+\infty}{e^{-\frac{x^2}{h \lambda^{2n}}}(1 - e^{-\frac{x^2}{h \lambda^{4n}}})dx}.$$ On fait le changement de variable $ y = \frac{x}{\sqrt{h} \lambda^n}$ pour trouver 

$$ | \int_{I_2}{(\widehat{M^n}(f_h)(x) - \mathcal L^h(n)(x)) \Phi^h_{(q_0+m, p_0+p(m))}(x)dx} |  \leq 2 C(h,n) C_h  e^{-\frac{1}{h}}  \sqrt{h} \lambda^n \int_{(q_0 - 1 + \frac{m}{2})h^{-\frac{1}{2}}\lambda^{-n}}^{+\infty}{e^{-y^2}(1 -e^{-\frac{y^2}{\lambda^{2n}}})dy}.$$ A une constante près dont on fera grâce au lecteur, 

$$ (1 -e^{-\frac{y^2}{\lambda^{2n}}})) \leq \min(1, \frac{y^2}{\lambda^{2n}}).$$ En utilisant $\int_{z}^{+\infty}{e^{-y^2}dy} \leq e^{-z^2}$ on obtient  pour l'intégrale générale de la forme $\int_{\alpha}^{+\infty}{e^{-y^2}(1 -e^{-\frac{y^2}{\lambda^{2n}}})dy}$ la majoration, si $\alpha \leq \lambda^n$

$$\int_{\alpha}^{+\infty}{e^{-y^2}(1 -e^{-\frac{y^2}{\lambda^{2n}}})dy} \leq \int_{\alpha}^{\lambda^n}{\frac{y^2}{\lambda^{2n}}e^{-y^2}} + e^{-\lambda^{2n}}$$

$$\int_{\alpha}^{+\infty}{e^{-y^2}(1 -e^{-\frac{y^2}{\lambda^{2n}}})dy} \leq \lambda^{-n}e^{-\alpha^2} + e^{-\lambda^{2n}}$$

 et si $\alpha \geq \lambda^n$ on a simplement
 
 $$\int_{\alpha}^{+\infty}{e^{-y^2}(1 -e^{-\frac{y^2}{\lambda^{2n}}})dy} \leq e^{-\alpha^2}$$

On obtient donc 

\begin{itemize}
\item $(q_0 -1 + \frac{m}{2}) \leq \lambda^{2n} \sqrt{h}$,

$$ | \int_{I_2}{(\widehat{M^n}(f_h)(x) - \mathcal L^h(n)(x)) \Phi^h_{(q_0+m, p_0+p(m))}(x)dx} |  \leq C(h,n) C_h  e^{-\frac{1}{h}}  \sqrt{h} \lambda^n (\lambda^{-2n}\cdot e^{-\frac{(q_0 - 1 + \frac{m}{2})^2}{h\lambda^{2n}}} + e^{-\lambda^{2n}}).$$

\item si $(q_0 -1 + \frac{m}{2}) \geq \lambda^{2n} \sqrt{h}$

$$ | \int_{I_2}{(\widehat{M^n}(f_h)(x) - \mathcal L^h(n)(x)) \Phi^h_{(q_0+m, p_0+p(m))}(x)dx} |  \leq C(h,n) C_h  e^{-\frac{1}{h}}  \sqrt{h} \lambda^n e^{-\frac{(q_0 - 1 + \frac{m}{2})^2}{h\lambda^{2n}}}.$$

\end{itemize}

\end{proof}

\begin{lemma}
\label{lemme:sommecoupee3}
Il existe une constante $C>0$ telle que

$$ | \int_{I_3}{(\widehat{M^n}(f_h)(x) - \mathcal L^h(n)(x)) \Phi^h_{(q_0+m, p_0+p(m))}(x)dx} | \leq C \lambda^{-\frac{n}{2}} e^{-\frac{m^2}{4h}} $$

\end{lemma}

\begin{proof}

On utilise que sur $I_3$, $|\Phi^h_{(q_0+m, p_0+p(m))}(x)| \leq C_h e^{-\frac{m^2}{4h}}$ 

$$| \int_{I_3}{(\widehat{M^n}(f_h)(x) - \mathcal L^h(n)(x)) \Phi^h_{(q_0+m, p_0+p(m))}(x)dx} |  \leq  C_h e^{-\frac{m^2}{4h}} \int_{\mathbb{R}}{|(\widehat{M^n}(f_h)(x) - \mathcal L^h(n)(x))|dx} $$

et que grâce au calcul dans la démonstration de la proposition précédente 

$$\int_{\mathbb{R}}{|(\widehat{M^n}(f_h)(x) - \mathcal L^h(n)(x))|dx} \leq C(n,h) \sqrt{h} \lambda^n (\lambda^{-n} + e^{- \lambda^{2n}}) $$ ce qui donne, en utilisant les valeurs explicites de $C(n,h)$ et $C_h$ la borne $C \lambda^{-\frac{n}{2}} e^{-\frac{m^2}{4h}} $.

\end{proof}

\begin{lemma}
\label{lemme:sommecoupee1}
Il existe une constante $C>0$ telle que si  $q_0 + m \notin [-1,1]$

$$ | \int_{I_1}{(\widehat{M^n}(f_h)(x) - \mathcal L^h(n)(x)) \Phi^h_{(q_0+m, p_0+p(m))}(x)dx} | \leq C h^{\frac{1}{4}} \sup_{x \in I_1}| \widehat{M^n}(f_h)(x) - \mathcal L^h(n)(x)|.$$

\end{lemma}

\begin{proof}

C'est une conséquence facile du fait que $$\int_{\mathbb{R}}{|\Phi^h_{(q_0+m, p_0+p(m))}(x)|dx} \leq C \cdot h^{\frac{1}{4}}$$.

\end{proof}

\paragraph*{\bf Somme des termes correspondant aux intervalles $I_1$.}

Soit $S_1$ cette somme. On ne traite que les $m \in \mathbb{N}$, l'autre partie de la somme se traite de manière symétrique. On a, grâce au lemme \ref{lemme:sommecoupee1}
\begin{align*}
|S_1| & \leq C_9 h^{\frac{1}{4}}\sum_{m \in \mathbb{Z}}{ \sup_{x \in I_1}| \widehat{M^n}(f_h)(x) - \mathcal L^h(n)(x)|} \\
 & \leq C_9  h^{\frac{1}{4}} \sum_{m\in \mathbb{Z}}{C(n,h) e^{-\frac{\beta \pi (q_0 + m +1)^2}{h \lambda^{2n}}}  \left\lvert 1 - e^{-\frac{(q_0 +m-1)^2}{h} R_n} \right\rvert } 
\end{align*} 

où on a obtenu la seconde inégalité en utilisant l'estimation de la proposition~\ref{prop:approx}. On coupe maintenant la somme du terme de droite en deux sommes $\sum_{|m| \leq \sqrt{h}\lambda^{\frac{3n}{2}}}$ et $\sum_{|m| > \sqrt{h}\lambda^{\frac{3n}{2}}}$. 

On obtient 
\begin{align*}
\left\lvert\sum_{|m| \leq \sqrt{h}\lambda^{\frac{3n}{2}}}{e^{-\frac{\beta \pi (q_0 + m)^2}{h \lambda^{2n}}} \big|1 - e^{-\frac{(q_0 +m)^2}{h} R_n}} \big| \right\rvert & \leq  |1 - e^{-\lambda^{3n} R_n}| \sum_{|m| \leq \sqrt{h}\lambda^{\frac{3n}{2}}}{e^{-\frac{\beta \pi (q_0 + m)^2}{h \lambda^{2n}}}} \\
 & \leq  C_{11} \frac{1}{\lambda^n} \sqrt{h} \lambda^n  = C_{11} \sqrt{h} \\
\end{align*} pour un constante $C_{11} >0$, et où on a utilisé les deux faits suivants :

\begin{enumerate}

\item $R_n \leq \frac{C}{\lambda^{4n}}$ pour une certaine constante $C_{12}$ indépendante de $h$ et $n$;

\item $\sum_{|m| \leq \sqrt{h}\lambda^{\frac{3}{2}}}{e^{-\frac{\beta \pi (q_0 + m)^2}{h \lambda^{2n}}}}  \leq C_{13} \sqrt{h} \lambda^n $ pour une constante $C_{13}$, toujours indépendante de $h$ et $n$.

\end{enumerate} 
On invite notre lecteur à noter que la comparaison $ |1 - e^{-\frac{(q_0 +m)^2}{h} R_n} |\leq |1 - e^{-\lambda^{3n} R_n}|$ qu'on a utilisé n'est valide que si $\sqrt{h}\lambda^{\frac{3n}{2}} \geq 1$, de sorte que ces calculs sont valides si $n \geq \frac{1}{3}  \frac{|\log h|}{\log \lambda}$.

\medskip On continue avec le second bout de somme 

\begin{align*}
\left\lvert \sum_{\vert m\vert \geq \sqrt{h}\lambda^{\frac{3n}{2}}}{e^{-\frac{\beta \pi (q_0 + m)^2}{h \lambda^{2n}}}|1 - e^{-\frac{(q_0+m)^2}{h} R_n}} | \right\rvert& \leq \left\lvert \sum_{|m| \geq \sqrt{h}\lambda^{\frac{3n}{2}}}{e^{-\frac{\beta \pi (q_0 + m)^2}{h \lambda^{2n}}}} \right\rvert\\
 & \leq  C_{14}  \sqrt{h} \lambda^n \exp(-\beta\pi \lambda^n)   \\
\end{align*} où on a utilisé une comparaison du type

$$ \sum_{|m| \geq \sqrt{h}\lambda^{\frac{3n}{2}}}{e^{-\frac{\beta \pi (q_0 + m)^2}{h \lambda^{2n}}}} \leq C_{15} \int_{x \geq \sqrt{h}\lambda^{\frac{3n}{2}}}{e^{-\frac{\beta \pi x^2}{h \lambda^{2n}}}dx} $$ pour une constante universelle $C_{15}$. 

En remettant tout ensemble, et en utilisant le fait que $C(n,h) \leq C_{16} \frac{h^{-\frac{1}{4}}}{\sqrt{\lambda^n}}$ pour une constante $C_{16}$ universelle on obtient l'estimation de la différence 

$$ |S_1| \leq C_{17} \sqrt{h} \lambda^{-\frac{n}{2}}.$$

\paragraph*{\bf Somme des termes correspondant aux intervalles $I_2$.}

On note $S_2$ cette somme. On fait l'hypothèse simplificatrice $q_0 -1 = 0$ (le calcul se fait de la même manière dans le cas général, mais ça nous simplifiera les notations). De même, on somme seulement sur $\mathbb{N}$ dans ce qui suit (ce qui par symétrie ne change rien). 
Le lemme \ref{lemme:sommecoupee2} nous donne (à une constante multiplicative universelle près qu'on omet)

$$ |S_2| \leq e^{-\frac{1}{h}}\sum_{0 \leq m \leq \lambda^{2n} \sqrt{h}}{(\lambda^{-\frac{3n}{2}}e^{-\frac{m^2}{h \lambda^{2n}}} + e^{-\lambda^n}) }  +  e^{-\frac{1}{h}} \sum_{ m > \lambda^{2n} \sqrt{h}}{\lambda^{\frac{n}{2}}e^{-\frac{m^2}{h \lambda^{2n}}}}.$$ En utilisant de simples comparaisons série/intégrale on obtient 

$$\sum_{0 \leq m \leq \lambda^{2n} \sqrt{h}}{e^{-\frac{m^2}{h \lambda^{2n}}}} \leq 1 + \sqrt{h}\lambda^{n} \int_0^{\lambda^{n}}{e^{-y^2}dy}$$ et

$$\sum_{ m > \lambda^{2n} \sqrt{h}}{\lambda^{\frac{n}{2}}e^{-\frac{m^2}{h \lambda^{2n}}}} \leq \sqrt{h}\lambda^{n}\int_{\lambda^n-1}^{+\infty}{e^{-y^2}dy}.$$ En mettant tout bout à bout on obtient 

$$ |S_2| \leq e^{-\frac{1}{h}}(1 + \sqrt{h} \lambda^{-\frac{n}{2}}\int_{0}^{+\infty}{e^{-y^2}dy} + \sqrt{h}\lambda^{2n}e^{-\lambda^n} + \sqrt{h}\lambda^n e^{-(\lambda^n - 1)^2}).$$ Comme les fonctions $x \mapsto xe^{-x^2}$ et $x \mapsto x^2e^{-x}$ sont bornée) on obtient que 

$|S_2| \leq C e^{-\frac{1}{h}}$ pour une certaine constante $C$.

\vspace{2mm} En regroupant les estimations obtenues respectivement pour $|S_1|$, $|S_2|$ et $|S_3|$ on obtient la proposition \ref{prop:poubelle2}.

\vspace{3mm}

\paragraph*{\bf Somme des termes correspondant aux intervalles $I_3$.}

Cette somme $S_3$ est plus facile à gérer. En utilisant le lemme \ref{lemme:sommecoupee3}, on obtient 

$$ S_3 \leq \lambda^{-\frac{n}{2}}\sum_{m \geq 1}{e^{-\frac{m^2}{4h}}} $$ ce qui, à une constante près, donne 

$$ S_3 \leq \lambda^{-\frac{n}{2}} \sqrt{h}.$$

\section{Calcul des termes d'interférences principaux}

Dans la section précédente, on a montré qu'a une erreur près (qu'on a calculée précisément), on pouvait remplacer le calcul de $$\somme{
 (q,p) \in (q_0,p_0) + \ZZ^2} \scal{ \widehat{M^n}(f_h)}{ \Phi^h_{(q ,p)}}$$ par  $$\sum_{m \in \mathbb{Z}}{ \scal{  \mathcal{L}^h(n)}{ \Phi^h_{(q_0 + m ,p_0 + p(m))}}}$$ où $\mathcal{L}^h(n)$ est la fonction lagrangienne amortie  $x \mapsto C(h,n) e^{ \frac{i \pi \tan(\theta) x^2}{h}} \cdot e^{- \frac{ \beta \pi x^2 }{h  \lambda^{2n}}}.$ Dans ce paragraphe on calcule explicitement cette dernière somme.

\vspace{3mm}

\fbox{
\parbox{\textwidth}{
\textbf{Note pour le lecteur aguerri :} On a jusqu'à présent effectué des calculs qui, bien que techniquement pénibles, n'utilisent jamais la forme particulière de l'état lagrangien qu'on utilise pour approximer l'évolué d'un paquet d'ondes (fonction lagrangienne quadratique) ou des paquets d'ondes (qui sont des paquets d'onde gaussiens). Ils devraient en particulier se généraliser à des cas non-linéaires.

\vspace{2mm} Le calcul à suivre va par contre utiliser cette forme particulière, principalement pour utiliser la formule exacte pour la transformation en onde planes d'un paquet d'ondes gaussien. Il s'agit d'un artifice de calcul dont on aurait pu se passer en utilisant la formule de la phase stationnaire complexe, ce qui aurait eu le mérite mettre l'intégralité du calcul des termes d'interférence sous une forme qui se prête à une généralisation directe à d'autres contextes. 

\vspace{2mm} A mesure que le calcul des nos termes d'interférence s'allongeait, nous prîmes la décision de couper court et d'utiliser des formules exactes. Puisse le lecteur ayant le goût des bonnes démonstrations nous pardonner cette petite excentricité. 
}
}

\subsection{Cas du propagé d'un paquet d'ondes en $(0,0)$}
\label{subsec:po00}

On rappelle pour la suite du calcul que 

\begin{itemize}

\item $C_h$ est la constante de normalisation $\mathrm{L}^2$ d'un paquet d'ondes;

\item $C(h,n)$ est la constante de normalisation $\mathrm{L}^2$ de $\mathcal{L}^h(n)$.

\end{itemize}

On va donc maintenant calculer explicitement 
\begin{align*}
 \scal{\mathcal{L}^h(n)}{ \Phi^h_{(q ,p)}} & = C(h,n)  C_h \inte{\RR} e^{ \frac{ i \pi \tan \theta x^2}{h}} \cdot e^{\frac{- \pi (x-q)^2}{h}}  \cdot e^{- \frac{ \beta \pi x^2 }{h  \lambda^{2n}}} e^{- \frac{ 2 i \pi p (x-q)}{h}} dx \\
  & = C(h,n) C_h e^{\frac{2i\pi\, pq}{h}} e^{- \frac{\pi q^2}{h}} \inte{\RR} e^{ \frac{ i \pi \tan \theta x^2}{h}} \cdot e^{- \frac{ \pi x^2}{h}}  \cdot e^{- \frac{ \beta \pi x^2 }{h  \lambda^{2n}}} e^{- \frac{ 2 i \pi x (p + iq)}{h}} dx \\
 % & =  C(h,n) C_h e^{\frac{2i\pi\, pq}{h}} e^{- \frac{\pi q^2}{h}} \inte{\RR}  \exp \left(  \frac{ i \pi \tan \theta x^2}{h} - \frac{ \pi x^2}{h} - \frac{ \beta \pi x^2 }{h  \lambda^{2n}} \right) e^{- \frac{ 2 i \pi x (p + iq)}{h}} dx \\
 & = C(h,n) C_h e^{\frac{2i\pi\, pq}{h}} e^{- \frac{\pi q^2}{h}} \inte{\RR}  \exp \left(  \frac{ - \pi x^2}{h} \left( - i \tan \theta + 1  + \frac{ \beta}{ \lambda^{2n}} \right) \right) e^{- \frac{ 2 i \pi x (p + iq)}{h}} dx .
\end{align*}  On reconnait la $h$-transformée de Fourier d'une gaussienne évalué au point %$q + ip$
$p+iq$. On se retrouve avec 
\begin{align*}
 \scal{\mathcal{L}^h(n)}{ \Phi^h_{(q ,p)}} & =  C(h,n) C_h h^{\frac{1}{2}} e^{- \frac{\pi q^2}{h}} e^{\frac{2i\pi\,pq}{h}} \exp \left(  \frac{ - \pi (p+ iq)^2}{h} \left( - i \tan \theta + 1  + \frac{ \beta}{ \lambda^{2n}} \right)^{-1} \right) \\
& = C(h,n) C_h h^{\frac{1}{2}}  \exp \left(  \frac{ - \pi}{h} \left( (p + iq)^2 \left( - i \tan \theta + 1  + \frac{ \beta}{ \lambda^{2n}} \right)^{-1} + q^2 - 2ipq \right) \right)
 \end{align*}
On note $\alpha(n)$ la fonction de $n$ telle que
 	$$ \left( - i \tan \theta + 1  + \frac{ \beta}{ \lambda^{2n}} \right)^{-1} = \cos^2(\theta) + i \sin(\theta) \cos(\theta)  + \alpha(n) \ . $$
Notons qu'il existe deux constantes $C_{18}, C_{19} > 0$ telle que $|\alpha(n)| = C_{18} \lambda^{-2n} + e(n)$ où la suite $e(n)$ satisfait $|e(n)| \leq C_{19} \lambda^{-4n}$ (c'est un simple développement limité). 
Examinons plus précisément le terme dans la gaussienne. Afin d'alléger les notations nous omettrons l'argument $\theta$ dans les fonctions trigonométriques. 
\begin{align*} 
(p + iq)^2 & \left( - i \tan  + 1  + \frac{ \beta}{ \lambda^{2n}} \right)^{-1} + q^2 - 2ipq\\ 
& = (p + iq)^2 ( \cos^2 + i \sin  \cos  + \alpha(n)) + q^2 - 2ipq \\
& = (p \cos - q \sin)^2 + i \sin \cos (p^2 - q^2) + 2 iq p \cos^2 + (p + iq)^2 \alpha(n) - 2ipq \\
%& = (p \cos - q \sin)^2 + i \cos \sin ( p^2 -q^2) - 2 i  q p \sin^2 + 2i pq + (p + iq)^2 \alpha(n) - 2ipq \\
& = (p \cos - q \sin)^2 + i \cos \sin ( p^2 -q^2) - 2 i  q p \sin^2 + (p + iq)^2 \alpha(n) 
\end{align*} 
On obtient donc

\begin{align*}
 \scal{\mathcal{L}^h(n)}{ \Phi^h_{(q ,p)}}  = C(h,n) C_h\sqrt{h} \exp \Big( \frac{- \pi \cos^2}{h} (p - & q \tan)^2 \Big)  \\
   \cdot \exp \Big( \frac{- i \pi \cos^2}{h} & \Big(  \tan ( p^2 -q^2) - 2 q p \tan^2 \Big) \\ 
  	& \cdot \exp \Big( \frac{- \pi }{h} \Big( (p + iq)^2 \alpha(n) \Big) \Big) 
\end{align*}

On rappelle qu'on regarde la somme sur les paquets d'ondes en des points $(q,p)$ qui sont dans un voisinage de la droite $\mathcal{D}_{\theta}$, ce qui implique qu'on puisse écrire, pour un $m \in \mathbb{N}$ : $q = q(m) := q_0 + m$ et $p = p(m) = p_0 + k(m)$ est l'unique %élément de $p_0 + \ZZ^2$ 
entier tel que 
	$$ q \tan - p \in \left]-1/2, 1/2\right[ $$
C'est à dire 
 $$ (q_0 + m) \tan - p_0 - k(m) \in  \left]-1/2, 1/2\right[ $$
Et donc $q - \tan p$ est le reste modulo 1 dans l'intervalle $\left]-1/2, 1/2\right[$ du nombre $ q_0 \tan - p_0 + m \tan.$ Si on note $s_0 := q_0 \tan - p_0$, $\tan \theta = \alpha$ et  $\mathrm{R}_{\tan (\theta)} = \mathrm{R}_{\alpha}$ la rotation d'angle $\alpha = \tan(\theta)$, on a que $$ q_0 \tan - p_0 + m \alpha = d_{S_1}(\mathrm{R}_{\alpha}^m(s_0),0) $$ où $d_{S^1}(\cdot,0)$ est la distance signée à $0$ dans le cercle $S^1 = \R/ \Z$.

	$$ q(m) \tan(\theta) - p(m) = d_{S^1}(\mathrm{R}_{\alpha}^m(s_0),0) \ . $$
On remplace alors $p$ par $q \tan -d_{S^1}(\mathrm{R}_{\alpha}^m(s_0),0)$. Le premier des trois termes ci-dessus devient 
	$$ \exp \Big( \frac{- \pi d_{S^1}(\mathrm{R}_{\alpha}^m(s_0),0)^2 \cos^2 }{h}  \Big) $$  
On calcule la phase du second 
\begin{multline*}
	\tan ( p^2 -q^2) - 2 q p \tan^2 = \\
	  \tan \Big( q^2( \tan^2 - 1)  + d_{S_1}(\mathrm{R}_{\alpha}^m(s_0),0)^2 - 2 q \tan d_{S^1}(\mathrm{R}_{\alpha}^m(s_0),0) \Big) - 2 \tan^2 q ( q \tan - d_{S^1}(\mathrm{R}_{\alpha}^m(s_0),0)) \\
	 = - \tan \left( q^2( \tan^2 + 1) - d_{S^1}(\mathrm{R}_{\alpha}^m(s_0),0)^2)  \right) 
\end{multline*} 
On réécrit alors 
\begin{multline*}
 \scal{\mathcal{L}^h(n)}{ \Phi^h_{(q ,p)}}  =  C(h,n) C_h \sqrt{h} \exp \Big(  \frac{- \pi \cos^2}{h} (1 - i \tan \theta) d_{S^1}(\mathrm{R}_{\alpha}^m(s_0),0)^2 \Big) \Big)  \\   
  	 \cdot \exp  \Big(\frac{ i \pi \cos^2}{h} \Big(  q^2( \tan^3 + \tan) \Big) \Big)
  	   \cdot \exp \Big( \frac{- \pi \alpha(n)}{h}  (p + iq)^2 \Big) 
\end{multline*}

On est content des deux premiers termes car on va pouvoir écrire sous la forme d'une fonction des orbites d'un système dynamique, nous y reviendrons dans un instant. On veut maintenant mettre le dernier terme sous la forme d'une fonction amortissement, en particulier sous la forme d'une fonction qui ne dépend que de la variable $m$. Rappelons que 
	$$ p(m) = q(m) \tan + d_{S^1}(\mathrm{R}_{\alpha}^m(s_0),0) \ .$$
 On compare le terme 
$$ \exp \Big( \frac{- \pi \alpha(n)}{h}  (p + iq)^2 \Big)  $$ 
au terme 
$$ \exp \Big( \frac{- \pi \alpha(n)}{h}  (q \tan + iq)^2 \Big) \  $$	
qui lui ne dépend que de $m$. Remarquons que l'on peut écrire la différence des deux termes précédents sous la forme 
	$$ G( \delta m) - G(\delta( m + E_m)) $$
Où 
\begin{itemize}
\item $G$ est une gaussienne d'une certaine variance qui ne dépend ni de $m$ ni de $n$ ni de $h$ ; 
\item  $\delta := C_{20}\sqrt{\frac{1}{\lambda^{2n}h}} \tends{ h \to 0} 0$ dès que $n \ge \frac{\ln \frac{1}{h}}{\log \lambda} $ avec $C_{20}$ une constante (qui peut être explicitée si besoin est) ;
\item $E_m := d_{S^1}(\mathrm{R}_{\alpha}^m(s_0),0) \tan^{-1} + \mathcal{O}(\lambda^{-2n})$ est le défaut. Il est uniformément borné par une constante $C_{21}$ ne dépendant ni de $n$ ni de $h$.
\end{itemize}

On majore 
\begin{align*}
 \big| G( \delta m) - G(\delta( m + E_m)) \big| & \le \Big| \inte{\delta m}^{\delta (m + E_m)} G'(t) dt \Big| \\
 & \le | \delta E_m | \supr{[\delta m, \delta (m + E_m)]} |G'| \\ 
 & \le \delta C_{21} |G'(t_m)| \ ,
 \end{align*} 
Pour un certain $t_m \in [\delta m, \delta (m + E_m)]$. On conclut en notant que les intervalles $[\delta m, \delta (m + E_m)]$ sont presque disjoints (il existe un nombre fini (uniforme en $\delta$) de recoupement de deux tels intervalles). En particulier il existe une constante $\kappa > 0$ telle que 
	$$ \somme{m \in \ZZ}  \big| G( \delta m) - G(\delta( m + E_m)) \big| \le \kappa \delta \frac{1}{\delta} \inte{\RR} |G'(t)| \le C_{22} $$
pour une constante $C_{22}$ bien choisie. Pour résumer, on a montré

\begin{lemma}
Il existe une constante $C_{23} > 0$ telle que pour tout $h > 0$ et pour tout $n \ge \frac{\ln \frac{1}{h}}{\log \lambda} $ alors 
\begin{multline*}
 \Big\lvert \somme{m \in \NN} \frac{\scal{\mathcal{L}^h(n)}{\widehat{T^h}_{(m, k(m))}(\Phi^h_{(q_0,p_0)})} }{\sqrt{h} C_hC(h,n)} \quad - \\  \somme{m \in \NN} G  \left(\frac{m}{\sqrt{h}\lambda^n} \right) \cdot \, \exp\left(-\frac{2i\pi}{h} \big(\frac{mk(m)}{2} + q_0k(m)\big)\right) \\
 \cdot \exp \Big( \frac{- \pi \cos^2}{h} (1 - i \tan \theta) d_{S^1}(\mathrm{R}_{\alpha}^m(s_0),0)^2 ) \Big)   \cdot
  	\exp  \Big(  \frac{ i \pi \tan \theta}{h}   (q_0+m)^2  \Big) \Big\rvert \le C_{23}
\end{multline*}
où $G : \mathbb{R} \longrightarrow \mathbb{C}$ est une fonction analytique à décroissance exponentielle (la fonction d'amortissement). 
\end{lemma}

\subsection{Cas d'un paquet d'ondes en un point arbitraire.}
\subsubsection{Discussion préliminaire}

Le calcul du paragraphe précédent concernait le propagé d'un paquet d'ondes centré au point $(0,0) \in \mathbb{T}^2$. On explique maintenant comment en déduire le cas général d'un paquet d'ondes centré en un point arbitraire $(a,b) \in \mathbb{T}^2$. On passera en accéléré sur certaines étapes qui se traitent exactement de la même manière dans le cas particulier $(a,b) = (0,0)$. En particulier, tous les calculs de la section \ref{sec:approx} fonctionnent exactement de la même manière, et on se ramène au calcul de la décomposition en paquet d'ondes de l'état lagrangien $\mathcal{L}^n_{(a',b')} = \widehat{T^h}_{a',b'}(\mathcal{L}^h(n))$ où l'on a continué de noter $$\mathcal{L}^h(n) =C(h,n) \exp(\frac{i\pi \tan(\theta) x^2}{h}) \exp(\frac{- \pi \beta}{h \lambda^{2n}}x^2)$$ et $(a',b') = M^n (a,b)$.

\vspace{1mm}

\fbox{
\parbox{\textwidth}{
On a montré que le propagé d'un paquet d'ondes est, à une erreur qui tend vers $0$ avec $h$, l'état lagrangien   
$$ \mathcal{L}^n_{(a',b')}:= C(h,n) \exp\left(\frac{2i\pi}{h}\frac{\tan(\theta)x^2 + s'x}{2}\right) \exp\left(\frac{- \pi \beta}{h \lambda^{2n}}(x-a')^2\right)$$ 
avec $s' = b' - \tan \theta a'$.

}
}

\vspace{2mm} 
Notre but ici est de refaire les calculs des paragraphes précédents. Le lecteur sera alors convaincu qu'on peut remplacer $(a',b')$ par son translaté entier dans $[0,1]^2$ sans changer l'élément de $\mathcal{H}^N$ qu'on obtiendra en symétrisant $\mathcal{L}^n_{(a',b')} := \widehat{T^h}_{a',b'}(\mathcal{L}^h(n))$. On s'attelle donc au calcul du terme suivant 
$$  \scal{\mathcal{L}^n_{(a',b')}}{ \widehat{T^h}_{(m,k)}\Phi^h_{(q_0 ,p_0)}} $$ 
qu'on réduit à celui de $  \scal{\mathcal{L}^n_{(a',b')}}{\Phi^h_{(q_0 + m ,p_0 + k)}} $ car 

$$  \scal{\mathcal{L}^n_{(a',b')}}{ \widehat{T^h}_{(m,k)}\Phi^h_{(q_0 ,p_0)}} =  e^{-\frac{i \pi}{h}km} e^{-\frac{2i \pi}{h} kq_0} \scal{\mathcal{L}^n_{(a',b')}}{ \Phi^h_{(q_0 + m ,p_0 + k)}} $$ (c'est une conséquence de la formule de la proposition \ref{prop:transPO}). Comme $\frac{1}{h} \in 2 \mathbb{N}$, on a $e^{-\frac{i \pi}{h}km}=1$ et on a donc  $$  \scal{\mathcal{L}^n_{(a',b')}}{ \widehat{T^h}_{(m,k)}\Phi^h_{(q_0 ,p_0)}} =   e^{-\frac{2i \pi}{h} kq_0} \scal{\mathcal{L}^n_{(a',b')}}{ \Phi^h_{(q_0 + m ,p_0 + k)}}.$$

\subsubsection{Le calcul}

Dans ce qui suit, on utilise la notation $(q,p) = (q_0 + m, p_0 + k)$.  Par définition, $\mathcal{L}^n_{(a',b')} =  \widehat{T^h}_{a',b'}(\mathcal{L}^h(n))$ donc

$$\mathcal{L}^n_{(a',b')} =  C(h,n) \exp(\frac{2i\pi (\tan(\theta) \frac{x^2}{2} + s'x)}{h}) \exp(\frac{- \pi \beta}{h \lambda^{2n}}(x-a')^2) $$ avec $s' = b' -  a' \tan \theta $. Notons que $ \varphi(x) = \tan(\theta) \frac{x^2}{2} + s'x$ est la fonction lagrangienne de la droite de pente $\theta$ passant $(a',b')$ dans $\R^2$.

\begin{multline*}
 \scal{\mathcal{L}^n_{(a',b')}}{\Phi^h_{(q ,p)}} = C(h,n)  C_h \inte{\RR} e^{ \frac{2 i \pi \varphi(x)}{h}} \cdot e^{\frac{- \pi (x-q)^2}{h}}  \cdot e^{- \frac{ \beta \pi (x-a')^2 }{h  \lambda^{2n}}} e^{- \frac{ 2 i \pi p (x-q)}{h}} dx \\
  = C(h,n) C_h e^{\frac{2i\pi pq}{h}} e^{- \frac{\pi q^2}{h} - \frac{\beta \pi a'^2}{h  \lambda^{2n}}}  \inte{\RR} e^{ \frac{ i \pi \tan \theta x^2}{h}} \cdot e^{ \frac{ 2i \pi s'x}{h}}  \cdot e^{- \frac{ \pi x^2}{h}}  \cdot e^{- \frac{ \beta \pi x^2 }{h  \lambda^{2n}}} e^{- \frac{ 2 i \pi x (p + iq)}{h}} e^{ \frac{ 2 \pi \beta a' x }{h \lambda^2n}} dx \\ 
  = C(h,n) C_h e^{\frac{2i\pi pq}{h}}e^{- \frac{\pi q^2}{h} - \frac{\beta \pi a'^2}{h  \lambda^{2n}}}  \inte{\RR} \exp \big( \frac{-\pi x^2}{h}( 1 -i \tan \theta + \frac{\beta}{\lambda^{2n}}) \big)  \cdot  e^{- \frac{ 2 i \pi x (p + iq - s' - \frac{\beta a'}{\lambda^{2n}})}{h})} dx \\ 
\end{multline*} 

On reconnait encore la formule de la $h$-transformée de Fourier d'une gaussienne, et on a donc 
\begin{multline*}
 \scal{\mathcal{L}^n_{(a',b')}}{\Phi^h_{(q ,p)}}  = \sqrt{h} C(h,n) C_h e^{\frac{2i\pi pq}{h}} e^{- \frac{\pi q^2}{h} - \frac{\beta \pi a'^2}{h  \lambda^{2n}}}  \\ \cdot \exp\left( \frac{-\pi}{h}\big(p + iq - s' - \frac{\beta a'}{\lambda^{2n}}\big)^2\big( 1 -i \tan \theta + \frac{\beta}{\lambda^{2n}}\big) ^{-1} \right)
 \end{multline*}
qu'on réécrit
 $$ \scal{\mathcal{L}^n_{(a',b')}}{\Phi^h_{(q ,p)}}  = \sqrt{h} C(h,n) C_h \exp(\frac{-\pi}{h} A(q,p)) $$ avec $$A(q,p) = (p + iq - s' - \frac{\beta a'}{\lambda^{2n}})^2  \left(1 - i \tan   + \frac{ \beta}{ \lambda^{2n}} \right)^{-1} + q^2 - 2ipq.$$ Comme précédemment, on calcule $A(q,p)$. On réécrit d'abord $\left(1 - i \tan   + \frac{ \beta}{ \lambda^{2n}} \right)^{-1} = \cos^{2} + i \sin \cos + \alpha(n)$ avec $\alpha(n)$ qui sera explicité plus tard quand on en aura besoin. Si on introduit 
 
\begin{itemize}

\item $A_1(q,p) = (p+iq)^2( \cos^{2} + i \sin \cos) +q^2 -2ipq$;

\item $A_2(q,p) = ((s')^2 - (p+iq)2s') (\cos^{2} + i \sin \cos)$;

\item $A_3(q,p) = -\frac{\beta a'}{\lambda^{2n}}((p + iq - s' - \frac{\beta a'}{\lambda^{2n}})) ( \cos^{2} + i \sin \cos + \alpha(n)) + ((s')^2 - (p+iq)2s') \alpha(n) + (p+iq)^2 \alpha(n)$;

\end{itemize}

on a $A(q,p) = A_1(q,p) + A_2(q,p) + A_3(q,p)$.
 
 \begin{align*} 
A_1(q,p)  & = (p + iq)^2 ( \cos^2 + i \sin  \cos ) + q^2 - 2ipq \\
& = (p \cos - q \sin)^2 + i \sin \cos (p^2 - q^2) + 2 iq p \cos^2  - 2ipq \\
& = (p \cos - q \sin)^2 + i \cos \sin ( p^2 -q^2) - 2 i  q p \sin^2 + 2i pq  - 2ipq \\
& = (p \cos - q \sin)^2 + i \cos \sin ( p^2 -q^2) - 2 i  q p \sin^2 
\end{align*}

On rappelle qu'on ne regarde que les points $(q,p)$ à distance moins de $\frac{1}{2}$ de la droite de pente $\theta$ passant par $(a',b')$. On écrit donc $q = q(m) := q_0 + m$ et $p = p(m) = p_0 + k(m)$, et la condition ci-dessus s'exprime de la façon suivante 
	$$ q \tan + s' - p \in ]-1/2, 1/2[.$$
C'est à dire 
 $$ (q_0 + m) \tan +s' - p_0 - k(m) \in  ]-1/2, 1/2[ $$
Et donc $q\tan - p + s'$ est l'unique représentant dans l'intervalle $]-1/2, 1/2[$ du nombre $ q_0 \tan - p_0 + m \tan + s'.$ Si on note $s_0 :=  p_0 - q_0 \tan$, $\tan \theta = \alpha$ et  $\mathrm{R}_{\tan (\theta)} = \mathrm{R}_{\alpha}$ du cercle d'angle $\tan(\theta)$, on a que $$ s'  + q_0 \tan - p_0 + m \alpha = d_{S^1}(\mathrm{R}_{\alpha}^m(s'),s_0) $$ où $d_{S^1}(\cdot,0)$ est la distance signée à $0$ dans le cercle $S^1 = \R/ \Z$. On a donc

	$$ q \tan(\theta) - p = d_{S^1}(\mathrm{R}_{\alpha}^m(s'),s_0) \ . $$
On remplace alors $p$ par $q \tan -d_{S^1}(\mathrm{R}_{\alpha}^m(s'),s_0)$. Pour simplifier les notations le temps du calcul on utilise $d = d_{S^1}(\mathrm{R}_{\alpha}^m(s'),s_0)$. Le terme $A_1$ devient alors 

$$ A_1(q,p) = \cos^2 d^2 + i \cos^2 \tan d^2   - i \tan q^2 $$

On s'occupe maintenant de $A_2$, 

$$A_2(q,p) = ((s')^2 - (p+iq)2s') (\cos^{2} + i \sin \cos)$$ On coupe le terme en deux 
$$A_2(q,p) = (s')^2  (\cos^{2} + i \sin \cos) - (p+iq)2s' (\cos^{2} + i \sin \cos).$$ Le premier terme n'est pas intéressant il contribuera à une constante multiplicative globale. On se concentre sur  $- (p+iq)2s' (\cos^{2} + i \sin \cos)$.

 \begin{align*} 
 &  - (p+iq)2s' (\cos^{2} + i \sin \cos)\\
& = -2s' (p \cos^2 - q\sin \cos) - 2s' i (q \cos^2 + p\sin \cos) \\
& =  -2s' \cos^2 (p  - q\tan) - 2s' i \cos^2(q + p \tan) \\
\end{align*}

On rappelle l'égalité $q \tan(\theta) - p = d_{S^1}(\mathrm{R}_{\alpha}^m(s'),s_0) -s' = d -s'$, on obtient 

 \begin{align*}  A_2(q,p) &  = (s')^2  (\cos^{2} + i \sin \cos) +  2s' \cos^2 (d-s') - 2s'i\cos^2(q(1+ \tan^2) + (s'- d)\tan) \\
 & = (s')^2  (\cos^2 + i \sin \cos -2 \cos^2 - \tan) +  2s' (\cos^2 + i \cos^2\tan) d - 2s'i q
\end{align*}

En se rappelant $ \scal{\mathcal{L}^n_{(a',b')}}{\Phi^h_{(q ,p)}}  = \sqrt{h} C(h,n) C_h \exp(\frac{-\pi}{h} A(q,p)) $ on obtient 

\begin{align*}
 \scal{\mathcal{L}^n_{(a',b')}}{\Phi^h_{(q ,p)}}  = \sqrt{h} C(h,n) C_h \exp(\frac{-\pi}{h} \cos^2 d^2 (1 + i \tan)) \\ \exp(\frac{2i\pi}{h}( \frac{\tan}{2} q^2 - s' q ) \exp(\frac{-\pi}{h} A_3(q,p))
\end{align*} En se rappelant que $\scal{\mathcal{L}^n_{(a',b')}}{ \widehat{T^h}_{(m,k(m))}\Phi^h_{(q_0 ,p_0)}} =   e^{-\frac{2i \pi}{h} k(m)q_0} \scal{\mathcal{L}^n_{(a',b')}}{ \Phi^h_{(q_0 + m ,p_0 + k)}}$ et que $k(m) =  s' - s_0 + m\tan - d)$ on obtient 

\begin{align*}
\scal{\mathcal{L}^n_{(a',b')}}{ \widehat{T^h}_{(m,k(m))}\Phi^h_{(q_0 ,p_0)}}  =  D(s',s_0) \sqrt{h} C(h,n) C_h \cdot \\
   \exp(\frac{-\pi}{h} (\cos^2 d^2 (1 + i \tan)) \\  \exp(\frac{2i\pi}{h}( \frac{\tan}{2} m^2 + s' m) \cdot \exp(\frac{2i \pi}{h}dq_0)\exp(\frac{-\pi}{h} A_3(q,p))
 \end{align*} où $D(s',s_0)$ est un terme de phase (\textit{i.e.} un nombre complexe de module $1$).  On règle son cas au dernier terme contenant la fonction $A_3$ dans la proposition suivante. 

\vspace{3mm}

\fbox{
\parbox{\textwidth}{
\begin{proposition}[Calcul des termes d'interférences]
\label{prop:interference}
Il existe 
\begin{itemize}
\item une fonction $G : \mathbb{R} \longrightarrow \mathbb{C}$ une fonction à décroissance sur-exponentielle (la fonction d'amortissement)
\item une constante $C_{24} > 0$ 
\end{itemize} tels que pour tout $h > 0$, tout $(a,b) \in \mathbb{T}^2$ et pour tout $n \geq \frac{\ln(\frac{1}{h})}{\log \lambda}$ on a 
\begin{multline*}
 \Big| \somme{m \in \NN} \frac{\scal{\mathcal{L}^n_{(a',b')}}{\widehat{T^h}_{(m, k(m))}(\Phi^h_{(q_0,p_0)})} }{\sqrt{h} C(h,n) C_h } -    \\ 
 \somme{m \in \NN}  G  \left(  \frac{m}{\sqrt{h}\lambda^n} \right) \cdot \exp \big(\frac{-\pi}{h} \cos^2 d^2 (1 + i \tan) - 2iq_0 d ) \big) \\ \exp \big( \frac{2i\pi}{h}( \frac{\tan}{2} m^2 + s' m \big) \big)  \Big| \le C_{24}.
\end{multline*}
où $s' = b' - \tan \theta a'$, $(a',b') = M^n(a,b)$ et $d = d_{S^1}(R^m_{\alpha}(s'),s_0)$.
\end{proposition}
}}

On aura remplacé le terme $\exp(\frac{-\pi}{h} A_3(q,p))$ par $G  \left(  \frac{m}{\sqrt{h}\lambda^n} \right)$ par un raisonnement analogue en tous points à celui de la fin du paragraphe \ref{subsec:po00}, où  $\alpha(n)(q \tan +iq)^2$ joue le rôle de $A_3(q,p)$.

\

\section[Interprétation dynamique des interférences : sommes de Birkhoff] {Interprétation dynamique du calcul d'interférence : sommes de Birkhoff pour des applications fibrées sur le tore.}

Dans cette partie on montre que les termes d'interférence calculés ci-dessus sont des sommes de Birkhoff (amorties) pour un certain système dynamique parabolique.

\subsection{Applications fibrées sur le tore}

On définit $T_h$ par 
%\begin{array}{ccccc}
\begin{align*}
T_h  \colon  \mathbb{T}^2 & \longrightarrow  \mathbb{T}^2 \\
 (x,y) & \longmapsto  \begin{pmatrix}
1 & 0 \\
N & 1
\end{pmatrix}(x,y) + (\alpha, \beta)
\end{align*}
pour des paramètres $\alpha, \beta \in S^1$ et $N \in \mathbb{Z}$. On choisit $\alpha = \tan \theta$. On peut calculer 

$$T_h^m(x,y) = (x + m \alpha, y + \frac{m(m-1)}{2}N\alpha + mNx + m\beta).$$ En choisissant $ \beta = \frac{\alpha N}{2}$ on obtient

\begin{equation}
\label{eq:iteres}
T_h^m(x,y) = (x + m \alpha, y + \frac{m^2}{2}N\alpha + mNx)
\end{equation}

\subsection{Sommes de Birkhoff (amorties)}

Soit $\chi : \mathbb{R} \longrightarrow \mathbb{R}$ et $f : \mathbb{T}^2 \longrightarrow \mathbb{C}$. On définit l'objet suivant, qu'on appelle \textit{somme de Birkhoff amortie}.

\begin{definition}
La somme de Birkhoff amortie, au temps $m$, de l'observable $f$, de fonction d'amortissement $\chi$, est le nombre

$$ \mathrm{S}^{\chi}_m(f)(x,y) := \sum_{k \in \mathbb{Z}}{\chi(\frac{k}{m}) f\circ T_h^k(x,y)}.$$

\end{definition}

Notons que quand $\chi$ est l'indicatrice de l'intervalle $[0,1]$, on retrouve les sommes de Birkhoff classiques. 

\fbox{
\parbox{\textwidth}{
\begin{theoreme}
\label{thm:principal}
Soit $\Phi^h_{(a,b)}$ et $\Phi^h_{(p_0,q_0)}$ deux paquets d'ondes (qu'on considère comme des éléments de $\mathcal{H}^h$). On a alors 
\begin{equation}
\langle (\widehat{M^h})^n \cdot \Phi^h(a,b), \Phi^h_{(p_0,q_0)} \rangle = \cdot \frac{1}{\sqrt{\lambda^n}}\mathrm{S}^{\chi}_M(f_{q_0,p_0})(s(a',b'),0) + R(h,n)
\end{equation} où 
\begin{itemize}
\item $M = \sqrt{h}\lambda^n$;
\item $f_{q_0,p_0}$ est une observable $\mathbb{T}^2 \longrightarrow \mathbb{C}$ donnée explicitement ci-après;
\item $\chi$ est une fonction d'amortissement donnée explicitement ci-après;
\item $s(a,b) = b - \tan \theta \cdot a$;
\item $(a',b') = M^n \cdot (a,b)$;
\item $R(h,n) : \mathbb{T}^2 \longrightarrow \mathbb{C} $ est un reste satisfaisant  
$$ ||R(h,n)||_{\infty} \leq D_k h^k + \sqrt{h} \lambda^{-\frac{n}{2}}.$$
\end{itemize}
On a $$f_{q_0,p_0}(x,y) = F_0(\frac{d_{S^1}(x,s_0)}{\sqrt{h}}) e^{\frac{2i\pi}{h} q_0 d_{S^1}(x,s_0)} e^{2i\pi y}$$ où $F_0 : \mathbb{R} \longrightarrow \mathbb{R}$ une fonction analytique et 
$$\chi(u) = \exp(-\gamma u^2) $$ où $\gamma$ est un nombre complexe de partie réelle strictement positive.
\end{theoreme}
}}

\vspace{2mm}

\begin{proof}

Il suffit d'injecter dans la formule $$\mathrm{S}^{\chi}_m(f)(x,y) := \sum_{k \in \mathbb{Z}}{\chi(\frac{k}{m}) f\circ T_h^k(x,y)}$$ la valeur $T_h^m(x,y) = (x + m \alpha, y + \frac{m^2}{2}N\alpha + mNx)$ et de comparer avec la proposition~\ref{prop:interference}, en se rappelant que  $C_h$ est de l'ordre de $h^{-\frac{1}{4}}$.

\end{proof}

\chapter{La méthode Marklof}
\label{chap:marklof}

\section{Introduction}
\label{sec:intro}

 Le but de ce chapitre est de décrire le comportement des sommes de Birkhoff (lissées) du système dynamique parabolique mis en évidence dans le chapitre \ref{chap:propagation},  pour une certaine classe d'observables. On rappelle que ces sommes de Birkhoff sont apparues naturellement dans le cadre d'un calcul d'interférences. 
 
\vspace{2mm}

\subsection{Présentation du problème}
\label{subsec:presentationdupb}

 Le système dynamique dont on veut faire l'étude est le suivant. On se donne $\alpha > 0$ un nombre quadratique et $h > 0$ tel que $h^{-1} \in \NN$ et on définit 
	$$ \fonction{ U = U_{h}}{\TT^2}{\TT^2}{\begin{pmatrix}
	q \\ p 
\end{pmatrix}	 }{ \begin{pmatrix}
	1 & 0 \\ h^{-1} & 1 
\end{pmatrix} \begin{pmatrix}
q \\ p
\end{pmatrix} + \begin{pmatrix}
\alpha \\ \frac{\alpha}{2h}
\end{pmatrix}	\ .  } $$

Une récurrence rapide nous donne que pour tout $(q,p) \in \TT^2$ et pour tout $k \in \NN$ on a 
    $$ U_h^k(q,p) =  \left( q + k \alpha \ , \ p + \frac{\alpha k^2}{2h}  + \frac{k q}{h} \right) \ . $$

On veut étudier des versions lissées des sommes de Birkhoff de ce système dynamique pour une certaine classe d'observable : on se donne $\chi : \RR \to \CC$ une fonction de lissage à décroissance rapide et $F : \TT^2 \to \CC$ une observable et on définit 
$$ S_{\chi,F}(y,q,p,h) := \somme{k \in \NN} \ \chi \left( \frac{k}{y} \right) F_h \circ U_{h}^k(q,p) \ . $$

Notons que le paramètre $y$ a pour but de 'sélectionner' un certain nombre de termes de la somme de Birkhoff. Dans notre cas $y$ est, en gros, la longueur de la variété instable après $n$ itération de la dynamique du chat quantique, c'est à dire 
    $$ y \sim h^{1/2} \cdot \lambda^n $$
où $\lambda$ est la valeur propre associé à la direction instable de l'application du chat.  \\

Notons aussi ici que le système dynamique ainsi que l'observable dépendent de $h$. En pratique, l'observable $F$ qui apparaît naturellement dans le chapitre \ref{chap:propagation} (théorème \ref{thm:principal}) a une forme produit particulière :
	$$ F_h(q,p) := \varphi_h(q) \cdot e^{2i \pi p} \ , $$
où $\varphi_h(q)$ est une fonction indicatrice lissée dont on peut voir le graphe ci-dessous. On ne rappelle pas ici la forme précise de l'observable car on ne l'utilisera pas vraiment (gardons cependant à l'esprit qu'elle dépend de $h$).

\remark{Les paramètre $(q_0,p_0)$ du centre du paquet d'ondes ne paraissent pas intervenir dans le problème mais il n'en est rien : ils sont cachés dans l'observable $F_h$.} \\

\begin{figure}
\label{fig:transverse0}	
\begin{center}
\def\svgwidth{0.6 \columnwidth}
\includegraphics[scale=1]{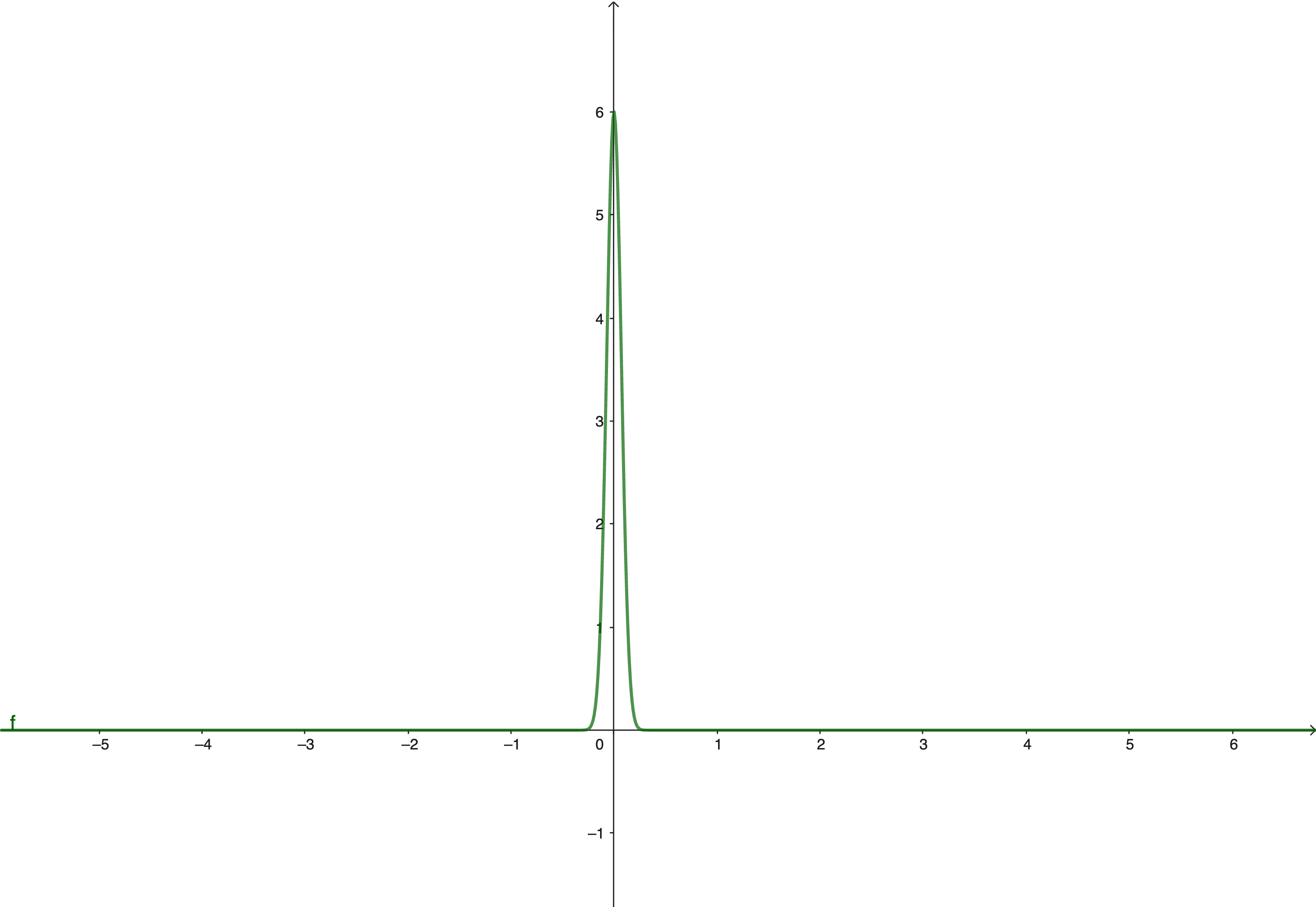}
\end{center}
\caption{Le graphe de la fonction $\varphi_h$ pour $h \sim 0.1$}
\end{figure}

\begin{remark} Notons que l'observable $F_h$ est de moyenne nulle, on cherche donc moralement le deuxième ordre de la théorie ergodique de $U_h$. Moralement seulement parce qu'ici la dépendance en $h$ de l'observable et du système dynamique ne nous permet pas de penser comme ça. On verra par ailleurs que la démonstration de notre théorème principal n'a pas grand chose de dynamique, tout du moins formellement, et s'appuie essentiellement sur des idées qui viennent de la théorie des nombres. 
\end{remark}

Il y a deux difficultés dans ce problème. La première nous vient de la nature du système dynamique en soi. En effet, on peut remarquer que la somme comporte un terme de la forme 
$ \exp \left( 2 i \pi \left( \frac{\alpha k^2}{2h} \right) \right) $ ce qui rend le problème au moins aussi compliqué que celui correspondant aux sommes de Gauss, c'est à dire que celui de l'étude quand $N \to \infty$ des sommes 
    $$ \somme{ 0 \le k \le N} e^{i \alpha k^2} \ . $$
L'étude des sommes de Gauss est un problème ancien qui vient de la théorie des nombres. L'étude de ces sommes est maintenant bien comprise grâce à la machinerie des fonctions $\theta$. La méthode et les outils de cette étude étant nécessaire pour la suite et pas évident, on a décider de dédier la section \ref{sec:casecole} de cet article à ce cas bien connu. Nos introduirons la plupart des techniques importantes dans cette section. Le point clé étant que ces fonctions sont invariantes sous l'action d'un groupe (ce qui peut faire penser à un argument de renormalisation). \\

Le second problème est lié à l'étude de ces sommes de Birkhoff concerne la dépendance en $h$ du système dynamique et de l'observable (à travers son premier facteur pour cette dernière). Cette dépendance n'est pas du tout anodine: notons en particulier que la norme $\mathcal{C}^1$ de l'observable explose quand $h \to 0$. \\

De manière surprenante, nous allons traiter ces deux difficultés, de nature pourtant différentes, de manières simultanées et avec un même outil, toujours les fonctions $\theta$, pour obtenir le résultat suivant.

\begin{theoreme}[Théorème principal]
\label{maintheo}
    Sous une certaine paire de conditions sur $h$ et $y$ il existe une constante $C > 0$ telle que pour tout $(q,p) \in \TT^2$ on a 
    $$  \Big| S_{\chi,F}(y,q,p,h) \Big| \le C \cdot h^{1/4} y^{1/2} $$
\end{theoreme}

Notons que la borne fait intervenir $h$ et $y$. Le produit $h^{1/4} y^{1/2} = (h^{1/2} y)^{1/2}$ peut être pensé comme la racine du nombre de termes de la somme 'qui comptent': on part de $y$ termes desquels on ne garde qu'une proportion $h^{1/2}$ correspondant à la sélection opérée par l'observable $\varphi_h$. On pourrait alors avoir envie de penser que c'est une réalisation d'une somme de variables aléatoires indépendantes mais ce n'est pas le cas. En effet pour une somme de variables aléatoires i.i.d., le théorème central limite s'applique et la quantité 
 $$
  \Big| S_{\chi,F}(y,q,p,h) \Big| \cdot h^{-1/4} y^{-1/2} 
  $$
ne peut pas être bornée car elle suit une distribution gaussienne (en tant que fonction de la variable $h^{-1/4} y^{-1/2}$). \\

On donnera une idée de la démonstration de ce théorème dans la section \ref{sec:interlude}, mais d'abord parlons du cas particulier, connu mais déjà intéressant, des sommes de Gauss. \\

Notons que l'on doit à Jens Marklof et ses co-auteurs (\cite{Marklof} \cite{MarklofWelsh}, \cite{MarklofWelsh2}) l'étude des sommes de Gauss généralisées par le lien qu'elles entretiennent avec la représentation metaplectique. De ce point de vue, cet article ne contient rien de nouveau et s'appuie essentiellement sur ces travaux.

\section{Le cas d'école.}

\subsection{Les sommes de Gauss et la représentation métaplectique.}
\label{sec:casecole} 
Cette section est dédiée à l'étude des sommes de Gauss, définies ci-dessous ; 
\begin{equation}
	\label{eq:sommesgauss}
	S_\chi(x,y) := \somme{ k \in \ZZ} \ \chi \left( \frac{k}{y} \right) e^{ 2 i \pi  x k^2} \ ,
\end{equation}
où $\chi$ est une fonction lisse (potentiellement à valeur complexe) à décroissance superexponentielle. Le terme $\chi(k /y)$ joue ici le rôle d'une fonction indicatrice lissée : lorsque $\chi = \mathds{1}_{[-1,1]}$ la somme est finie est ne porte que sur les entiers de $[-y, y]$. \\

Notons également que, toujours dans ce cas, si $x$ est entier la somme ci-dessus vaut exactement $2 [y] +1$. Si $x$ est un demi entier, celle-ci oscille entre $1$ et $-1$ en fonction de la parité de $N$ (les termes correspondants à des entiers paires valent 1 alors que les autres valent $-1$). Le comportement de cette somme quand $y \to \infty$ semble donc fortement dépendre du choix de la valeur de $x$. On va s'intéresser ici au cas où $x$ est un nombre quadratique, par exemple le nombre d'or. Plus précisément, on aimerait obtenir le taux de croissance polynomiale d'une telle somme, en démontrant par exemple l'énoncé suivant. 

\begin{theoreme}
\label{theo:sommesgauss}
Pour toute indicatrice lissée $\chi$, il existe une constante $C > 0$ telle que pour tout $y \ge 1$ 
	$$ \big| S_\chi(x,y) \big| \le C \cdot y^{1/2} \ .$$
\end{theoreme}

Pour démontrer ce théorème on va réécrire la somme $S_\chi(x,y)$ en terme metaplécticoreprésentatoire. On reviendra dans la section \ref{sec:groupedeliemetaplectique} sur la définition générale de la représentation métaplectique. Afin d'épurer au maximum le schéma de la démonstration on va partiellement la redéfinir dans le cas de $\mathrm{Sl}_2(\RR)$. Si $\chi : \RR \to \CC$ est une fonction à décroissance rapide et si 
$$ A := \begin{pmatrix}
    a & b \\ c & d 
\end{pmatrix} $$ 
est une matrice de $\mathrm{Sl}_2(\RR)$ alors on définit 
   $$ \fonction{\widehat{A}(\chi)}{\RR}{\CC}{t}{ a^{\frac{-1}{2}} \inte{\RR} e^{2i \pi S_A(t, \xi)  } \ \hat{\chi}(\xi) \ d \xi } \ , $$
 où  $S(t, \xi) := \frac{c t^2 - 2 \xi t + b\xi^2 }{2a} $ et où $\widehat{\chi}$ est la transformée de Fourier de la fonction $\chi$. \\

Le premier fait remarquable sur lequel on va s'appuyer est que la formule précédente définie bien une représentation de $\mathrm{Sl}_2(\RR)$ dans les applications unitaires de $L_2(\RR)$. On réfère à \cite{Folland} pour une démonstration. On utilise cette représentation pour réécrire les sommes de Gauss en commençans par une paire de calculs. Soit $y > 0$, on a 
\begin{align*}
	   \widehat{\begin{pmatrix}
y & 0 \\ 0 & y^{-1}
\end{pmatrix}} (f)(t) & =  \frac{1}{y^{\frac{1}{2}}} \inte{\RR} e^{\frac{i \pi}{y}  2 t \xi } \ \hat{f}(\xi) \ d \xi    \\
& = \frac{1}{y^{\frac{1}{2}}}  \widehat{\hat{f}}^{-1} \left( \frac{t}{y} \right) \\
& =  \frac{1}{y^{\frac{1}{2}}} 
f \left( \frac{t}{y} \right) 
\end{align*}

Soit maintenant $x \in \RR$, on calcul cette fois

\begin{align*}
	   \widehat{\begin{pmatrix}
1 & 0 \\ x & 1
\end{pmatrix}} (f)(t) & =  \inte{\RR} e^{i \pi \left( x t^2 + 2 t \xi \right) } \ \hat{f}(\xi) \ d \xi    \\
& = e^{i \pi x t^2} \inte{\RR} e^{ 2 i \pi t \xi } \ \hat{f}(\xi) \ d \xi    \\
& = e^{i \pi x t^2} f(t) \ .  
\end{align*}

On combinant les deux calcul précédent on obtient,

\begin{align*}
	   \widehat{\begin{pmatrix}
1 & 0 \\ x & 1
\end{pmatrix} \cdot \begin{pmatrix}
y & 0 \\ 0 & y^{-1}
\end{pmatrix}} (f)(t) & = \frac{e^{i \pi x t^2} }{y^{\frac{1}{2}}} \ f \left( \frac{t}{y} \right)  
\end{align*}

On introduit pour finir la forme linéaire 'somme sur les entiers relatifs', 
$$ SE := f \mapsto \somme{k \in \ZZ} \ f(k) \ , $$
qui nous permet d'écrire 
	$$  y^{\frac{1}{2}} \cdot S_\chi(x,y) = SE \left(\widehat{\begin{pmatrix}
1 & 0 \\ x & 1
\end{pmatrix} \cdot \begin{pmatrix}
y & 0 \\ 0 & y^{-1}
\end{pmatrix}} (\chi) \right) \ . $$
Les conclusions de la proposition \ref{theo:sommesgauss} peuvent se relire alors comme l'existe d'une constante $C > 0$ telle que pour tout $y > 0$
\begin{equation}
	\label{eq:reformulationsl2}
 \Big| SE \left(\widehat{\begin{pmatrix}
1 & 0 \\ x & 1
\end{pmatrix} \cdot \begin{pmatrix}
y & 0 \\ 0 & y^{-1}
\end{pmatrix}} (\chi) \right) \Big| \le C \ . 
\end{equation}

\subsection{Invariance sous l'action d'un réseau.} On se donne une fonction $f$ à décroissance rapide. 
Pour démontrer l'inégalité ci-dessus, on va montrer une invariance de $S_f$ sous l'action (d'un revêtement) du groupe modulaire. On note $\Gamma(n)$ le sous groupe de $\SL_2(\ZZ)$ constitué des matrices dont la réduction des coefficients modulo $n$ est l'identité.

\begin{lemma}
	\label{lemma:generalisationsommatoire}
	Pour toute matrice $M \in \Gamma(2)$ on a 
		$$  SE(f) = SE \left( \widehat{M} (f) \right) \ . $$
\end{lemma}

Avant de démontrer le lemme ci-dessus, voyons comment en déduire les conclusions du théorème \ref{theo:sommesgauss}. Commençons par revenir sur l'inégalité reformulée \eqref{eq:reformulationsl2} que l'on  peut relire comme suit. L'application qui, à une matrice qui s'écrit comme le produit de \eqref{eq:reformulationsl2}, associe le point de coordonnée $(x,y)$ du demi plan supérieur vient en fait de l'identification de $\SL_2(\RR)$ avec $T^1 \HH^2$, le fibré unitaire tangent de $\HH^2$. L'invariance donnée par le lemme \ref{lemma:generalisationsommatoire} peut donc se lire comme le fait que la fonction 
	$$ \fonction{\widetilde{S_f}}{\quotientg{\Gamma(2)}{\SL_2(\RR)}}{\RR}{M}{ SE(\widehat{M} \cdot f)}   $$
est bien définie. L'inégalité \eqref{eq:reformulationsl2} devient alors une conséquence du fait suivant : si $x$ est un nombre quadratique, par exemple le nombre d'or, le sous ensemble donné par (les classes) des matrices
$$ \left\{ \widehat{\begin{pmatrix}
1 & 0 \\ x & 1
\end{pmatrix} \cdot \begin{pmatrix}
y & 0 \\ 0 & y^{-1}
\end{pmatrix}} \ , \ y > 0 \right\} $$
est contenue dans un compact de $\quotientg{\Gamma(2)}{\SL_2(\RR)}$. C'est une conséquence du fait que l'ensemble précédent s'accumule, quand $y \to 0$, sur une géodésique périodique. \\

\textbf{Démonstration du lemme \ref{lemma:generalisationsommatoire}.} Commne nous savons déjà que c'est une représentation, il suffit de montrer l'invariance pour les deux matrices 
	$$ \begin{pmatrix}
	1 & 0 \\ 2 & 1
	\end{pmatrix} \hspace{0.2cm}  \begin{pmatrix}
	1 & 2 \\ 0 & 1
	\end{pmatrix} \  $$
qui engendrent bien $\Gamma(2)$. Commençons par la matrice de gauche pour laquelle on a déjà fait le calcul de la représentation métapléctique : 
$$ \widehat{\begin{pmatrix}
	1 & 0 \\ 2 & 1
	\end{pmatrix}} (f)(t) = e^{ 2 i \pi t^2} f(t) $$
En particulier pour tout $k \in \ZZ$ on a 
	$$ \widehat{\begin{pmatrix}
	1 & 0 \\ 2 & 1
	\end{pmatrix}} (f)(k)  = f(k) \ ,$$
ce qui montre bien l'invariance (de chacun des termes de la somme). \\

 Pour la seconde matrice, c'est une conséquence de la formule sommatoire de Poisson : si $f$ est une fonction lisse à bonne décroissance alors 
	\begin{equation}
		\label{eq:formulesommatoiredepoisson}
			\somme{k \in \ZZ} \ f(k) = \somme{k \in \ZZ} \  \hat{f}(k) \ .
	\end{equation}	
	
\begin{remark}
On peut en faire réécrire l'égalité ci-dessus comme suit
$$ SE(f) = SE(\hat{f}) = SE \left( \widehat{\begin{pmatrix}
0 & -1 \\ 1 & 0
\end{pmatrix}} (f) \right) \ , $$ 
car la représentation métaplectique de la rotation d'angle $\pi/2$ correspond à la transformée de Fourier. 
\end{remark}

La démonstration de la formule sommatoire de Poisson est élémentaire et consiste à considérer la fonction 
	$$ \tilde{f} := x \mapsto \somme{k \in \ZZ} \ f(x +k) \ , $$
définie sur le cercle $\quotient{\RR}{\ZZ}$. On peut alors décomposer cette fonction en série de Fourier 
	$$ \tilde{f}(x) = \somme{k \in \ZZ} c_k(\tilde{f}) \ e^{i n x} \ , $$
où la convergence se fait en norme sup dès lors que $f$ est assez régulière (et décroit assez vite). On évalue ensuite cette égalité en $0$. le membre de gauche correspond à celui de \eqref{eq:formulesommatoiredepoisson} et celui de droite correspond lui à celui de \eqref{eq:formulesommatoiredepoisson} : une permutation de somme/intégrale donne $c_k(\tilde{f}) = \hat{f}(k)$. \\

Revenons à l'invariance selon la deuxième matrice. Par définition on a 
\begin{align*}
 \widehat{\begin{pmatrix}
	1 & 2 \\ 0 & 1
	\end{pmatrix}} (f)(t)  & = \inte{\RR} e^{i \pi (- 2 \xi^2 + 2 t \xi)} f(\xi) \ d \xi \\ 
	& = \inte{\RR} \left( e^{ - 4 i \pi \xi^2}  f(\xi) \right)  \ e^{ 2i \pi t \xi} \ d \xi \\ 
	& = \widehat{\left( e^{ - 4 i \pi \xi^2}  f(\xi) \right) }(- t) 
\end{align*}

On utilise alors la formule sommatoire de Poisson pour obtenir 
\begin{align*}
 \somme{k \in \NN} \widehat{\begin{pmatrix}
	1 & 2 \\ 0 & 1
	\end{pmatrix}} (f)(k)  & =  \somme{k \in \NN} \widehat{\left( e^{ - 4 i \pi  \xi^2}  f(\xi) \right) }(- k)  \\
	& =  \somme{k \in \NN}  e^{ -4 i \pi k^2}  f(-k) \\
	& =  \somme{k \in \NN} f(-k) 
\end{align*}
car $k^2$ est un entier. On conclue en échangeant la variable $k$ pour $-k$. \hfill $\blacksquare$  

\section[Retour sur le théorème principal]{Retour sur le théorème \ref{maintheo}.}
\label{sec:interlude}

Cette section est dédiée à présenter la démonstration du théorème \ref{maintheo}. On y décrit les deux étapes principales, qui sont les analogues des deux étapes discutées dans la section précédente pour démontrer le théorème \ref{theo:sommesgauss}. \\ 

Rappelons la forme des sommes dont on veut faire l'étude en commençant par réintroduire le système dynamique
	$$ \fonction{U_{h}}{\TT^2}{\TT^2}{\begin{pmatrix}
	q \\ p 
\end{pmatrix}	 }{ \begin{pmatrix}
	1 & 0 \\ h^{-1} & 1 
\end{pmatrix} \begin{pmatrix}
q \\ p
\end{pmatrix} + \begin{pmatrix}
\alpha \\ \frac{\alpha}{2h}
\end{pmatrix}	\ .  } $$

Rappelons que $\chi : \RR \to \CC$ est une fonction poids et que l'observable à la forme suivante 
    $$\fonction{F_h}{\TT^2}{\CC}{(q,p)}{\varphi_h(q) \cdot e(p)} \ , $$
\textbf{où l'on a noté $e(p) :=  e^{2i \pi p}$, notation que l'on conservera pour la suite de l'article.} On veut donc étudier les sommes suivantes 
\begin{align*}
    S(y,q,p,h) & := \somme{k \in \NN} \ \chi \left( \frac{k}{y} \right) F_h \circ U_{h}^k(q,p) \\ 
     & = \somme{k \in \ZZ} \ \chi \left(\frac{k}{y} \right) \ e \left( p + \frac{k^2}{2h}\alpha + \frac{kq}{h} \right) \ \varphi_h( \alpha k + q) \ . 
\end{align*} 

On a ici deux difficultés de plus à traiter par rapport au cas de la section précédente.
\begin{itemize}
	\item On a deux paramètres : $y$ qui, moralement, nous donne le nombre de terme de la somme et $h$, le paramètre de l'observable.
	\item L'observable est plus compliquée, du fait de la présence de $\varphi_h$ (qui en plus dépend de $h$). 
\end{itemize}

Afin de traiter la forme plus compliquée de l'observable et l'apparition de ce paramètre $h$ on va passer en dimension 4 (à la place des deux dimensions de la section précédente). \\

En pratique la démarche est la suivante. On va montrer qu'il existe un groupe de Lie $G$ et une représentation $\rho$ de ce groupe de Lie dans les applications unitaires de $L^2(\RR^2)$ telle que pour tout jeu de paramètres définissant $S_{\chi}(y,q,p,h)$ il existe un élément $g \in G$ \textbf{(qui dépend donc de $(y,q,p,h)$)} et une fonction $\psi : \RR^2 \to \CC$ \textbf{(qui ne dépend pas de $h$)} tels que 
	\begin{equation}
	\label{eq:lienfonctionthetasommebibi}
		 S_{\chi}(y,q,p,h) = \somme{k \in \ZZ^2} \rho(g)(\psi)(k) \ . 
	\end{equation}

A ce stade, on ne parle que d'une réécriture qui pourrait être complètement creuse. C'est la proposition suivante, analogue à \ref{lemma:generalisationsommatoire}, qui vient lui donner une utilité claire.

\begin{proposition}
\label{prop.invariancefonctiontheta}
Il existe un réseau $\Gamma$ de $G$ tel que pour toute fonction $\psi : \RR^2 \to \CC$ et tout $\gamma \in \Gamma$ 
	$$ \somme{k \in \ZZ^2} \rho(\gamma)(\psi)(k) = \somme{k \in \ZZ^2} \psi(k) \ .$$
\end{proposition}

La prochaine section est dédiée à la définition et l'étude de la représentation $\rho$ et la suivante \ref{sec:liensommedebibi} au lien qu'elle entretien avec les sommes que l'on veut étudier. La récolte se fera dans la section \ref{sec:dynamiquehomogene} avec la démonstration du théorème \ref{maintheo}. 

\section{Le groupe de Lie et sa représentation}
\label{sec:groupedeliemetaplectique}

\subsection{Un peu d'algèbre linéaire (quantique)} Dans cette section nous introduisons le groupe qui va jouer le rôle analogue à celui de $\mathrm{SL_2}(\RR)$ croisé dans la section précédente ainsi que la représentation métaplectique qui lui est associée. Ce groupe est
	$$ G := \mathrm{Sp}_4(\RR) \ltimes \HH_5 $$ 
où 
\begin{itemize}
\item le groupe $\mathrm{Sp}_4(\RR)$ est le groupe symplectique de dimension 4. C'est à dire le groupe des applications linéaires de $\RR^4$ qui laisse invariant la structure symplectique standard de $T \RR^2 = \RR^4$; 
\item $\HH_5$ est le groupe d'Heisenberg de dimension 5;
\item le produit semi-direct est construit \textit{via} l'action naturelle de $\mathrm{Sp}_4(\RR)$ sur $\HH_5$, nous y reviendrons.
\end{itemize}  

On va commencer par décrire l'action de $\HH_5$ sur les fonctions, cette partie correspond à l'action quantique des translations. Rappelons qu'une translation dans l'espace des
\begin{enumerate}
	 \item  positions agit par translation sur la fonction : si $Q_0 \in \RR^2$ et si $f : \RR^2 \to \CC$ alors
	$$ \tau_{(Q_0,0)}(f)(Q) = f(Q - Q_0) \ ; $$
	\item des vitesses agit par translation sur la fonction dans l'espace des vitesses : si $Q_0 \in \RR^2$ et si $f : \RR^2 \to \CC$ alors
	$$ \tau_{(0,P_0)}(f)(Q) = e^{2i \pi  P_0 \cdot Q} \cdot f(Q) \ .$$
\end{enumerate}
Notons que les inverses de deux opérateur définis ci-dessus sont respectivement $\tau_{(-P_0,0)}$ et $\tau_{(0,-P_0)}$. Les deux opérateurs ci-dessus ne commutent pas, on peut en effet calculer que

\begin{align*}
\tau_{(P_0,0)} \circ \tau_{(0,Q_0)} \circ \tau_{(-P_0,0)} \circ \tau_{(0,-Q_0)} = \exp \left( 2 i \pi  \omega \left( \begin{pmatrix}
Q_0 \\ 0 
\end{pmatrix} , \begin{pmatrix}
0 \\ P_0  \end{pmatrix} \right) \right) \mathrm{Id}   \ ,
\end{align*}
où $\omega$ est la forme symplectique standard sur $T \RR^2 = \RR^2 \times \RR^2$. La relation ci-dessus est exactement la relation qui définit le groupe de Heisenberg $\HH_5$, que l'on peut décrire comme l'ensemble $\RR^2 \times \RR^2 \times \RR$ muni de la loi 
	$$ (Q,P,t) \cdot (Q',P',t') = (Q+Q', P+ P', t + t' + \omega \left( \begin{pmatrix}
Q \\ P 
\end{pmatrix} , \begin{pmatrix}
Q' \\ P' \end{pmatrix} \right) ) \ . $$

On notera $\tau$, par abus, la représentation induite de $\HH_5$ dans les applications unitaire de $L^2(\RR^2)$ : 

\begin{equation}
\label{eqdef.actionheisenberg}
\fonction{\tau}{\HH_5 \times L^2(\RR^2)}{L^2(\RR^2)}{(Q_0,P_0,t), f}{ \left[ Q \mapsto e^{2 i \pi \left( t - \frac{P_0 \cdot Q_0}{2} + P_0 \cdot Q \right)} \cdot f( Q - Q_0) \right] \ . }  
\end{equation}

C'est maintenant naturellement que le groupe symplectique apparaît. En effet, étant par définition, le groupe qui préserve la forme symplectique standard de $T \RR^2$, il agit naturellement sur $\HH_5$ : si $A \in \mathrm{Sp}_4(\RR)$ alors 
$$ \fonction{\Psi_A}{\HH_5}{\HH_5}{(Q,P,t)}{(A(Q,P), t)} \  $$ 
est bien un morphisme de groupe et c'est cette action qui défini le produit semi-direct qui apparaît dans la définition du groupe $G$.

\subsection{La représentation métaplectique.}

\label{subsec:metaplectique}

Cette section est dédiée à la repré-sentation du groupe $ \mathrm{Sp}_4(\RR) \ltimes \HH_5 $
dans les applications unitaires de $L^2(\RR^2)$. Notons que l'on a déjà décrit dans la section précédente l'action du groupe $\HH_5$ sur cet espace, c'est l'action par translations quantiques en vitesse et en position. \\ 

Il nous faut encore comprendre comment agit la partie linéaire $\mathrm{Sp}_4(\RR)$. On voudrait naïvement faire agir   $ \mathrm{Sp}_4(\RR) \ltimes \HH_5 $ par 
$$ (A, \nu) \cdot f := A(\gamma) \cdot f \ ,$$ mais cela ne définit pas une action. Il faudrait en fait une représentation de $\mathrm{Sp}_4(\RR)$ qui conjugue la tentative ci-dessus. En effet, supposons qu'il existe une représentation 	
	$$ \fonctionbis{\mathrm{Sp}_4(\RR)}{U(L^2(\RR^2))}{A}{\widehat{A}} $$
qui conjugue l'action de $\nu$ et celle de $A(\nu)$, c'est à dire telle que pour tout $\nu \in \HH_5$ 
	$$ \widehat{A} \left( \nu(f) \right)  = A(\nu) \left(  \widehat{A}(f)\right)  \ , $$
alors la formule suivante définit bien une action 
\begin{equation}
	\label{eqdef.representationsymplectique}
	(A, \nu) \cdot f := \nu \cdot \widehat{A}(f) \ . 
\end{equation}

En effet, on a bien
\begin{align*}
	(A', \nu') \cdot (A, \nu) \cdot f & = (A', \nu') \cdot (\nu \cdot \widehat{A}(f)) \\ 
	& = \nu' \cdot \widehat{A'}( \nu \cdot \widehat{A}(f)) \\
	& = \nu' \cdot A'(\nu) (\widehat{A'} \cdot \widehat{A}(f))  \\
	& = (A'A, \nu' A'(\nu)) \cdot f \ .
\end{align*}

Notons que l'on a pour tout $A \in \mathrm{Sp}_4(\RR)$ on a une représentation du groupe de Heisenberg $\HH_5$ : 
	$$ \fonctionbis{\HH_5 \times L^2(\RR^2)}{L^2(\RR^2)}{(\nu, f)}{A(\nu) (f)} \ .$$

Le théorème de \href{https://www.math.columbia.edu/~woit/LieGroups-2023/heisenberg.pdf}{Stone-Von Neumann} nous assure en fait que toutes les repré-sentations (unitaires et irréductibles) du groupe de Heisenberg sont conjuguées : pour tout $A \in \mathrm{Sp}_4(\RR)$ il existe une application unitaire de $L^2(\RR^2)$, que l'on notera encore $\widehat{A}$, telle que pour tout $\nu \in \HH_5$ et pour tout $f \in L^2(\RR^2)$,
\begin{equation}
	\label{eqdef.representationsymplectique}
	\widehat{A} \left( \nu(f) \right)  = A(\nu) \left(  \widehat{A}(f)\right)    \ .
\end{equation}

On a donc bien défini une action de $G$ sur $L^2(\RR^2)$ mais il nous faut encore lui donner une forme explicite et la relier aux sommes que l'on veut étudier. \\
	
On peut, comme dans le cas étudié dans la section \ref{sec:casecole}, donner une formule explicite à la représentation métapléctique. On va se placer dans le cadre d'intérêt en explicitant cette formule à travers la décomposition d'Iwasawa : toute matrice $A$ de $\mathrm{Sp}_4(\RR)$ s'écrit comme un produit 
	$$ \begin{pmatrix}
\mathrm{I}_2 & 0 \\
X & \mathrm{I}_2 
\end{pmatrix} \cdot \begin{pmatrix}
Y & 0 \\
0 & Y^{-1}
\end{pmatrix} \cdot \begin{pmatrix}
\mathrm{Re}(B) & \mathrm{Im}(B) \\ 
-\mathrm{Im}(B) & \mathrm{Re}(B) 
\end{pmatrix}  $$

avec $X$ symétrique, $S \in \mathrm{SL}(2,\mathbb{R}) $ symétrique et $B \in \mathrm{SU}(2)$.  Venons en aux formules explicites de la représentation métapléctique. Les deux lemmes ci-dessous sont les analogues de ceux présentés dans la section précédente. 

\begin{lemma}
\label{lem:formulexpliciterepmeta}
Soit $X,Y$ deux matrices symétriques avec de plus $Y$ inversible et $f : \RR^2 \to \CC$ une fonction. On a pour tout $Q \in \RR^2$
	$$ \Big( \widehat{\mathcal{Y}} \cdot f \Big) (Q) = \frac{1}{\sqrt{\det Y}} \cdot f(Y^{-1} Q) \ , $$
où $$ \mathcal{Y} := \begin{pmatrix}
	Y & 0 \\ 0 & Y^{-1} 
	\end{pmatrix} \ , $$
et 
$$ \Big( \widehat{\mathcal{X}} \cdot f \Big) (Q) = e \left( \frac{ XQ \cdot Q}{2} \right) \cdot f(Q) \ , $$
où $$ \mathcal{X} := \begin{pmatrix}
	\mathrm{Id} & 0 \\ X & \mathrm{Id}
	\end{pmatrix} \ . $$
\end{lemma}

\textbf{Démonstration.} Il faut vérifier \eqref{eqdef.representationsymplectique}. On commence par $\mathcal{Y}$. On pose $\nu := (Q_0, P_0, t)$ et on calcule le membre de gauche 
	\begin{align*}
		\widehat{\mathcal{Y}} \Big( \nu (f) \Big) (Q)  & =  \frac{ \nu(f)(Y^{-1} Q)}{\sqrt{\det Y}} \\
	  & = \frac{1}{\sqrt{\det Y}} e \left( -t - \frac{Q_0 \cdot P_0}{2} + Y^{-1} Q \cdot P_0 \right) f( Y^{-1} Q - Q_0) \ .
	\end{align*}	
On fait maintenant le membre de droite ; 	
\begin{align*}
		\mathcal{Y}(\nu)  \Big( \widehat{\mathcal{Y}} (f)  \Big) (Q)  & =  \mathcal{Y}(\nu) \Big( Q \mapsto  \frac{ f(Y^{-1} Q)}{\sqrt{\det Y}} \Big) \\				& = e \left( -t - \frac{Y Q_0 \cdot Y^{-1} P_0}{2} +  Q \cdot Y^{-1} P_0 \right)  \frac{f( Y^{-1}(Q - Y Q_0)) }{\sqrt{\det Y}}    \ ,
	\end{align*}	
ce qui égal bien le membre de gauche car $Y$ est symétrique. On fait le même calcul pour $\widehat{\mathcal{X}}$. Le membre de gauche :
	\begin{align*}
		\widehat{\mathcal{X}} \Big( \nu(f)\Big) (Q)  & = e \left( \frac{Q \cdot X Q}{2} \right) \  \nu(f) (Q) \\ 
	 & = e \left( \frac{Q \cdot X Q}{2} \right) \  \nu(f) (Q) \\
	  & = e \left( \frac{Q \cdot X Q}{2} \right) \  e \left( -t - \frac{Q_0 \cdot P_0}{2} + Q \cdot P_0 \right) f(Q - Q_0) \\
	  & = e \left( - t - \frac{Q_0 \cdot P_0}{2} + Q \cdot P_0  +  \frac{Q \cdot X Q}{2} \right) f(Q - Q_0) 
	  	\end{align*}	
Le membre de droite :
\begin{align*}
	X(\nu) \Big( \widehat{\mathcal{X}}(f) \Big) (Q) & =  e \left( -t - \frac{Q_0 \cdot (P_0 + X Q_0)}{2} + Q \cdot (P_0 + X Q_0) \right) \\
	 & \cdot e \left( \frac{(Q -Q_0)\cdot X (Q- Q_0)}{2} \right) f(Q - Q_0) \\
 & = e \left( -t - \frac{Q_0 \cdot P_0}{2} + Q \cdot P_0 \right) \cdot e \left( Q \cdot X Q_0 - \frac{Q_0 \cdot X Q_0}{2} \right)  \\
	 & \cdot e \left( \frac{(Q -Q_0)\cdot X (Q- Q_0)}{2} \right) f(Q - Q_0) \\
	 & =  e \left( - t - \frac{Q_0 \cdot P_0}{2} + Q \cdot P_0  +  \frac{Q \cdot X Q}{2} \right) f(Q - Q_0) \ ,
   	\end{align*}
ce qui conclut.  \hfill $\blacksquare$  \\

On peut en fait, grâce aux formules ci-dessus, donner une forme générale à la représentation métaplectique. Nous en aurons besoin pour démontrer l'invariance des fonctions $\theta$ le long d'un réseau de $G$.

\begin{proposition}
\label{prop.repmetaplectiqueformefourier}
 Si l'on note 
	$$ M := \begin{pmatrix}
	 A & B \\ C & D 
	\end{pmatrix} \in \mathrm{Sp}_4(\ZZ) $$
alors 
	$$ \widehat{M}(f)(x) = \det A^{-1/2} \inte{\RR^2} e \left( \frac{S(x, \xi)}{2}  \right)  \ \widehat{f}(\zeta) \ d \zeta \ ,$$
avec $$ S(x, \xi) := \scal{ x}{ C A^{-1}x} - 2 \scal{\xi}{A^{-1}x} + \scal{ \xi}{A^{-1} B \xi} \ .$$
\end{proposition}

Pour la démonstration de cette proposition, voir \cite[Theorem 4.51]{Folland}\\

\subsection{Fonction $\theta$ et invariance.}
\label{subsec:fonctiontheta}
Dans cette section, nous discutons le cadre dans lequel on va démontrer \ref{prop.invariancefonctiontheta}. \\

Pour cela, on introduit de manière similaire la fonction 'somme sur les entiers du plan'. Si $f : \RR^2 \to \CC$ est une fonction gentille on définit
	$$ SE(f) := \Big| \somme{ k \in \ZZ^2} f(k) \Big| \ . $$

 Notons qu'on a ici considéré les sommes de Gauss en module directement, \textit{a contrario} du choix fait dans la section \ref{sec:casecole}, on pourrait garder l'information de la phase mais cela alourdirait les calculs déjà pas simples. 
    
\begin{proposition}
\label{prop:existencereseauinvariancetheta}
Il existe un réseau $\Gamma$ de $ G := \mathrm{Sp}_4(\RR) \ltimes \HH_5 $ tel que pour toute fonction $f$ gentille et pour tout $ \gamma \in \Gamma$ 
	$$ SE(\gamma \cdot f) = SE(f) \ . $$
\end{proposition}

\begin{remark}
Le résultat analogue dans la section \ref{sec:casecole} ne faisait intervenir que l'invariance sous l'action d'un réseau de $\mathrm{SL}_2(\RR)$, qui correspond à la partie linéaire du produit semidirect ici. C'est essentiellement équivalent en ce qui concerne notre problème car la partie 'groupe de Heisenberg' est compacte : en particulier les trajectoires restent dans un compact si et seulement si elle reste dans un compact de la partie linéaire. 
\end{remark}

\textbf{Démonstration.} On va commencer par montrer qu'il existe un réseau (cocompact) de $\HH_5$ qui laisse invariant la fonction $\theta_f$ : si $\nu = (Q_0, P_0, t) \in \HH_5$ on calcule
\begin{align*}
	SE(\nu (f)) & = \Big| \somme{ k \in \ZZ^2} \nu(f)(k) \Big| \\
	 & = \Big| \somme{ k \in \ZZ^2} (Q_0 + k,P_0, t ) \cdot(f) (k)  \Big| \\
	 & = \Big| e \left( - t - \frac{Q_0 \cdot P_0}{2} + Q \cdot P_0 \right) \ \somme{ k \in \ZZ^2} e \left( - \frac{ k \cdot P_0}{2} \right) f( k - Q_0 ) \Big| \\
	 & = \Big| \somme{ k \in \ZZ^2} e \left( - \frac{ k \cdot P_0}{2} \right) f( k - Q_0 ) \Big| \\
\end{align*}	
Dans le cas particulier où $P_0$ est un entier relatif paire et $Q_0$ est un entier on observe 
$$ SE(\nu(f)) = SE(f) \ .$$ 
Les sommes de Gauss sont donc invariantes sous l'action du sous groupe $H(2)$ de $\HH_5$ engendré par
$$ \{ (Q_0, P_0, t) \ , \ Q_0 \in \ZZ \ , \ P_0 \in 2 \ZZ \ , \ t \in  \RR \} \ .$$
L'espace quotient est compact. Afin de conclure quant à la démonstration de \ref{prop:existencereseauinvariancetheta} on va démontrer que le groupe $S(2)$ des éléments de  $\mathrm{Sp}_4(\ZZ)$ qui sont l'identité modulo réduction des coefficients modulo 2 laisse aussi invariant les fonctions $\theta$. Combiné à l'invariance déjà obtenue pour la partie Heisenberg et à la définition de l'action du groupe $G$, cela donne l'invariance des fonctions $\theta$ sous l'action du produit semidirect $S(2) \ltimes H(2)$ qui est bien un réseau de $\mathrm{Sp}_4(\RR) \ltimes \HH_5$. \\

Pour montrer l'invariance sous l'action de $S(2)$ on utilise le même argument que celui proposé dans la section \ref{sec:casecole}. Pour cela on s'appuie sur le lemme suivant, dont on laisse la démonstration en exercice.
\begin{lemma}
	Le groupe $\mathrm{Sp}_4(\ZZ)$ est engendré par des matrices 
		$$ M := \begin{pmatrix}
	 A & B \\ C & D 
	\end{pmatrix} \in \mathrm{Sp}_4(\ZZ) $$
où $A, B, C, D$ sont des matrices carrés de taille $2$ avec $\det A = 1$. 
\end{lemma}

Revenons à la représentation métaplectique. Rappelons que l'on a 
	$$ \widehat{M}(f)(x) = \det A^{-1/2} \inte{\RR^2} e \left( \frac{S(x, \xi)}{2}  \right)  \ \widehat{f}(\zeta) \ d \zeta \ ,$$

avec $$ S(x, \xi) := \scal{ x}{ C A^{-1}x} - 2 \scal{\xi}{A^{-1}x} + \scal{ \xi}{A^{-1} B \xi} \ .$$

Si l'on suppose $A$ de déterminant $1$ on a simplement
$$ \widehat{M}(f)(x) = \inte{\RR^2} e \left( \frac{S(x, \xi)}{2}  \right)  \ \widehat{f}(\zeta) \ d \zeta \ , $$
et donc 
\begin{align*}
	SE(\widehat{M}(f)) & = \somme{k \in \ZZ^2} \inte{\RR^2} e \left( \frac{S(k, \xi)}{2}  \right)  \ \widehat{f}(\zeta) \ d \zeta 
\end{align*}	  
Rappelons que $$ S(k, \xi) :=  \scal{k}{ C A^{-1}k} - 2 \scal{\xi}{A^{-1}k} + \scal{ \xi}{A^{-1} B \xi} \ $$
Comme $A$ est de déterminant $1$, $A^{-1}$ est a coefficients entiers et donc $\scal{k}{ C A^{-1}k}$ est un entier. On conclut alors exactement comme dans la section \ref{sec:casecole} en utilisant la formule sommatoire de Poisson pour les fonctions de $\RR^2$. \hfill $\blacksquare$

\section{Lien avec les sommes de Birkhoff} \label{sec:liensommedebibi} On va maintenant utiliser les formules explicite \ref{lem:formulexpliciterepmeta} afin de réécrire les sommes de Gauss sous la forme de série : si $f : \RR^2 \to \CC$ on définit les fonctions $\theta$ comme suit 
$$ \fonction{\theta_f}{\mathrm{Sp}_4(\RR) \ltimes \HH_5}{\CC}{(A, \nu)}{ SE( (A, \nu) \cdot f)} \ . $$

Dans le cas d'intérêt, c'est à dire où 
	$$ A := \begin{pmatrix}
	\mathrm{Id} & 0 \\ X & \mathrm{Id}
	\end{pmatrix} \begin{pmatrix}
	Y & 0 \\ 0 & Y^{-1} 
	\end{pmatrix} \hspace{0.2 cm} \text{ et où } \nu := (Q_0, P_0, t) $$ 
avec $X$ et $Y$ symétrique, on notera lourdement 
 $$ \theta_{f}(X,Y,Q_0,P_0,t) \ ,$$
La fonction définie précédemment. Cette fonction ne dépend pas de $t$ (parce qu'on a pris des valeurs absolues) et on omettra de le préciser dans les arguments à partir de maintenant. La démonstration du lemme suivant est élémentaire et suit du lemme \ref{lem:formulexpliciterepmeta}.

\begin{lemma}
\label{lem:formuleexplicite1}
Avec $X$ et $Y$ comme ci-dessus, les fonctions $\theta(X,Y,Q_0,P_0)$ se mettent sous la forme suivante :
\begin{align}
     \theta_{f}(X,Y,Q_0,P_0)  =
      \det Y^{-1/2} \Big| \somme{K \in \ZZ^2} & \ e \left(\frac{\scal{K}{ X K)}}{2}  + (P_0 + XQ_0) \cdot K  \right) \\ & \cdot f(Y^{-1} (K - Q_0)) \Big| \ . 
\end{align}
\end{lemma}

Passons donc maintenant au lien avec les sommes de Birkhoff amorties. Rappelons le système dynamique $\TT^2 \to \TT^2$
\begin{equation}
\label{eq:iteres}
U_h^k(q,p) = \left( q + k \alpha, p + \frac{k^2}{2h}\alpha + \frac{kq}{h} \right) \ . 
\end{equation}

On se donne une fonction d'amortissement $\chi : \RR \to \CC$, une observable $f : \TT^2 \to \CC$, un réel $y > 0$. On veut faire l'étude des sommes suivantes,
\begin{equation}
\label{eq:sommegausslafin}
     S_f(y,p,q) := \somme{k \in \ZZ} \ \chi \left(\frac{k}{y} \right) \ f \circ U_h(p,q) \  
\end{equation}

pour les observables de la forme de celles proposées dans le chapitre \ref{chap:propagation}, c'est-à-dire avec 
$$ f(q,p)  = \varphi \left( \frac{[q - s_0]_{1/2}}{h^{1/2}} \right) \cdot e \left( \frac{q_0 [q - s_0]_{1/2}}{h} \right) \cdot e(p) \ , $$
où $\varphi : \RR \to \CC$ est une gaussienne complexe à décroissance rapide (dont la partie réelle du paramètre est négative) et où $q_0,s_0$ sont deux nombres réels. \\

Afin de mettre cette somme sous la forme d'une fonction $\theta$ on va avoir besoin de modifier un peu l'observable $\varphi$ pour une meilleure gestion de sa dépendance en $h$. \\

On notera $\Phi_h$ la fonction suivante définie sur $\RR$ :
$$ q \mapsto \somme{m \in \ZZ} \varphi \left( \frac{q - s_0 +m }{h^{1/2}}  \right) e \left( \frac{q_0  (q - s_0 + m)}{h} \right)  \ . $$
Cette fonction est $1$-périodique et donc induit bien une fonction définie sur le cercle. On notera cette nouvelle fonction $\Phi_h^{\mathbb{S}^1}$. \\

En pratique, donc, on va faire l'étude des sommes de Birkhoff associée à l'observable 
    $$ \widetilde{f}(q,p) := \Phi_h^{\mathbb{S}^1}(q) \cdot e(p) \ . $$
En effet la fonction $\varphi$ étant gaussienne, il existe une constante $C > 0$ telle que 
    $$ \Big| \left( \somme{m \in \ZZ} \varphi \left( \frac{q - s_0 +m }{h^{1/2}} \right) \right) -  \varphi \left( \frac{[q - s_0]_{1/2}}{h^{1/2}} \right)   \Big| \le C e^{ -C/h} \ . $$
    
En particulier on a que pour tout $(p,q) \in \TT^2 $ et pour tout $y > 0$
$$ \Big| S_f(y,p,q) - S_{\widetilde{f}}(y,p,q) \Big|  \le C \ y \ e^{-C/h} \ . $$
L'étude de ces deux sommes de Birkhoff est donc similaire dans le régime où $y << e^{C/h}$, c'est à dire pour des temps de la dynamique (de l'application du chat) polynomial en $h$. \\

Venons-en au lien avec les fonction $\theta$. 

\begin{proposition}
\label{prop.fonctionthetasommedebibi}
    Pour tout $y,p,q$  il existe un jeu de donnée $X,Y, P_0, Q_0$ tel que 
        $$ S_{\widetilde{f}}(y,q,p) = \theta_g(X,Y, P_0, Q_0) \ ,  $$
avec
    $$ g(q,p) := \chi(q) \cdot \widehat{\varphi}(p) \ , $$
où l'on a noté $\widehat{\varphi}$ la transformée de Fourier de la fonction $\varphi$.
\end{proposition}

\textbf{Démonstration.} On va chercher $X,Y,Q_0,P_0$ et $t$ de telle sorte à avoir 
    $$\theta_g(X,Y,Q_0, P_0) = S_{\widetilde{f}}(y,q,p) \ .$$

On pose 
    $$ X :=  \begin{pmatrix}
         \alpha/h & \alpha \\ 
         \alpha & 0
    \end{pmatrix} \ ; \ Y := \begin{pmatrix}
         y & 0 \\ 
         0 & h^{-1/2} 
    \end{pmatrix} \hspace{0.2cm} \ ; \ Q_0 =  \begin{pmatrix}
    \delta \\ 0
\end{pmatrix}. $$
On va ensuite choisir $P_0$ de tel sorte à fixer la valeur du vecteur
$$ (P_0 + X Q_0) := \begin{pmatrix}
    \beta \\ \gamma 
\end{pmatrix} \ . $$
Notons que l'on peut réaliser n'importe quel couple $(\beta, \gamma)$ en ajustant $P_0$.

On calcul alors 
    \begin{align*}
        \theta_g(X,Y,Q_0, P_0) & :=  \det Y^{-1/2} \somme{K \in \ZZ^2} \ e \left( \frac{\scal{K}{ X K}}{2} + (P_0 + X Q_0) \cdot K  \right)  g(Y^{-1} (K- Q_0)) \\
        & = y^{-1/2} h^{1/4} \somme{k,m \in \ZZ} \ e \left( \frac{\alpha k^2}{2h} + \alpha k l + \beta k + \gamma l  \right)  \chi \left(\frac{k}{y}\right) \widehat{\varphi}(h^{1/2}(l- \delta) ) \\
        & = y^{-1/2} h^{1/4} \somme{k\in \ZZ} \chi \left(\frac{k}{y}\right) e \left( \frac{\alpha k^2}{2h} + \beta k \right) \somme{l \in \ZZ} \ e \left( \alpha k l + \gamma l  \right) \widehat{\varphi}(h^{1/2}(l - \delta)) \ . 
    \end{align*}

Il nous faut reconnaître dans la somme qui porte sur la variable $l$ la décomposition en série de fourier de la fonction $\Phi_h^{\mathbb{S}^1}$ introduite plus haut. On s'appuie sur le lemme suivant dont on laisse la démonstration en exercice.

\begin{lemma}
   Le $l$-coefficient de Fourier $c_l(\Phi^{\mathbb{S}^1}_h)$ de la fonction $\Phi^{\mathbb{S}^1}_h$ vérifie 
    $$ c_l(\Phi^{\mathbb{S}^1}_h) = h^{1/2} \widehat{\varphi} \left( \frac{q_0}{h^{1/2}} +  h^{1/2}l \right) e(s_0 l) \ . $$
\end{lemma}

On fixe alors
$$ \delta := \frac{-q_0}{h} $$

et l'on réécris alors la fonction $\theta$ apparaissant plus haut comme suit 
\begin{align*}
        \theta_g(X,Y,Q_0, P_0)
           = y^{-1/2} h^{1/4} & \somme{k\in \ZZ} \chi \left(\frac{k}{y}\right) e \left( \frac{\alpha k^2}{2h} + \beta k \right) \\ & \cdot \somme{l \in \ZZ} \ e \left( \alpha k l + (\gamma -s_0) l  \right) \ h^{-1/2} c_l(\Phi^{\mathbb{S}^1}_h) \ . 
 \end{align*}

On remarque alors que
$$ \somme{l \in \ZZ} \ e \left( \alpha k l + (\gamma -s_0) l  \right) \ c_l(\Phi^{\mathbb{S}^1}_h) = \Phi^{\mathbb{S}^1}_h( \alpha k + \gamma- s_0) $$
 ce qui donne 
\begin{align*}
        \theta_g(X,Y,Q_0, P_0)
         = y^{-1/2} h^{-1/4} \somme{k\in \ZZ} \chi \left(\frac{k}{y}\right) e \left( \frac{\alpha k^2}{2h} + \beta k \right) \Phi^{\mathbb{S}^1}_h( \alpha k + \gamma- s_0) \ , 
 \end{align*}
ce qui est bien une somme de Gauss de la forme \eqref{eq:sommegausslafin}. On peut alors déterminer $\beta$ et $\gamma$ en rappelant que

\begin{align}
S_{\widetilde{f}(y,p,q)}  & = \somme{k \in \ZZ} \ \chi \left(\frac{k}{y} \right) \ \widetilde{f} \circ U_h(p,q) \\
& = \somme{k \in \ZZ} \ \chi \left(\frac{k}{y} \right) \ e \left( p + \frac{k^2}{2h}\alpha + \frac{kq}{h} \right) \Phi_h^{\mathbb{S}^1}( \alpha k + q) \\
& = e(p) \ \somme{k \in \ZZ} \ \chi \left(\frac{k}{y} \right) \ e \left( \frac{k^2}{2h}\alpha + \frac{kq}{h} \right) \Phi_h^{\mathbb{S}^1}( \alpha k + q) \ .
\end{align}

 En effet on peut prendre 
\begin{align}
    \gamma := s_0 + q \hspace{0.2cm} \text{et} \hspace{0.2cm} \beta := \frac{q}{h} \ . 
\end{align}
\hfill $\blacksquare$ 

\begin{remark}
    Notons que l'on aurait pu avoir l'égalité entre les fonctions $\theta$ et les sommes de gauss avec une définition de ces sommes sans la valeur absolue. Il faut dans ce cas aussi prendre en compte la variable $t$ pour ajuster la phase. la valeur de $t$ dépend alors de toutes les variables présentes ($h, \alpha, q, q_0, s_0$ et $p$). 
\end{remark}

\section{Dynamique homogène}
\label{sec:dynamiquehomogene}

On va maintenant pouvoir tirer parti de la relation entre les sommes de Birkhoff et les fonctions $\theta$. Rappelons que
    $$ G := \mathrm{Sp}_4(\RR) \ltimes \HH_5 $$ 
et 
    $$ H := S(2) \ltimes H(2) $$
le sous-groupe de $G$ introduit dans la section précédente (le sous-groupe dont 'la réduction modulo 2 des coefficients' est trivial). On sait que les fonctions $\theta$ sont invariantes sous l'action de $H$ à gauche, les sommes de Birkhoff associées par la proposition \ref{prop.fonctionthetasommedebibi} le sont donc aussi. Pour montrer le théorème \ref{maintheo} on va montrer que dans le régime 
    $$ y = h^{\alpha} \ , $$
(qui correspond à un certaine nombre de fois le temps d'Ehrenfest) les éléments du groupe $G$ dont les fonctions $\theta$ correspondent à $S_{\widetilde{f}}(y,q,p,h)$ reste dans un compact de $G$. Pour cela commençons par remarquer que l'application 
    $$ G \to \Sp \to \quotientg{S(2)}{\Sp} \  $$
descend bien en une application de 
    $$\quotientg{H}{G} \to \quotientg{S(2)}{\Sp} $$
qui est propre (la préimage d'un point correspond topologiquement au quotient de $\HH_5$ par $H(2)$ qui est bien compact). Afin de déterminer si une suite reste dans un compact de $\quotientg{G}{H}$ il nous suffit donc de regarder s'il reste dans un compact de $\quotientg{S(2)}{\Sp}$. Ce dernier quotient est un revêtement fini de $\quotientg{\Sp}{\mathrm{Sp}_4(\ZZ)}$ et rester dans un compact du premier est équivalent à rester dans un compact du deuxième. On s'appuie finalement sur le critère de Mahler. 

\begin{lemma}[critère de Mahler pour l'action à gauche]
Soit $\epsilon > 0$. La projection dans $\quotientg{\Sp}{\mathrm{Sp}_4(\ZZ)}$ de l'ensemble 
$$ \{ A \in \Sp \ |  \ \forall K \in \Z^4, \ || \text{transposée}(A) \cdot K || \geq \epsilon  \}$$ est compacte. 
\end{lemma}

\subsection{Démonstration du théorème \ref{maintheo}}
On finit donc par l'étude du critère de Mahler dans le cas d'intérêt de cet article. Rappelons que la partie $\Sp$ est donnée par le produit suivant 
    $$ \begin{pmatrix}
        \mathrm{I}_2 & 0 \\ X & \mathrm{I}_2 
    \end{pmatrix}  \begin{pmatrix}
        Y & 0 \\ 0 & Y^{-1}
    \end{pmatrix} =  \begin{pmatrix}
        Y & 0 \\ X Y & Y^{-1}  
    \end{pmatrix}  = \begin{pmatrix}
        y & 0 & 0 & 0 \\ 
        0 & h^{-1/2} & 0 & 0 \\
         \alpha y h^{-1} & \alpha h^{-1/2} & y^{-1} & 0 \\
          \alpha y  & 0 & 0 & h^{1/2} 
    \end{pmatrix} \ . $$

On veut utiliser le critère de Mahler. En notant $K := (k_1, k_2, k_3, k_4) \in \ZZ^4$ on veut comprendre les solutions du système d'équations suivant 
    \begin{equation}
    \label{eq:dynamiquehomogene1}
         \left\{ \begin{array}{l}
           | y k_1 + \alpha y h^{-1} k_3 + \alpha y k_4 | \le \epsilon \\
          | h^{-1/2} k_2 + \alpha h^{-1/2} k_3 | \le \epsilon \\
          y^{-1} |k_3| \le \epsilon \\
          h^{1/2} |k_4| \le \epsilon 
    \end{array} \right. 
    \end{equation}

Notons pour commencer que l'on peut réécrire la deuxième équation comme suit 
 $$ \Big| \frac{k_2}{k_3} + \alpha \Big| \le \frac{\epsilon h^{1/2}}{k_3} \ , $$
ce qui est un problème d'approximation diophantienne. Le nombre $\alpha$ étant quadratique, on a toujours l'inégalité ci-dessous 
    $$ \Big| \frac{k_2}{k_3} + \alpha \Big| \ge \frac{C}{(k_3)^2} \ , $$
pour une constante $D > 0$ bien choisie. \\

Avant de coupler l'étude de cette équation avec toutes les autres faisons une remarque en ne regardant que l'inégalité 3. Rappelons que nous travaillons dans le régime d'une constante arbitraire fois le temps de délocalisation, ce qui se traduit par le fait que 
$$ y \sim h^{- C \ , } $$ 
où $C$ dépend du 'quelque'. Dans le régime où $C < 1/2$ (juste en dessous du premier temps d'interférence) on a donc
    $$ \Big| \frac{k_2}{k_3} + \alpha \Big| \le \frac{1}{(k_3)^{2 + \delta}} \ , $$
pour un $\delta > 0$, ce qui contredit la deuxième inégalité de \eqref{eq:dynamiquehomogene1}. \\

Pour résumer, dans ce régime le critère de Mahler n'est pas vérifié, et c'est tant mieux, car on sait que le propagé du paquet d'ondes n'est pas encore délocalisé. On peut relire le couplage de 2 et 3 comme la traduction arithmétique du fait qu'il n'y a pas d'interférence sous le temps d'interférence (et donc pas de comportement chaotique des sommes interférentielles). A partir de maintenant on va se mettre dans le régime d'intérêt, c'est à dire le cas où $C > 1/2$. \\

Si l'on ajoute l'équation 1 et l'équation 4 on peut comprendre ce système d'équations comme un problème d'approximation d'un nombre quadratique sous contrainte.\\

Examinons la première équation et réécrivons là comme suit
    $$ \Big|  \frac{k_1}{h^{-1}( k_3 + k_4)} + \alpha \Big| \le \frac{\epsilon}{ y  h^{-1}( k_3 + k_4)} \ . $$
Rappelons que $h^{-1}$ est un entier, que l'on notera $N$ pour simplifier. Cette équation est donc elle aussi sous la forme 'approximation de $\alpha$ par des rationnels'. Les deux équations restantes nous donnent des bornes sur $k_3$ et $k_4$. Rappelons que l'on est dans le cas où $C > 1/2$ et on a 
\begin{align*}
 |k_3 + k_4| \le 2 h^{-C} \\
 k_3 \le h^{-C}
\end{align*}

Au final on obtient 

 $$ \Big| \frac{k_2}{k_3} + \alpha \Big| \le \frac{\epsilon h^{1/2}}{k_3} \ , $$

     $$ \Big|  \frac{k_1}{h^{-1}( k_3 + k_4)} + \alpha \Big| \le \frac{\epsilon}{ y  h^{-1}( k_3 + k_4)} \ . $$ On a donc démontré la version suivante du Théorème \ref{maintheo}.
     
\begin{theoreme}[Théorème principal]
\label{thm:principal2}
Il existe une constante $C_0$ telle que pour toute constante $C \geq C_0$, il existe un $\epsilon = \epsilon(C) > 0$ pour lequel l'énoncé suivant est vérifié. Si, pour tout  $k \leq \epsilon y$ et $k' \leq \epsilon h^{-\frac{1}{2}}$,  une des deux conditions suivantes est vérifiée

\begin{equation}\tag{A(C)}\label{eq:condAC}
\begin{cases} 
d_{S^1}(k \alpha,0) &\geq \epsilon h^{\frac{1}{2}}\ , \\
d_{S^1}((k h^{-1} +k') \alpha,0) &\geq  \dfrac{\epsilon}{y}
\end{cases}
\end{equation}

alors pour tout $(q,p) \in \TT^2$ on a 
    $$  \Big| S_{\chi,F}(y,q,p,h) \Big| \le C \cdot h^{1/4} y^{-1/2} $$
\end{theoreme}

Si $k$ est tel que les conditions sont vérifiées, on dit que $k $ vérifie $A(C)$, où $\epsilon = \epsilon(C)$.

\chapter{Applications, perspectives et commentaires}
\label{chap:applications}

Dans ce chapitre fourre-tout, on explique comment utiliser les estimations et le critère arithmétique obtenu dans le chapitre \ref{chap:marklof} pour décrire, de manière qualitative, la propagation des paquets d'ondes. Puis, en se basant sur une analogie entre les applications du chat quantiques et l'équation de Schrödinger sur les surfaces, on explique comment (conjecturalement) les résultats de ce texte devraient se généraliser au cas des surfaces riemanniennes.

\section{Lien entre le critère arithmétique et la propagation des paquets d'ondes.}

On rappelle qu'un des objectifs que nous nous sommes fixés est de comprendre la propagation \textit{après} le temps d'Ehrenfest, quand se produisent des interférences. Pour le moment nous avons :

\begin{itemize}
\item exprimé les termes d'interférence comme les sommes de Birkhoff renormalisées d'un certain système dynamique (Théorème \ref{thm:principal});

\item utilisé une méthode classique de renormalisation pour déterminer  quand ces sommes sont bornées (Théorème \ref{thm:principal2}).
\end{itemize}

\subsection{La mesure de Husimi du propagé d'un paquet d'ondes} 

On rappelle qu'une manière de formaliser la description d'une fonction d'onde est de décrire sa mesure de Husimi. On en rappelle la définition, si $f$ est une distribution de $\mathcal{H}^N$, 

$$ \mathrm{Hus}^h(f) := \frac{1}{h} | \langle f, \Phi^h_{q,p} \rangle|^2 dqdp $$ où $\Phi^h_{q,p}$ est un $h$-paquet d'ondes centré au point $(q,p) \in \mathbb{T}^2$.

\vspace{2mm} C'est une mesure de probabilité, à densité et dont la densité est donnée par le décomposition en paquets d'ondes de $f$ (cf paragraphe \ref{sec:husimi}).

\vspace{3mm}

\paragraph{\bf Ce à quoi on peut s'attendre.} Lorsqu'on propage un paquet d'ondes, l'approximation semi-classique nous assure que jusqu'à des temps de l'ordre de grandeur du temps d'Ehrenfest, la mesure de Husimi du propagé d'un paquet d'ondes est plus ou moins égale à la mesure image, par la dynamique classique, d'une mesure à densité supportée sur un voisinage de diamètre $\sqrt{h}$ du point $(q,p)$.

\begin{enumerate}
\item Ainsi, pour des temps $n$ petits, la mesure de Husimi ressemble à une masse de Dirac lissée autour de l'image de $(q,p)$ par la dynamique classique $M^n(q,p)$.

\item Vers le temps d'Ehrenfest $t_E(h) = \frac{ \log(\frac{1}{h})}{\log \lambda}$, la masse de la mesure commence à s'étaler le long de la variété instable passant par $M^n(q,p)$. Comme cette variété instable remplit de manière dense $\mathbb{T}^2$, on peut montrer qu'à ces échelles de temps la mesure de Husimi du propagé du paquet d'ondes est proche de la mesure de Lebesgue.

\item Après deux fois le temps d'Ehrenfest, la variété stable devient tellement étalée que des phénomènes \textit{d'interférences} commencent à se produire. C'est ce phénomène qu'on essaie de comprendre.
\end{enumerate}

\vspace{3mm} 

\fbox{
\parbox{\textwidth}{
\textit{En règle très générale}, on s'attend à ce qu'après le temps d'Ehrenfest les propagés de paquets d'ondes restent \textit{équidistribués}, ce qui veut dire que leur mesure de Husimi est proche d'être la mesure de Lebesgue. 
}}

\vspace{3mm}

Un moyen de démontrer des résultats allant dans ce sens est d'essayer de comprendre et contrôler la fonction de densité $| \langle f, \Phi^h_{q,p} \rangle|^2$ si $f$ est le propagé d'un paquet d'ondes. Parce que cette densité est de moyenne $1$, si on sait qu'elle est bornée \textit{uniformément, indépendamment de $h$}, cela implique immédiatement que les limites semi-classiques de mesures de Husimi de propagés de paquets d'ondes sont absolument continues.

\vspace{3mm}

\subsection{Résultats antérieurs et simulations de Frédéric Faure.} On commence par rappeler le résultat important de deBièvre-Faure-Nonnenmacher. 

\vspace{2mm}

\paragraph{\bf Les périodes courtes.} Il existe des valeurs de $h$ très particulières pour lesquelles le propagé d'un paquet d'ondes, après un temps de l'ordre de deux fois le temps d'Ehrenfest, est égal au paquet d'ondes de départ. C'est un fait remarqué par de Bièvre-Faure-Nonnenmacher, dans l'article \textit{Scarred eigenstates for quantum cat maps of minimal periods} \cite{BNF}. Ces valeurs de $h = \frac{1}{N}$ sont très rares, mais cela implique que la borne donnée par le théorème \ref{thm:principal2} ne saurait être uniforme. Pour ces valeurs de $h$ (par exemple $h = \frac{1}{76}, \frac{1}{199}, \cdots$, la distribution de Husimi du propagé d'un paquet d'ondes à $2t_E(h)$ est "presque" une masse de Dirac en un point, sa densité est une fonction de norme $\mathrm{L}^1$ égale à $1$ dont toute la masse se concentre dans une boule de volume d'ordre $h$. La densité est donc probablement de l'ordre $N = \frac{1}{h}$, et donc la borne $C$ des sommes de Birkhoff renormalisées obtenues dans le théorème \ref{thm:principal2} doit être de l'ordre de $ h^{-\frac{1}{2}} = \sqrt{N}$.

\vspace{2mm}

\paragraph{\bf Les $h = \frac{1}{N}$ génériques.} Des simulations numériques dues à Frédéric Faure suggèrent au contraire que pour des valeurs de $h$ génériques, la distribution de Husimi du propagé d'un paquet d'ondes reste équidistribuée pendant des temps assez long. On pourra consulter par exemple le lien suivant 

\begin{center}

\href{https://www-fourier.univ-grenoble-alpes.fr/faure/enseignement/meca_q/animations/node40.html}{Propagation d'un paquet d'ondes après le temps d'Ehrenfest.}\footnote{\url{https://www-fourier.univ-grenoble-alpes.fr/faure/enseignement/meca_q/animations/node40.html}}

\end{center}

De plus, on s'aperçoit que la distribution de Husimi du propagé du paquet d'ondes possède un certain nombre de symétries.  

\vspace{3mm}

\fbox{
\parbox{\textwidth}{
On remarque le fait expérimental suivant : il semblerait que la densité de la mesure de Husimi reste uniformément bornée pendant longtemps. 
}}

\vspace{2mm}

\paragraph{\bf Les chats quantiques non-linéaires.} Que la densité de la mesure de Husimi reste uniformément bornée est une manière très forte d'assurer que les mesures semi-classiques des propagés de paquets d'ondes tendent vers une mesure absolument continue par rapport à la mesure de Lebesgue. Dans le cas d'applications quantiques non-linéaires, on observe (expérimentalement) les trois faits suivants :

\begin{itemize}
\item comme dans le cas linéaire, la mesure de Husimi du propagé d'un paquet d'ondes semble être proche de la mesure de Lebesgue;

\item par contre, la densité de la mesure de Husimi semble varier plus fort et il semble possible d'imaginer que la distribution de la densité ne soit pas uniformément bornée (on pourrait imaginer que la distribution de cette densité soit une gaussienne, après tout pourquoi pas ?).

\item finalement, les symétries qui semblaient être présentent dans le cas linéaire ont complètement disparu.

\end{itemize}

Ces phénomènes sont visible dans la simulation numérique suivante, toujours due à Frédéric Faure.

\begin{center}

\href{https://www-fourier.univ-grenoble-alpes.fr/~faure/enseignement/meca_q/animations/node47.html}{Propagation d'un paquet d'ondes après le temps d'Ehrenfest, pour une application quantique non-linéaire.}\footnote{\url{https://www-fourier.univ-grenoble-alpes.fr/~faure/enseignement/meca_q/animations/node47.html}}

\end{center}

\subsection{Que nous disent la méthode Marklof et le critère arithmétique}

Passons donc à ce que disent nos résultats. On rappelle que $ \widehat{M}^h$ est l'application du chat quantique.  Si $\Phi^h_{q_0,p_0}$ est un $h$-paquet d'ondes centré en $(q_0,p_0) \in \mathbb{T}^2$ on a démontré (via un calcul d'interférence dont on n'est pas peu fier !) une formule du type 

$$  
\langle (\widehat{M}^h)^n \Phi^h_{q_0,p_0}, \Phi^h_{q_0,p_0} \rangle = \frac{1}{ \sqrt{\lambda^n}}S_{\chi,F}(\lambda^n,q',p',h). 
$$ 
c'est essentiellement le contenu de notre théorème \ref{thm:principal}. 

\vspace{2mm} 
Par ailleurs, à condition que la condition arithmétique \eqref{eq:condAC} soit vérifiée, on aurait, par le théorème \ref{thm:principal2} 
$$ 
S_{\chi,F}(\lambda^n,q',p',h) \leq C \sqrt{h \lambda^n}.
$$ 

En injectant dans l'équation du calcul d'interférence on obtiendrait,

$$ 
| \langle (\widehat{M}^h)^n \Phi^h_{q_0,p_0}, \Phi^h_{q_0,p_0} \rangle |^2 \leq C \cdot h.
$$ 
Comme la densité de la mesure de Husimi de $(\widehat{M}^h)^n \Phi^h_{q_0,p_0}$ est égale à $\frac{1}{h} | \langle (\widehat{M}^h)^n \Phi^h_{q_0,p_0}, \Phi^h_{q_0,p_0} \rangle |^2$, on obtient.

\begin{theoreme}\label{thm:husimi}
Pour les temps $n$ et les $h = \frac{1}{N}$ tels que la condition arithmétique \eqref{eq:condAC} est vérifiée, la densité de la mesure de Husimi du propagé d'un paquet d'ondes $ (\widehat{M}^h)^n \Phi^h_{q_0,p_0}$  est bornée par $C$. 
\end{theoreme}

Cela nous laisse la question 

\vspace{2mm}

\fbox{
\parbox{\textwidth}{
Fixons une constante $C > 0$. Pour quels paramètres $h = \frac{1}{N}$ la condition \eqref{eq:condAC} est-elle satisfaite pour des temps $n$ grands devant le temps d'Ehrenfest ? Pour ces $h$, les propagés de paquets d'ondes restent équidistribués pour des échelles de temps bien supérieures au temps d'Ehrenfest.
}}

\vspace{2mm} On fait quelques remarques importantes.

\begin{itemize}
\item Il faut faire attention, pour tout $N = \frac{1}{h}$, il va exister une constante $C = C(h)$ telle que la densité de la mesure de Husimi du propagé est bornée par $C(h)$ pour tout temps $n$ (c'est un exercice d'approximation diophantienne de montrer que la condition \eqref{eq:condAC} est satisfaite pour tout $n$, à $h$ fixé, pour une constante $C$ suffisamment grande).

\item Par ailleurs, étant donné $h = \frac{1}{N}$ une constante $C$ suffisamment petite (disons suffisamment plus petit que $h^{-\frac{1}{4}}$, qui est l'ordre de grandeur du supremum de la densité de la mesure de Husimi d'un paquet d'ondes), il est impossible que la densité de la mesure de Husimi soit bornée pour tout temps par $C$ après le temps d'Ehrenfest. En effet, le propagateur étant périodique, il existe un temps $n$ très grand pour lequel le propagé d'un paquet d'ondes est un paquet d'ondes, la densité de sa mesure de Husimi est donc de nouveau de l'ordre de grandeur $h^{-\frac{1}{4}} = N^{\frac{1}{4}}$. 

\item La question intéressante est donc la suivante : on fixe une constante $C > 0$  de taille macroscopique (disons $C = 10$) et on se demande si il est vrai que, pour la plupart des $h$ la densité de la mesure de Husimi reste grande pour des grands multiples du temps d'Ehrenfest. Ici, ce que veut dire "grands" reste à définir.

\item Comme on sait qu'il existe des valeurs spéciales de $h$ telles que la période du propagateur est à peu près égale à deux fois le temps d'Ehrenfest, il n'est pas possible que les propagés de paquets d'ondes soit bien équidistribués pour des grands multiples du temps d'Ehrenfest pour tout $h$.

\end{itemize}

\section{Une manière dynamique de voir le critère arithmétique}

On rappelle comment fonctionne le critère arithmétique pour les interférences constructives. Pour $\epsilon > 0$, on dit qu'il y a interférence constructive en $k$ si $k$ est le plus petit entier tel que

\begin{enumerate}

\item $d_{S^1}(k \alpha,0) \leq \epsilon h^{\frac{1}{2}}$;
\item  $d_{S^1}(k h^{-1}\alpha,0) \leq  \frac{\epsilon}{k}$;

\end{enumerate}

(C'est une version simplifiée du critère \eqref{eq:condAC}  où on a supprimé la dépendance en $k'$ pour rendre la discussion plus lisible, on la réintègrera bientôt.) 

\vspace{2mm} Un petit $\epsilon$ correspond à une grande constante $C$, un tel $k$ correspond à un premier temps $n$ auquel se produisent de grosses interférences constructives (on rappelle que la relation entre $n$ est $k$ est donnée par $k = \sqrt{h} \lambda^n$. On rappelle que 

\begin{itemize}
\item si $\epsilon$ est trop petit comparé à $h$ (si $C(\epsilon) >> h^{-\frac{1}{4}})$, un tel $k$ n'existera pas;

\item au contraire, si $C << h^{-\frac{1}{4}}$, l'existence d'un tel $k$ est garantie par la périodicité du propagateur;

\item comme on sait que la période du propagateur est plus petite que $3N$, on obtient que $k \leq \frac{\lambda^{3N}}{\sqrt{N}}$;

\item au contraire, si $h$ correspond à une période courte $2t_E(h) = 2\frac{\log N}{\log \lambda}$, on doit s'attendre à ce que la condition arithmétique soit violée pour $k \sim N^{2} =\frac{1}{h^2}$.
\end{itemize}

\subsection{Une reformulation dynamique}

Considérons l'application 

$$\begin{array}{ccccc}
T & := & \mathbb{T}^2 & \longrightarrow & \mathbb{T}^2 \\
   &  & (x,y) & \longmapsto & (x + \alpha, x + N\alpha)
\end{array}.$$

Si on considère la condition arithmétique \eqref{eq:condAC}, on peut la reformuler de la manière suivante. On définit d'abord le rectangle $R_k = [-\frac{\epsilon}{\sqrt{N}}, \frac{\epsilon}{\sqrt{N}}] \times [-\frac{\epsilon}{k}, \frac{\epsilon}{k}] $.

\vspace{3mm}

\fbox{
\parbox{\textwidth}{
\eqref{eq:condAC} au temps $n$ devient, pour tout $k \leq \epsilon n$, 
$$ T^k(0,0) \notin R_k.$$
}}

\vspace{2mm}

En d'autres termes, la condition arithmétique est violée quand l'image du point $(0,0)$ au temps $k$, par la translation $T$, tombe dans le petit rectangle $R_k$. C'est en dynamique ce qu'on peut appeler un problème de \textit{"cible rapetissante"}. 

\vspace{3mm} Pour deviner le temps auquel on peut s'attendre à ce que \eqref{eq:condAC} cesse d'être vraie, on peut utiliser l'heuristique suivante. Si on pense (et bien sûr c'est loin de marcher comme ça en vrai) que la position de $p_k := T^k(0,0)$ est un événement aléatoire, la probabilité que $p_k$ appartiennent à $R_k$ est égale à l'aire de $R_k$, qui vaut $\frac{\epsilon^2}{k\sqrt{N}}$. Ainsi la probabilité qu'il existe un $k \leq k_0$ tel que $p_k \in R_k$ est égale à 

$$ \sum_{k=1}^{k_0}{\frac{\epsilon^2}{k\sqrt{N}}} \sim \frac{\epsilon^2}{\sqrt{N}} \log(k_0).$$  Quand cette "probabilité" devient de l'ordre de $1$, on peut s'attendre à ce qu'un $k$ soit entré dans $R_k$, ce qui nous donne 

$ k_0 \sim \exp(\frac{\sqrt{N}}{\epsilon^2})$. Pour calculer ce à quoi ça correspond en terme de multiples du temps d'Ehrenfest, comme $k \sim \sqrt{h} \lambda^n$ et que le temps d'Ehrenfest est $t_E(h) = \frac{\log N}{\log \lambda}$ on trouve 

$$ n \sim \frac{\sqrt{N}}{\epsilon^2}.$$ Bien sûr tout cela est hautement non-rigoureux, mais c'est une heuristique. On en déduit la conjecture suivante.

\vspace{2mm}

\fbox{
\parbox{\textwidth}{
\begin{conjecture}
\label{conj:interf}
Fixons une constante $C > 0$. Pour une suite de $h = \frac{1}{N}$ de densité $1$, les premières interférences constructives\footnote{comprendre le premier temps $n$ auquel la densité la mesure de Husimi d'un propagé de paquet d'ondes est plus grande que $C$} d'ordre $C$ arrivent à un temps de l'ordre de $\sqrt{N} \cdot t_E(N) = h^{-\frac{1}{2}} \cdot t_E(h)$.
\end{conjecture}
}}

\vspace{2mm} Il semblerait raisonnable d'étendre une telle conjecture, précisément formulée, aux surfaces hyperboliques.

\subsection{Comparaison avec les périodes des applications quantiques}

Il est bien connu (on pourra consulter \cite{BNF}) que pour un entier $N = \frac{1}{h}$, l'application quantique $\widehat{M}^h : \mathcal{H}^h \longrightarrow \mathcal{H}^h$ est périodique, et que ça période est égale à la période de la matrice $M \in \mathrm{SL}(2, \mathbb{Z})$ agissant sur $\mathbb{Z}/N \mathbb{Z} \times \mathbb{Z}/N \mathbb{Z}$. On a le théorème suivant.

\begin{theoreme}
Pour une suite de densité $1$ de valeurs de $N = \frac{1}{h}$ la période de $\widehat{M}^h$ est supérieure à $\sqrt{N}$.
\end{theoreme}

Il est donc remarquable que l'heuristique qui conduit à la conjecture \ref{conj:interf} prédise que les interférences constructives se produisent (presque toujours) à des temps de l'ordre de grandeur de la période du propagateur.

%
%
%\section{Simulations numériques}
% 
% 
%Ici on considère la matrice $M := \begin{pmatrix}
%2 & 1 \\
%1 & 1
%\end{pmatrix}$, son quantifié $\widehat{M}^h$ pour $h = \frac{1}{N}$ variable. On a 
%
%\begin{enumerate}
%\item $\lambda = \frac{3 + \sqrt{5}}{2}$;
%\item $\alpha = \frac{1 + \sqrt{5}}{2}$.
%\end{enumerate} 
%
%
%\begin{center}
%\begin{tabular}{ |c|c|m{3cm}|m{5cm}| } 
% \hline
%$N = \frac{1}{n}$ & Période quantique & P.Q. en multiple de $t_E(N)$ & Temps de première interférence (en multiple de $t_E(N)$) \\ 
%\hline
% 20 & cell5 & cell6 & \\ 
% 21 & cell8 & cell9 & \\ 
% \hline
%\end{tabular}
%\end{center}
%

\section[Dynamique parabolique et mélange]{Dynamique parabolique et mélange, lien avec les travaux de Fred Faure}
\label{sec:persp}

\subsection{Une petite prise de recul}

On récapitule ce qu'on a fait dans ce texte. 

\begin{enumerate}
\item On a exprimé le propagateur de l'application du chat quantique en utilisant des sommes de Birkhoff d'un système dynamique parabolique vaguement lié à notre problème de départ. Plus précisément, on a exprimé les termes d'interférences qui apparaissent après le temps d'Ehrenfest de manière purement dynamique. 

\item On a montré que ces termes d'interférences peuvent être exprimés en termes d'une fonction theta définie sur un fibré en tores au-dessus de $\mathrm{Sp}(4, \mathbb{R})/\mathrm{Sp}(4, \mathbb{Z})$.
\end{enumerate}

Quand bien même nous n'avons pas réussi à extraire toutes les informations qu'on aurait souhaité de cette dernière description, cela donne un cadre qui explique bien pourquoi les mesures de Husimi de propagés de paquets d'ondes, entre le temps d'Ehrenfest et le temps de Heisenberg, doivent rester (du moins génériquement) très équidistribuées. Le fait important est le suivant. 

\vspace{3mm}

\fbox{
\parbox{\textwidth}{
On a réussi à exprimer les coefficients du propagateur $(\widehat{M}^h)^n$ au temps $n$, pour tout temps $n \in \mathbb{N}$ en fonction d'une fonction 
$$ \Theta : X \longrightarrow \mathbb{C} $$ définie explicitement sur $X$ qui fibre en $\mathbb{T}^4$ au-dessus de $\mathrm{Sp}(4, \mathbb{R})/\mathrm{Sp}(4, \mathbb{Z})$.
}}

La dépendance entre $(\widehat{M}^h)^n$ et $\Theta$ est explicite. Pour tout $n$, il existe un point $p(n)$ dans $\mathrm{Sp}(4, \mathbb{R})/\mathrm{Sp}(4, \mathbb{Z})$ tels que les coefficients de $(\widehat{M}^h)^n$ s'expriment bien en fonction de la restriction de $\Theta$ à la fibre au-dessus de $p(n)$.

\vspace{3mm} On en déduit, au moins sur le principe, que toutes les questions du chaos quantique pour les quantifiés d'applications linéaires du tores devraient se réduire à l'étude de la fonction $\Theta$ et de la suite de points $p(n)$. Ca n'est pas nécessairement une mince affaire : comme on l'a vu la compréhension de la suite $p(n)$ se réduit à l'étude de la condition arithmétique \eqref{eq:condAC}, qui pour l'instant nous échappe !

\subsection{Des questions sur les chats quantiques}

On regroupe dans ce paragraphe des questions plus ou moins explicites qui nous semblent pertinentes.

\vspace{2mm}

\fbox{
\parbox{\textwidth}{
\begin{question}
Soit $C > 1$ une constante arbitraire, que l'on pense grande. Est-il vrai qu'il existe une suite d'entiers $N = \frac{1}{h}$, de densité $1$, telle que les mesures de Husimi de propagé de $h$-paquets d'ondes, prisent entre des temps $t_E(N) \leq n \leq C \cdot t_E(N)$ convergent toutes vers la mesure de Lebesgue du tore $\mathbb{T}^2$ ?
\end{question}
}}

\vspace{2mm} D'une certaine manière, cet énoncé est une version quantique de la notion d'ergodicité. Il s'avère que la terminologie "ergodicité quantique" est déjà prise pour une autre propriété, c'est bien dommage !

\subsection{Un formalisme général pour les interférences}

On aimerait dans ce paragraphe rendre plus naturelle l'apparition du système dynamique 

$$ \fonction{ U = U_{h}}{\TT^2}{\TT^2}{\begin{pmatrix}
	q \\ p 
\end{pmatrix}	 }{ \begin{pmatrix}
	1 & 0 \\ h^{-1} & 1 
\end{pmatrix} \begin{pmatrix}
q \\ p
\end{pmatrix} + \begin{pmatrix}
\alpha \\ \frac{\alpha}{2h}
\end{pmatrix}	\ .  } $$ qui nous permet, via l'étude de ses sommes de Birkhoff, de rendre compte des termes d'interférence dans la propagation.  

\vspace{3mm}

\paragraph{\bf Flots sur une nilvariété} Il est connu (\cite{FlaminioForni2}) que cette application est le premier retour d'un nilflot sur une nilvariété de dimension $3$, c'est à dire un flot homogène sur un quotient 

$$ \mathcal{H}(\mathbb{R})/ \Gamma_N $$ où $\mathcal{H}(\mathbb{R})$ est le groupe de Heisenberg de dimension $3$, et $\Gamma_N$ est un réseau de $\mathcal{H}(\mathbb{R})$. Il n'est pas forcément nécessaire de connaitre toute la théorie des nilvariétés pour comprendre ce qui suit. Il suffit de savoir que   

\begin{enumerate}
\item pour tout $N$, $\mathcal{H}(\mathbb{R})/ \Gamma_N$ est un fibré en cercle au-dessus de $\mathbb{T}^2$ (de nombre d'Euler égal à $N$);

\item il existe un automorphisme $T_N$ de $\mathcal{H}(\mathbb{R})/ \Gamma_N$ qui préserve cette structure fibrée, et dont l'application induite sur $\mathbb{T}^2$ n'est autre que l'application de chat induite par la matrice $M \in \mathrm{SL}(2,\mathbb{Z})$.
\end{enumerate}

Cet automorphisme $T_N$ agit fibre à fibre comme une translation. Let fait important pour replacer nos travaux dans un contexte plus général est le suivant. 

\vspace{2mm}

\fbox{
\parbox{\textwidth}{
L'application $T_N$ est partiellement hyperbolique et l'application $U = U_{h}$ est juste une application de premier retour de son feuilletage stable fort.
}}

\vspace{2mm}

On présente souvent l'application du chat quantique comme une quantification d'un difféomorphisme d'Anosov du tore. Nous pensons que ceci est une erreur conceptuelle, qui empêche l'analogie naturelle avec les surfaces à courbure négative d'apparaitre. Il semble plus raisonnable de voir l'application du chat quantique comme un quantifié de l'application 

$$ T_N : \mathcal{H}(\mathbb{R})/ \Gamma_N \longrightarrow \mathcal{H}(\mathbb{R})/ \Gamma_N $$

$T_N$ devient l'analogue du temps $1$ du flot géodésique d'une surface compacte à courbure négative. 

\vspace{3mm}

\paragraph{\bf Analogie avec les surfaces négativement courbées.} 

On développe ici l'analogie formelle donnée entre les chats quantiques et les surfaces à courbure négative.

\begin{center}
\begin{tabular}{ |m{5cm}|m{5cm}|m{5cm}| } 
 \hline
 & {\bf Chat quantique} & {\bf Surface à courbure négative $\Sigma$} \\ 
\hline
 {\it Espace des phases} & Variété de Heisenberg $\mathcal{H}(\mathbb{R})/ \Gamma_N$ & Fibré unitaire tangent de $\Sigma$ \\ 
  \hline
 {\it Dynamique classique} & Automorphisme partiellement hyperbolique $T_N$ & Flot géodésique \\ 
 \hline
 {\it Structure de contact} & Forme de contact de Heisenberg &  Forme de Liouville \\ 
  \hline
 {\it Dynamique instable} & Nilflot parabolique  & Flot horocyclique  \\
 \hline
{\it Théorèmes limites pour la dynamique instable} & Travaux de Marklof \cite{Marklof} et Flaminio-Forni \cite{FlaminioForni2} & Travaux de Flaminio-Forni \cite{FlaminioForni}\\ 
 \hline
{\it Action pour les orbites périodiques} & Nombre de translation dans la fibre au-dessus d'une orbite périodique de $M$ & Longueurs des orbites périodiques \\ 
 \hline
\end{tabular}
\end{center}

\vspace{2mm}

Les résultats du présent texte suggèrent donc le fait général suivant (en poussant l'analogie)
\vspace{2mm}

\begin{center}

\fbox{
\parbox{\textwidth}{
Les termes d'interférences lors de la propagation pour l'équation de Schrôdinger sur une surface à courbure négative s'expriment en fonction de sommes de Birkhoff renormalisées du flot horocyclique.
}}

\end{center}

\subsection[Lien avec le mélange exponentiel]{Lien avec le mélange exponentiel, travaux de Faure et Faure-Tsuji}

\vspace{2mm}
\paragraph*{\bf Moyennes horocycliques et mélange} On fait maintenant la remarque suivante, qui est maintenant étayée par un certain nombre de résultats rigoureux (\cite{GiuliettiLiverani,AlexanderBaladi,ButterleySimonelli}). 

\vspace{2mm}

 \fbox{
\parbox{\textwidth}{
\textbf{Principe général :}  Soit $T$ un système dynamique hyperbolique (un difféomorphisme ou un flot d'Anosov, ou un difféomorphisme partiellement hyperbolique). Les deux objets suivants sont essentiellement les mêmes. 

\begin{enumerate}
\item Les termes de mélange pour $T$.
\item Les moyennes de Birkhoff pour la dynamique engendrée par le feuilletage instable fort de $T$ (qu'on appelle dynamique horocyclique).
\end{enumerate}
}}

\vspace{2mm}

En particulier, les résultats de ce texte auraient très bien pu être exprimés grâce aux termes de mélange de l'application 

$$ T_N : \mathcal{H}(\mathbb{R})/ \Gamma_N \longrightarrow \mathcal{H}(\mathbb{R})/ \Gamma_N.$$ Une fois de plus, en extrapolant au cas des surfaces à courbure négative on se permet de spéculer. 

\begin{center}

\fbox{
\parbox{\textwidth}{
Les termes d'interférences lors de la propagation pour l'équation de Schrôdinger sur une surface à courbure négative s'expriment en fonction des termes de mélange du flot géodésique.
}}

\end{center}

Cette spéculation pourrait en fait être rendue beaucoup plus précise, mais cela dépasserait le cadre de ce chapitre.

\vspace{2mm}
\paragraph*{\bf Travail de Faure sur le chat quantique}

Le lecteur bien au fait de la littérature sur le sujet aura très probablement pensé à l'article de Frédéric Faure \textit{Prequantum chaos: Resonances of the prequantum cat map} (\cite{Faure}). Il est démontré dans cet article que la propagation de l'application du chat quantique peut s'exprimer en fonction de la dynamique classique d'une certaine application définie sur un espace homéomorphe à  $\mathcal{H}(\mathbb{R})/ \Gamma_N$. 

\vspace{2mm} Quand bien-même le résultat de ce texte est très différent formellement du notre, nous prétendons que le phénomène caché derrière est le même. Il nous semble que le consensus sur cette question est que cette correspondance entre le mélange de la dynamique classique et la propagation quantique est une sorte de curiosité qui ne se produit que dans des cas très algébriques. On espère que les résultats de notre texte convaincrons les lecteurs du contraire : on a mis en évidence le fait que un phénomène expliquant comment calculer les termes d'interférences en fonction de la dynamique classique. Ce phénomène n'est pas spécifique au cas algébrique et devrait en toute vraisemblance se produire dans tous les cas de propagation quantique dont la dynamique sous-jacente est (partiellement) hyperbolique. Notons qu'un équivalent pour les surfaces hyperboliques (liant la propagation de l'équation des ondes au mélange du flot géodésique) est connu.

\vspace{2mm}
\paragraph*{\bf Travaux de Faure-Tsuji sur le mélange exponentiel}

Le mémoire de Faure et Tsuji \textit{"Micro-local analysis of contact Anosov flows and band structure of the Ruelle spectrum"} (\cite{FaureTsuji}) contient la construction d'une quantification d'un flot de contact "émergeant" de la dynamique classique. La manière dont cette quantification est construite consiste à considérer l'action de mélange du flot géodésique sur des fonctions oscillant à haute fréquence, après renormalisation par la \textit{racine} du taux d'expansion de la dynamique pour obtenir un propagateur unitaire.

\vspace{2mm} Dans cet article, on a démontré que les termes d'interférence de la dynamique s'obtiennent en prenant des termes de mélange renormalisés par la \textit{racine} du taux d'expansion de la dynamique. Il semble raisonnable de conjecturer la chose suivante.

\vspace{2mm}

\begin{center}

\fbox{
\parbox{\textwidth}{
La quantification de Faure-Tsuji coïncide avec la quantification classique de la dynamique, tout du moins à un ordre qui permet de faire de la propagation à des temps multiples arbitraires du temps d'Ehrenfest.
}}

\end{center}

\vspace{2mm}
\paragraph*{\bf Rôle du spectre marqué des longueurs}

On termine par une suggestion qu'on espère sera fructueuse. \'A grands coups de spéculation et d'extrapolation, on se sera convaincu de l'énoncé suivant

\begin{center}

\fbox{
\parbox{\textwidth}{
La propagation quantique, pour les quantifiés d'applications du chat ou pour l'équation de Schrödinger sur les surfaces hyperboliques, peut s'exprimer complètement grâce à la dynamique classique, au moins pour des échelles de temps bien supérieures au temps d'Ehrenfest.
}}

\end{center}

Plus précisément, elle s'exprime explicitement comme une fonction des lois limites de la dynamique horocyclique associée (ou de manière équivalente, de la propagation émergeant du mélange une fois les bons termes remis à la bonne échelle). Il nous semble que les résultats existant dans le cas des surfaces ne permettent pas de dire quoi que ce soit d'intéressant sur ce problème. Les résultats de Flaminio-Forni (coté horocyclique) ou à la Faure et Tsuji (coté mélange) garantissent que les choses se passent à la bonne échelle et que nos résultats et spéculations ne contredisent pas ce qu'on connait déjà de la dynamique. Pour conclure sur les questions de chaos quantique, il faudra aller chercher l'ordre suivant. Il semble raisonnable, au vu des balafres de l'application du chat quantique, que les propriétés des ces lois limites ne soient pas universelles, et dépendent fortement de la surface.

\vspace{2mm}

On fait maintenant les deux remarques suivantes. 

\begin{enumerate}
\item Les résultats de distribution microlocale des fonctions propres disponibles à ce jour n'utilisent que la propagation à des temps de l'ordre de deux fois le temps d'Ehrenfest.

\item La dynamique des flots géodésiques/automorphismes partiellement hyperboliques sont entièrement déterminée (grâce au théorème de Livsic) par leur spectre \textit{marqué} des longueurs.
\end{enumerate}

On est donc en droit d'espérer pouvoir décrire des propriétés de localisation/délocalisation des fonctions propres des quantifiés de hamiltoniens hyperboliques en fonction des propriétés arithmétiques du spectre marqué des longueurs. En fait, on conjecture que "l'ordre suivant" dans la compréhension de la théorie ergodique du flot horocyclique dépendra précisément du spectre marqué des longueurs. Ce phénomène est déjà visible dans le cas de l'application du chat quantique, où les résultats de Marklof permettent d'accéder à cet ordre suivant, via les fonctions théta.

\vspace{2mm}

 En particulier, le lien entre action classique le long des orbites périodiques donné par les formules de trace en tout genres devrait se généraliser à des énoncés beaucoup plus précis et plus locaux, et avec une portée temporelle dépassant largement le temps d'Ehrenfest.

\appendix

\chapter{Annexes}

\section{Formules explicites pour le propagateur.}
\label{ann:propagateur}

Dans cette annexe on démontre les formules explicites dont on se sert pour l'expression du propagateur d'un hamiltonien quadratique sur $\mathbb{R}$. Deux textes traitant de manière plus exhaustive ce thèmes sont \cite{Folland} et \cite{Combescure}.

On rappelle les conventions qu'on a choisi d'utiliser dans cet article. Les quantifiés des opérateurs positions et moments sont donnés respectivement par 

\begin{eqnarray*}
\hat{x}^h(f) &:=& x \mapsto xf(x),\\
\hat{\xi}^h(f) &:=& x \mapsto -i \frac{h}{2\pi} \frac{\partial f}{\partial x}(x) = -i \hbar  \frac{\partial f}{\partial x}(x).
\end{eqnarray*}
Notons que les "fonctions propres" de ces deux opérateurs sont respectivement les masses de Diract et les ondes planes $x\mapsto e^{2i\pi \xi x /h}$.
\subsection{Hamiltoniens quadratiques}
\label{sec:hamquad}

Pour le hamiltonien $H = \frac{\alpha}{2} x^2 + \gamma x \xi + \frac{\beta}{2} \xi^2 $, le gradient symplectique s'écrit 

$$ \nabla H(x,\xi) = \begin{pmatrix}
\frac{\partial H}{\partial \xi}(x,\xi) \\
- \frac{\partial H}{\partial x}(x,\xi)
\end{pmatrix} = \begin{pmatrix}
\gamma & \beta \\ 
-\alpha & - \gamma
\end{pmatrix} \begin{pmatrix}
x \\ 
v
\end{pmatrix}.$$ La solution de l'équation différentielle de Hamilton 

$$ \frac{d}{dt} \begin{pmatrix}
x(t) \\ \xi(t)
\end{pmatrix} = \nabla H(x(t),\xi(t)) $$ avec condition initiale $\begin{pmatrix}
x_0 \\ 
\xi_0
\end{pmatrix}$  est donc donnée par 

$$ t \longmapsto \exp(tm)  \begin{pmatrix}
x_0 \\ 
\xi_0
\end{pmatrix} $$ avec $m = \begin{pmatrix}
\gamma & \beta \\ 
-\alpha & - \gamma
\end{pmatrix}$. En notant $\exp(tm) = \begin{pmatrix}
a_t & b_t \\ 
c_t & d_t
\end{pmatrix}$, on obtient que les coefficients de $\exp(tm)$ satisfont au système d'équations différentielles suivant

\begin{eqnarray*}
a_t' &=& \gamma a_t + \beta c_t\ ,\\
b_t' &=& \gamma b_t + \beta d_t\ ,\\
c_t'&=& -\alpha a_t -\gamma c_t\ ,\\
d_t' &=& -\alpha b_t -\gamma d_t\ .
\end{eqnarray*} 
avec conditions initiales $a_0 = d_0 = 1$ et $b_0 = c_0= 0$.

\subsection{Propagateur de l'équation de Schrödinger}

On rappelle que l'équation de Schrödinger est donnée par 

\begin{equation}
i \hbar \frac{\partial}{\partial t}u = \hat{H}^{h}(u) 
\end{equation}  où $\hat{H}^h$ est le quantifié du hamiltonien $H$ (quantifié dans notre cas en utilisant les règles de quantifications de Weyl).

\begin{proposition}
\label{prop:propagfourier}
Avec les notations $\exp(tm) = \begin{pmatrix}
a_t & b_t \\ 
c_t & d_t
\end{pmatrix}$, la fonction 

$$
u(t,x) = \frac{1}{a_t^{1/2}} \exp\big ( \frac{i}{\hbar} \frac{1}{2a_t}(c_t x^2 + 2 x \xi - b_t \xi ^2) \big) 
$$
 est solution de l'équation de Schrödinger pour la condition initiale $u_0(x) = \exp(\frac{2i\pi }{h} x \xi)$.

\end{proposition}

On dédie le reste de ce paragraphe à la démonstration de cette proposition. On procède par analyse, et on suppose que la fonction $u$ de la forme ci dessus est solution de l'équation de Schrödinger, \textbf{sans faire d'hypothèse quelconque sur les fonctions $a_t, b_t$ et $c_t$}. Notre but est de montrer que $u$ est solution de si et seulement si $a_t, b_t$ et $c_t$ sont les coefficients de la matrice $\exp(tm)$.

\paragraph*{\bf} On commence par calculer explicitement l'opérateur  $\hat{H}^h$.

\begin{lemma} 
\label{lemme:hamiltonienquantique}
On a l'expression explicite suivante

 $$ \hat{H}^h(f) (x) =  \frac{\alpha}{2}x^2 f(x) - i \hbar \frac{\gamma}{2} ( 2 x \frac{\partial f}{\partial x}(x) + f(x) ) -  \hbar^2 \beta \frac{\partial^2 f}{\partial x^2}(x) $$ 

\end{lemma}

\begin{proof}
C'est un calcul direct basé sur la règle de quantification suivante pour les hamiltoniens quadratiques : le produit $x\xi$ est quantifié par l'opérateur symétrisé $\frac{1}{2}( \hat{x}^h \circ \hat{\xi}^h + \hat{\xi}^h \circ \hat{x}^h)$.

\end{proof}

On note maintenant $u(t,x) = a_t^{-1/2}\exp(\frac{i}{\hbar} S(t,x))$ avec 

$$ 
S(t,x) = \frac{1}{2a_t}(c_t x^2 + 2 x \xi + b_t \xi ^2).
$$

\begin{proposition}
\label{prop:hamiltonienexplicite}
Pour une fonction $u$ comme ci-dessus, on a 

$$
\hat{H}^h(u(t, \cdot))(x) =  a_t^{-1/2}\Big(\frac{\alpha}{2}x^2 + \gamma x \frac{\partial S}{\partial x} +  \frac{\beta}{2} (\frac{\partial S}{\partial x})^2 - i \hbar\big (\frac{\gamma}{2} + \frac{\beta}{2} \frac{\partial^2 S}{\partial x^2}\big ) \Big) \exp(\frac{i}{\hbar} S) 
$$

\end{proposition}

\begin{proof}

C'est un calcul direct à partir du lemme \ref{lemme:hamiltonienquantique}. 

\end{proof}

On peut ensuite facilement calculer l'expression de $\partial_t u$ en fonction de $S$ et $a_t$. 

$$
 \partial_t u := (-\frac{a_t'}{2a_t^{3/2}}  + \frac{1}{a_t^{1/2}} \frac{i}{\hbar}  \partial_t S) \exp(\frac{i}{\hbar} S).
$$ 
Si $u$ est solution de l'équation de Schrödinger, en utilisant l'expression de $\hat{H}^h(u)$  donnée par la proposition \ref{prop:hamiltonienexplicite} et en identifiant parties imaginaire et réelle on obtient le système d'équations suivant :

\begin{eqnarray}
\dfrac{a_t'}{a_t} &=& \gamma + \beta \dfrac{\partial^2 S}{\partial x^2} \ , \label{eq:eqdiff_a_t}\\
\partial_t S &=& - \dfrac{\alpha}{2}x^2 - \gamma x \ \partial_x S - \dfrac{\beta}{2} \ (\partial_x S)^2 \ . \label{eq:eqdiff_S}
\end{eqnarray}

À partir de maintenant, on montre que ce système d'équations est équivalent au fait que les coefficients $a_t,b_t$ et $c_t$ sont solutions du système d'équation de la fin du paragraphe \ref{sec:hamquad}. On commence par le petit formulaire suivant :

\begin{eqnarray*}
\partial_x S(t,x) &=& \frac{1}{a_t}(c_t x + \xi);\\
\dfrac{\partial^2 S}{\partial x^2}(t,x) &=& \frac{c_t}{a_t};\\
\partial_t S(t,x) &=& -\frac{\partial_t a_t}{2a_t^2}(c_t x^2 + 2 x \xi + b_t\xi^2) + \frac{1}{2a_t}(\partial_t c_t x^2 +  \partial_t b_t\xi^2).
\end{eqnarray*}

L'équation \eqref{eq:eqdiff_a_t} peut donc être réécrite

\begin{equation}
\label{eq:systeme1}
a_t' = \gamma a_t + \beta c_t.
\end{equation} 

L'équation  \eqref{eq:eqdiff_S} est une équation dont le membre de gauche, à $t$ fixé, est un polynôme homogène de degré $2$ en les variables $x, \xi$. Les équations qu'on cherche vous nous être donnée en calculant les coefficients de ce polynôme en $x^2$, $x \xi$ et $\xi^2$ respectivement. Comme ce polynôme doit être nul pour tout $t$, ces coefficients doivent être identiquement nuls. Le coefficient en $x^2$ est donné par 

$$
-\frac{a_t'}{2a_t^2}c_t  + \frac{c_t'}{2a_t} + \frac{\alpha}{2} + \gamma \frac{c_t}{a_t} + \frac{\beta}{2} \frac{c_t^2}{a_t^2}.
$$
 En multipliant cette équation par $a_t^2$ et en remplaçant $a_t'$ par $\gamma a_t + \beta c_t$ on obtient 

\begin{equation}
\label{eq:systeme2}
c_t' =  - \alpha a_t - \gamma c_t.
\end{equation}

On continue avec le coefficient en $x\xi$. Celui-ci est donné par 

$$
 - \frac{a_t'}{a_t^2} +  \gamma \frac{1}{a_t} + \frac{\beta}{2}\frac{2c_t}{a_t^2}.
$$
En multipliant cette équation par $a_t^2$ on retrouve la première équation \ref{eq:systeme1}.

\vspace{2mm}

On termine avec le coefficient en $\xi^2$. Celui-ci est égal à 

$$
\frac{a_t'}{2a_t^2} b_t - \frac{b_t'}{2a_t} + \frac{\beta}{2}\frac{1}{a_t^2}.
$$ 
En multipliant par $a_t^2$ et en remplaçant $a_t'$ par $\gamma a_t + \beta c_t$ on obtient 

$$
(\gamma a_t + \beta c_t)b_t - a_t \partial_t b_t + \beta.
$$ 
Si l'on suppose maintenant que les coefficients $a_t, b_t, c_t, d_t$ satisfont aux équations de Hamilton comme à la fin du paragraphe \ref{sec:hamquad}, on écrit $1 = a_t d_t - c_t b_t$ car la matrice $\exp(tm)$ est de déterminant $1$. On obtient dans ce cas que le coefficient en $\xi^2$ est égal à

$$
 -b_t' + \gamma b_t + \beta d_t.
$$ 
En particulier, comme les coefficients $a_t, b_t, c_t, d_t$ satisfont les équations de Hamilton le coefficient s'annule. On a terminé la démonstration de la proposition \ref{prop:propagfourier}. De cette proposition on peut déduire une expression explicite pour le propagateur de l'équation de Schrödinger. 

\begin{proposition}
Soit $f \in \mathrm{L}^2(\mathbb{R})$ et soit $\hat{H}^h$ comme ci-dessus. On a pour tout temps $t \in \mathbb{R}$ 

$$
 \big (e^{-\frac{i}{\hbar}t \hat{H}^h}f\big ) (x) = \frac{1}{a_t^{\frac{1}{2}}}\int_{\mathrm{R}}{\hat{f}^h(\xi)e^{\frac{2i\pi}{h}(\frac{1}{2a_t}(c_t x^2 + 2 x \xi - b_t \xi ^2))} d\xi} .
 $$

\end{proposition}

\begin{proof}

On a vu que $(t,x) \mapsto \frac{1}{a_t^{1/2}}\exp(\frac{2i\pi}{h}(\frac{1}{2a_t}(c_t x^2 + 2 x \xi - b_t \xi ^2)$ est la solution de l'équation de Schrödinger avec condition initiale $x \mapsto \exp(\frac{2i\pi \xi \cdot x}{h})$. La solution générale est ensuite obtenue par tranformée de Fourier et la linéarité de l'équation de Schrödinger.

\end{proof}

\section{Lemmes techniques.}
\label{ann:demos}

On rappelle que $\mathcal{S}(\R)$ est l'espace des distributions standard, dual de $\mathcal{S}'(\mathbb{R})$ l'espace des fonctions de classe $\mathcal{C}^{\infty}$ à décroissance exponentielle.

\begin{proposition}
\label{prop:sym}

Soit $f \in \mathcal{S}(\R)$. Le symétrisé 

$$ \Sigma^h(f) := \varphi \longmapsto \sum_{(k_1,k_2) \in \mathbb{Z}^2}{\langle \widehat{T^h_{(k_1,k_2)}}(f), \varphi \rangle} $$ défini une distribution dans $\mathcal{S}'(\R)$. Cette distribution est invariante par $\widehat{T^h_{(1,0)}}$ et $\widehat{T^h_{(0,1)}}$ quand $h = \frac{1}{N}$.

\end{proposition}

\begin{proof}
On doit vérifier que la somme converge. Cela découle du théorème de convergence dominée: on a
%\begin{multline*}
$$ \sum_{(k_1,k_2) \in \mathbb{Z}^2}{\langle \widehat{T^h_{(k_1,k_2)}}(f), \varphi \rangle} = \sum_{k_1,k_2} \int_\R e^{-i\pi k_1k_2/h}e^{2i\pi k_2x/h}f(x-k_1) \varphi(x) dx$$
% \end{multline*}
or
\begin{itemize}
\item
La fonction $f$ est Schwartz, donc la somme sur $k_1$ converge (absolument).
\item
pour la somme sur $k_2$, on étudie 
$$\sum_{k_2} e^{-i\pi k_1k_2/h}e^{2i\pi k_2x/h} = \sum_{k_2}\left(e^{i\pi k_1(2x-k_1)/h}\right)^{k_2}$$
qui est une série géométrique. En particulier elle converge, et les sommes partielles $\sum_{k_1=-N_1}^{N_1}\sum_{k_2=-N_2}^{N_2} e^{-i\pi k_1k_2/h}e^{2i\pi k_2x/h}f(x-k_1) \varphi(x)$  sont bornées, en module, par la fonction intégrable $C \, f(x-k_1) \varphi(x) $ pour une constante $C$ assez grande. On peut donc permuter la somme double et l'intégrale, d'où la conclusion.

Le fait que la distribution ainsi définie soit invariante sous l'action du tore est clair.
\end{itemize}
%\textcolor{red}{A écrire}

\end{proof}

On donne aussi la démonstration de la proposition \ref{prop:symsurjective}, dont on rappelle ici l'énoncé.

\begin{proposition}
\label{prop:surjective}
L'application linéaire "moyenne" définie précédemment 
$$\mathcal S(\R) \to \mathcal H^N$$
$$\varphi \mapsto \Sigma^h(\varphi)$$
est surjective.
\end{proposition}
\begin{proof}

On rappelle qu'un élément $D$ de $\mathcal{H}^N$ induit une distribution continue 

$$ D : \mathcal{C}^{\infty}(S^1, \mathbb{C}) \longrightarrow \mathcal{C}^{\infty}(S^1, \mathbb{C})$$ donc les coefficients de Fourier sont $N$-périodiques. Notons $(d_n)_{n \in \mathbb{Z}}$ la suite des coefficients de Fourier de $D$. Soit $\psi := x \mapsto \sum_{k=0}^{N-1}{d_i e^{2i\pi k x}}$ qui est un élément de $\mathcal{C}^{\infty}(S^1, \mathbb{C})$. 
%La distribution $D$ est alors le $h$-symétrisé de $\psi$. 
Pour démontrer la proposition, il suffit de montrer qu'il existe $\varphi \in S(\R)$ telle que le $\tilde{\psi}$ relevé de $\psi$ à $\mathbb{R}$ est tel que 

$$ \tilde{\psi}(x) = \sum_{k \in \mathbb{Z}}{\varphi(x + k)}.$$ 
Ce résultat est donné par le lemme  \ref{lemme:periodise} ci-dessous. Une fois qu'on a une telle fonction $\varphi$ on vérifie sans peine que 

$$  D = \Sigma^h(\varphi).$$
\end{proof}

\begin{lemma}
\label{lemme:periodise}

Soit $\psi : S^1 \longrightarrow \mathbb{C}$ une fonction $\mathcal{C}^{\infty}$. Il existe alors une fonction $\varphi : \mathbb{R} \longrightarrow \mathbb{C}$ de classe $\mathcal{C}^{\infty}$ à support compact telle que le relevé $\tilde{\psi} : \mathbb{R} \longrightarrow \mathbb{C}$ satisfait pour tout $x \in \mathbb{R}$

$$ \tilde{\psi}(x) = \sum_{k \in \mathbb{Z}}{\varphi(x + k)}.$$

\end{lemma}

\begin{proof}

Il suffit de démontrer le lemme pour la fonction constante égale à $1$. En effet si on dispose de $\varphi_0 : \mathbb{R} \longrightarrow \mathbb{C}$ lisse à support compact telle que $ \sum_{k \in \mathbb{Z}}{\varphi_0(x + k)}=1$, c'est-à-dire que les translatés de $\varphi_0$ forment une partition de l'unité sur $\mathbb{T}^2$, la fonction $\varphi_0 \cdot \tilde{\psi}$ fait l'affaire. 

\vspace{2mm}

Pour construire une telle fonction on part de $u : \mathbb{R} \longrightarrow \mathbb{R}$ qui a les propriétés suivantes :

\begin{enumerate}
\item $u$ est positive ou nulle. 

\item $u$ est de classe $\mathcal{C}^{\infty}$.

\item $u$ est à support compact.

\item $u$ ne s'annule pas sur l'intervalle $[-1,1]$. 
\end{enumerate}

On définit $v(x)= \sum_{k \in \mathbb{Z}}{u(x+k)}$. Cette somme fait intervenir un nombre borné de termes non nuls. De par les propriétés de $u$, $v$ est $1$-périodique, positive et ne s'annule jamais. On peut donc poser $\varphi_0 := \frac{u}{v}$; $\varphi_0$ a bien la propriété qu'on voulait car 

$$ \sum_{k \in \mathbb{Z}}{\varphi_0(x +k)} = \sum_{k \in \mathbb{Z}}{\frac{u(x+k)}{v(x+k)}} =  \frac{1}{v(x)} \sum_{k \in \mathbb{Z}}{u(x+k)} = \frac{v(x)}{v(x)} = 1.$$
\end{proof}

\begin{lemma}
\label{lemme:ps}
Il existe une constante $C$ telle que pour tout $(q_0,p_0) \in \mathbb T^2$, pour tout $m\in \ZZ, n\ge 0, h >0$, on ait
$$\left\lvert \Big\langle\big(\widehat{M^n}(f_h) - \mathcal L^h(n)\big) \cdot \Phi^h_{(q_0+m, p_0+p(m))}\Big\rangle\right\rvert \le C \left\lvert \big(\widehat{M^n}(f_h) - \mathcal L^h(n)\big)(q_0+m)\right\rvert h^{1/4}
$$
\end{lemma}

\bibliographystyle{alpha} 
\bibliography{biblio}

\end{document}